\documentclass[a4paper,12pt]{article}
\usepackage{amsmath}
\usepackage{amsthm}
\usepackage{amssymb}
\usepackage{amscd}
\usepackage[dvips]{graphicx}
\newtheorem{definition}{Definition}[section]
\newtheorem{theorem}[definition]{Theorem}
\newtheorem{proposition}[definition]{Proposition}
\newtheorem{lemma}[definition]{Lemma}
\newtheorem{observation}[definition]{Observation}
\newtheorem{expectation}[definition]{Expectation}
\newcommand{\bfm}{\mathbf{m}}
\newtheorem{cor}[definition]{Corollary}
\newtheorem{example}[definition]{Example}
\newtheorem{remark}[definition]{Remark}
\newtheorem{conjecture}[definition]{Conjecture}

\newtheorem{question}[definition]{Question}
\newtheorem{questions}[definition]{Questions}

\newtheorem{sublemma}[definition]{Sub-lemma}
\newcommand{\s}{^{\sharp}}
\newcommand{\n}{^{\natural}}
\newcommand{\sigi}{\sigma ^{-1}}

\newcommand{\cH}{\mathcal{H}}
\newcommand{\gH}{\mathfrak{H}}
\newcommand{\gI}{\mathfrak{I}}

\newcommand{\infgal}{\mathrm{Inf}\text{-}\mathrm{gal}\, }
\newcommand{\cinfgal}{\mathrm{CInf}\text{-}\mathrm{gal}\, }
\newcommand{\ncinfgal}{\mathrm{NCinf}\text{-}\mathrm{gal}\, }

\newcommand{\QSI}{$q$-SI $\sigma$-differential }

\newcommand{\com}{\mathbb {C}} 

\newcommand{\cZ}{\mathcal{Z}}
\newcommand{\gd}{\delta}
\newcommand{\D}{\mathcal{D}}

\newcommand{\N}{\mathbb{N}}
\newcommand{\G}{\mathbb{G}}

\newcommand{\calH}{\mathcal{H} \,}
\newcommand{\Hom}{\mbox{$\mathrm{Hom}$}}
\newcommand{\Id}{\mathrm{Id}}

\newcommand{\GL}{\mathrm GL}
\newcommand{\M}{\mathbf{M}} 
\newcommand{\Q}{\mathbb{Q}}
\newcommand{\by}{\mathbf{y}}

\newcommand{\Z}{\mathbb{Z}}
\newcommand{\K}{\mathcal{K}}
\newcommand{\eL}{\mathcal{L}}
\newcommand{\R}{\mathcal{R}}

\newcommand{\itqsi}{{\it qsi }}
\newcommand{\mapue}[1]{%
  \smash{\mathop{%
  \hbox to 1cm{\rightarrowfill}}\limits^{#1}}}
\newcommand{\NCF}{\mathcal{NCF}\sb{L/k}}
\newcommand{\NCA}{\text{$(NCAlg/L^{\natural})$}}
\begin{document}
\title{Toward quantization of Galois theory}
\author{Akira Masuoka\footnote{Institute of Mathematics, University of Tsukuba. E-mail: akira@math.tsukuba.ac.jp},\ Katsunori Saito\footnote{Graduate School of Mathematics, Nagoya University. E-mail: m07026e@math.nagoya-u.ac.jp}\ and\ 
Hiroshi Umemura\footnote{Graduate School of Mathematics, Nagoya University. E-mail: umemura@math.nagoya-u.ac.jp}}
%\date{}
\maketitle
\begin{abstract}
This article was born from our experiments, the first explorations of 
an unknown land of quantized Galois theory. 
We know Hopf-Galois theory for linear equations or Picard-Vessiot theory in terms of Hopf algebra 
\cite{amaetal09} 
that is a general Galois theory of linear equations with a set of non-commutative operators. 
The Hopf algebras in this theory are, however, essentially assumed to be {\it \/co-commutative. \/} 
In other words, they are interested in only commutative rings with operators. Consequently their Galois groups are linear algebraic groups and the Galois theory is not quantized. 
\par 
Heiderich \cite{hei10} discovered that we can combine the Hopf Galois theory for linear equations and our general Galois theory of non-linear equations. 
We apply this theory to some concrete examples and show that 
the quantization of 
Galois group happens in the Part I. 
\par 
In fact, quantization occurs even for linear equations. 
In the Part II, we analyze, one particular example 
of linear difference-differential equation
to show the unique existence of the 
non-commutative 
 Picard-Vessiot ring and asymmetric Tannaka theory. 
\par 
Studying our examples, 
we succeeded in generalizing 
the examples to any Hopf linear equations over a constant field.
So for any $C$-Hopf algebra $H$ and for any left $H$ module $M$ 
that is finite-dimensional over the base field $C$, we have a quantized Hopf Galois theory. 
See Part III.

%Dealing with linear equations in Hopf Galois theory, so far we did not encounter quantum group as Galois 
%group except for linear algebraic group. This made us 
%believe that it would be impossible to get a quantum group as the Galois group in linear Hopf Galois theory. 
%\par 
%So we are led to non-linear Hopf Galois theory to give examples %of non-linear equations 
%in which Galois theory is quantized.
%In th%e first part, we give three examples of
%non-linear difference-differential equations 
 %in which quantum groups 
 %that are neither commutative nor co-commutative 
 % naturally arise
 %as Galois groups.
%Moreover, one of the three examples is reduced to a linear %equation revealing our belief that considering linear %equations, we can not 
% quantize Galois theory is false (See Section \ref{140116a}). %It suggests that there would be quantum Picard-Vessiot theory. %In part II, we analyze this Example and we give further %Examples. 
 %Looking at these examples, Akira Masuoka 
 %\cite{mas14}
% proposed a new Picard-Vessiot theory for not necessary co-%commutative Hopf algebras.
 \end{abstract}
 \part{Quantization of non-linear \QSI equations}
\section{Introduction} 
%%%%%%%%%%%%%%%%%%%%%%%%%%%%%%%%%%%%%%%%%%%%%%%%%%%%%%%%%%%%%%%%
The pursuit of $q$-analogue of hypergeometric series goes back to 
Heine \cite{hein46}, 1846. However, 
Galois group of a $q$-hypergeometric series is not a quantum 
group but it is a linear algebraic group. 
This shows that so far as we consider 
the $q$-analogue of the hypergeometric equation according to Heine, Galois theory is not quantized. 
In fact, 
generally we know that the Galois group of a linear 
difference equation is a linear algebraic group. 
So we may as well consider that Heine's $q$-hypergeometric series would be unsatisfactory as a quantization. 
To be more precise, the following question comes into our mind.
\begin{question}\label{q1}
Does there exist a thorough quantization of 
hypergeometric series in such a way that 
Galois groups of quantized series 
are general quantum groups? \par
More generally, we wonder if there would be a quantized Galois theory. 
\end{question}
We owe Question \ref{q1} to 
Y.~Andr\'e \cite{and01} who first studied linear difference-differential equations in the 
framework of non-commutative geometry. He encountered only linear algebraic groups. Later, 
Hardouin \cite{har10} also studied Picard-Vessiot theory of $q$-skew iterative $\sigma$-differential field extensions but also in this theory, the Galois group is 
a linear algebraic group. 
%%%%%%
We clarified the situation in \cite{ume11}. 
Namely, so far as we study 
linear difference-differential equations, however twisted or non-commutative the ring of difference and differential operators might be, Galois group, according to general Hopf Galois theory, is a linear algebraic group.
\par 
All the attempts of answering affirmatively Question \ref{q1} had been so far failed. 
%%%%%%%%%%%%%
 Our results settle the Question \ref{q1} for linear equations with constant coefficients. 
 See Part III.
%%%%%%%% 
%At the end of the Part I, namely in subsections \ref{151007a}, %\ref{151007b}, \ref{151007c}, 
%we also give examples of non-linear equations of which Galois group is a quantum group that is not a linear algebraic group. %The first one can be interpreted as a linear case so that it is very much inspiring for general theory of linear \QSI equations. 
 
%%%%
Question \ref{q1} is vague and we have to start by clarifying the nature of Question \ref{q1}.
\subsection{What is $q$?}
In alphabet we have 26 letters. That is surly a small set if we compare it with the huge set of Chinese characters. 
It would be certainly by chance that the letter $q$ appears often in mathematics in different contexts. 
\begin{enumerate} 
\renewcommand{\labelenumi}{(\arabic{enumi})}
\item $q$ of $q$-analogue studied by L.~Euler \cite{eul48} 
and E.~Heine \cite{hein46} and many mathematicians after the last century. 
\item $q =p^n$. The number of elements of the finite field $\mathbb{F}\sb {q} = \mathbb{F}\sb {p^n}$. The letter $q$ comes after $p$ of the prime number.
\item $q= e^{2\pi \tau}$, where $\tau \in H
= \{ \tau \in \com \, | \, \Im (\tau) > 0 \}$. Elliptic modular.
\item $q$ of quantum physics. Non-commutative geometry. 
\end{enumerate}
These subjects were not a priori logically supposed to be related. However, there are surprising mathematical relations. For example, Euler \cite{eul48} proved the pentagonal identity in 1848 which belongs to (1) and Jacobi \cite{jac29} showed Euler's identity is a consequence of the triple product formula, 
revealing an unexpected relation between (1) and (3). So it is important to discover 
surprising relations among the subjects. 
Question \ref{q1} asks 
if there exist 
a Galois theoretic relations between (1) and (4).

%%%%%%%%
 % the first Example in Section \ref{10.4a} redices to a pair of linear difference-differential equations. 
 %%%%%%%%%%%%%%%%%%%%
 \subsection{Linear vs. non-linear equations}
 We believed for a long time that it was impossible to 
quantize Picard-Vessiot theory, 
Galois theory for 
linear difference or differential equations. 
Namely, there was no Galois theory for linear difference-differential equations, of which the Galois group is a quantum group that is, in general, 
neither commutative nor co-commutative. 
This is not correct as we are going to see examples in the Part II and general theory of the third Part of this note. 
Our mistake came from a misunderstanding of preceding 
works of Hardouin \cite{har10} and of Masuoka and Yanagawa \cite{masy}. 
\par 
{ \it Now it has been clear that 
the correct understanding of the picture is that despite they considered a set of non-commutative operators, 
as they assumed that the rings of functions on which the set of 
 non-commutative operators act were commutative, they did not arrive at a quantization of Galois theory.} 
 In fact, in their Picard-Vessiot theory, a Picard-Vessiot extension is a difference-differential field extension. 
 \par 
 With this misbelief, 
 it was natural to wonder how about considering non-linear difference-differential equations. We proposed to study the $q$-Painlev\'e equations in \cite{ume11}. %
 We elaborated and we can answer this question in the following way. 
 As we observe in the Part I, 
 quantization of Galois group happens for much simpler equations than the 
 $q$-Painlev\'e equations (Sections 
 \ref{10.4a}, \ref{10.4b} and \ref{10.4c}). 
% In the second part, Section \ref{140116a}, 
Moreover the First Example reduces to a pair of linear difference-differential equations breaking our wrong belief. 
%we show that 
%%%%%%%%%
In the Part I, 
after a brief review of our framework, 
we analyze three examples of difference-differential field extensions. In these examples, however, the Galois hulls or the 
normalizations are not commutative rings yielding 
quantum Galois group that are neither commutative nor co-commutative Hopf algebras. 
\par 
Among these three examples the first one is given by a pair of linear difference-differential equations. In the Part II, we analyze this example thoroughly. We show that the Picard-Vessiot ring exists uniquely and the asymmetric Tannaka theory holds for this particular example. 
Looking at this and further examples found in Section \ref{141223a}, we have established a general quantum Picard-Vessiot theory over a constant field in Part III.  
 %that we did not succeed in quantizing Galois theory 
%See Introduction to the second Part. 
\par
While in Parts I and II, 
we work exclusively over a constant field $C$ of characteristic $0$, 
in the third Part, the constant base field $C$ is of characteristic $p \ge 0$.
We consider $C$-algebras. 
Except for Lie algebras, all the rings or algebras are associative $C$-algebras and 
contain the unit element. 
So the field $C$ is in the center of the algebras. 
 Morphisms between 
them are unitary $C$-morphisms. For a commutative algebra
 $A$, we denote by $(Alg/ A)$ the category of commutative $A$-algebras, which we sometimes denote by $(CAlg/ A)$ 
to emphasize that we are dealing with commutative $A$-algebras. In fact, to study quantum groups, we have to also consider non-commutative $A$-algebras. We denote by $(NCAlg/A)$ the category of not necessarily commutative $A$-algebras $B$ such that 
$A$ ( or to be more logic, the image of $A$ in
 $B$) is contained in the center of $B$. 
 \par 
 %We thank Professors Akira Masuoka and Katsutoshi Amano for 
%teaching us their Galois theory and for 
 %valuable 
%discussions.

%%%%%%%%%%%%%%%%%%%%%%%%%%%%%%%%%%%%%%%%%%%%%%%%%%%%%%%%%%%%%%%%%%%%%%%% 
\section{Foundation of a general Galois theory \cite{ume96.2}, \cite{ume06}, \cite{ume07}}
\subsection{Notation}
\par 
Let us recall basic notation. 
Let $(R, \gd )$ be a differential ring so that $\gd:R \to R$ is a derivation of a
commutative ring $R$ of characteristic $0$. When there is no danger of confusion of the derivation 
 $\gd$, we simply say the differential ring $R$ without referring to the derivation 
 $\gd$. We often have to talk, however, about the abstract ring $R$ that we denote by $R\n$. 
 For a commutative ring $S$ of characteristic $0$, the power series ring 
 $S[[X]]$ with derivation $d/dX$ gives us an example of differential ring. 
%%%%%%%%%%%%%%%%%%%%%%%%% 
\subsection{General Galois theory of differential field extensions}
 Let us start by recalling our general Galois theory of differential field extensions.
 \subsubsection{Universal Taylor morphism}
 \par 
 Let $(R, \, \gd )$ be a differential algebra
 so that $R$ is a commutative $C$-algebra and 
 $\gd : R \to R$ is a $C$-derivation:
 \begin{enumerate}
 \renewcommand{\labelenumi}{(\arabic{enumi})}

 \item 
 $\gd : R \to R$ is a $C$-linear map.
 \item 
 $\gd (ab) = \gd (a) b + a\gd (b) $ for all $a, \, b \in R$.
 \end{enumerate}
 For the differential algebra $( R, \, \gd )$ and a commutative $C$-algebra $S$, 
 a Taylor morphism is a differential morphism 
 \begin{equation}\label{a2.1} 
 ( R, \,\gd ) \to (S[[X]], \, d/dX ).
 \end{equation}
 Given a differential ring $(R, \, \gd )$, 
 among the Taylor morphisms \eqref{a2.1}, there exists the universal one. 
 In fact, for an element $a\in R$, we define the power series 
\begin{align*}
 \iota (a) = \sum \sb{n=0} ^{\infty} \frac{1}{n!}\gd^{n}(a)X^n \in R\n [[X]].
\end{align*}
Then the map
\begin{equation}\label{9.19d}
\iota \colon (R, \, \gd ) \to (R\n [[X]],\, d/dX )
\end{equation}
is the universal Taylor morphism.
%%%%%%%%%%%%%%%% 
\subsubsection{Galois hull $\eL /\K$ for a differential field extension $L/k$}\label{9.19h}
Let $(L, \, \gd)/(k, \, \gd )$ be a differential field extension such that 
the abstract field $L\n $ is finitely
 generated over the abstract base field $k\n$. We constructed the Galois hull 
 $\eL / \K$ in the following manner. 

 We take a mutually commutative basis 
\begin{align*}
 \{ D \sb 1, \, D \sb 2 , \cdots 
 , D\sb d \}
\end{align*}
 of the $L\n$-vector space $\mathrm{Der} \, (L\n / k\n )$ 
 of $k\n$-derivations of the abstract field $L\n$. So we have 
\begin{align*}
[ D\sb i , D\sb j ] =D\sb i D\sb j - D\sb j D\sb i= 0 \qquad \text{\it \/for } 1 \le i, \, j \le d. 
\end{align*}
\par 
%%%%%%%%%%%%%%%%%%%%%%%%%%%%%%%%%%%%%%%%%%%%%%%%%%%%%%%%%%%%%%%%%%%%%%%%
Now we introduce a partial 
differential structure on the abstract field $L\n$ 
using the derivations $\{D\sb 1 , \, D\sb 2, \,
 \cdots , D\sb d \}$. 
 Namely we set 
$$
L\s := ( L\n , \,\{ D\sb 1 , \, D\sb 2, \, \cdots , D\sb d \} )
$$ 
that is a partial differential field. 
Similarly we define a differential structure on the power series ring 
$L\n [[X]]$ with coefficients in $L\n$ by considering  
the derivations $$\{D\sb 1 , \, D\sb 2, \,
 \cdots , D\sb d \}$$ 
 that operate on the coefficients of the power series. 
 In other words, 
 we work with the differential ring 
 $L\s [[X]] $. 
 So the power series ring 
 $L\s[[X]]$ has differential structure 
 defined by the differentiation $d/dX$ with respect to the variable $X$ 
 and the set 
 $$\{ D\sb 1 , \, D\sb 2, \, \cdots , D\sb d \}$$ 
 of derivations. Since there is no danger of confusion of the choice of the differential operator $d/dX$, 
 we denote this differential ring by 
 $$
 L\s [[X]]. 
 $$
 %%%%%%%%%%%%%%%%%%%%%%%%%%%%%%%%%%%%%%%%%%%%%%%%
We have the universal Taylor morphism 
\begin{equation}\label{m27.0.5}%m27.0.5
\iota \colon L \to L\n [[X]]
\end{equation}
that is a differential morphism. We added further the
$\{ D\sb 1 , \, D\sb 2, \, \cdots , D\sb d \}$-differential structure 
on $L\n [[X]] $ or we replace the target space 
$ L\n [[X]] $ of the universal Taylor morphism \eqref{m27.0.5}
 by $L\s [[X]] $ so that we have 
 $$
 \iota \colon L \to L\s [[X]] . 
 $$
\par 
In Definition \ref{a4.6} below, we work in the differential ring 
$L\s[[X]] $ with differential operators 
$d/dX $ and 
$
\{ D\sb 1 , \, D\sb 2, \, \cdots , D\sb d \}.
$
We identify the differential field $L\s$ 
with 
the set of power series
consisting only of constant terms. 
Namely, 
$$
L \s = \{ \sum \sb {n=0}^{\infty} a\sb n X^n \in L\s [[X]] \, | \, \,\text{\it \/The 
coefficients }
a\sb n =0 \text{\it \/ for every } n\ge 1 \}.
$$
Therefore $L\s$ is a differential sub-field of the differential 
ring $L\s [[X]] $. The differential operator $d/dX$ 
kills $L\s$. 
 Similarly, 
we set 
$$
k \s := \{
 \sum \sb {n=0}^{\infty} a\sb n X^n \in L\s [[X]] \, | \, \,\text{\it \/The 
coefficients $a\sb 0 \in k$ and }
a\sb n =0 \,\text{\it \/ for every } n\ge 1 \}.
$$
So all the differential operators 
$d/dX, \, D\sb 1, \, D\sb 2, \cdots , D\sb d$ 
 act trivially on 
$k\s$ and so $k\s$ is a differential sub-field of $L\s$ and hence 
of the differential algebra $L\s [[X]] $. 
\begin{definition}\label{a4.6}
The Galois hull $\eL / \K $ is the differential sub-algebra 
of $L\s[[X]] $, where 
$\eL$ is the differential sub-algebra generated by the 
image $\iota (L)$ and $L\s$ 
and $\K $ is the sub-algebra generated by the image $\iota (k ) $ 
and $ L\s$. So $\eL / \K$ is a differential algebra extension with differential operators $d/dX $ and 
$
\{ D\sb 1 , \, D\sb 2, \, \cdots , D\sb d \}.
$
\end{definition} 
\subsubsection{Universal Taylor morphism for a partial differential ring}
%We explain how we define our Galois group $\infgal (L/k)$. 
\par 
The universal Taylor morphism has a generalization for 
partial differential ring. 
Let $$(R, \,\{ \partial \sb 1 ,\, \partial \sb 2 ,\, \cdots , \partial \sb d \}) $$ be a partial differential ring. So $R$ is a commutative ring of characteristic $0$ and 
$\partial \sb i \colon R \to R$ are mutually commutative derivations. 
For a ring $S$, the power series ring 
$$
( S[[ X\sb 1,\, X\sb 2,\, \cdots , X\sb d ]], \,\{ 
\frac{\partial }{\partial X\sb 1 }, \,
\frac{\partial }{\partial X\sb 2 }, \, \cdots , 
 \frac{\partial }{\partial X\sb d } \}) 
$$
gives us an example of 
partial differential ring. 
\par   
A Taylor morphism is a differential morphism 
\begin{equation}\label{taylor}%taylor
(R,\, \{ \partial \sb 1 ,\, \partial \sb 2 ,\, \cdots ,\, \partial \sb d \}) 
\to 
( S[[ X\sb 1,\, X\sb 2,\, \cdots ,\, X\sb d ]],\, \{ 
\frac{\partial }{\partial X\sb 1 }, \,
\frac{\partial }{\partial X\sb 2 }, \, \cdots , 
 \frac{\partial }{\partial X\sb d } \}). 
\end{equation}
For a differential algebra 
$
(R,\, \{ \partial \sb 1 ,\, \partial \sb 2 ,\, \cdots ,\, \partial \sb d \}) , 
$
among Taylor morphisms \eqref{taylor}, 
 there exists the universal one $\iota \sb R$ given below. 
 \begin{definition}
 The universal Taylor morphism is a differential morphism 
 \begin{equation}\label{m28.1}
 \iota \sb R \colon (R,\,\, \{ \partial \sb 1 ,\, \partial \sb 2 ,\, \cdots ,\, \partial \sb d \}) 
 \to 
 ( R\n[[ X\sb 1,\, X\sb 2,\, \cdots ,\, X\sb d ]],\, \{ 
\frac{\partial }{\partial X\sb 1 }, \,
\frac{\partial }{\partial X\sb 2 }, \, \cdots , 
 \frac{\partial }{\partial X\sb d } \})
\end{equation}
%where the morphism $\iota \sb R$ is 
defined by the formal power series expansion 
%%%%%%%%%%%%%%%%%%%%%%%%%%%%%%%%%%%%%%%%%%%%%%%%%%%%%%%%%%%
\begin{equation*}
\iota\sb {R} (a) = \sum\sb {n \in \N ^d} \frac{1}{n!}
\partial ^n (a) \, X^n 
\end{equation*}
for an element $a \in R$, 
where we use the standard notation for multi-index. 
\par Namely, 
for $n= (n\sb 1,\, n\sb 2, \cdots , n\sb d) \in \N ^d$, 
$$
 |n| = \sum \sb {i=1}^d n\sb i, 
 $$
 $$
\partial ^n = \partial \sb 1 ^ {n\sb 1}
\partial \sb 2 ^ {n\sb 2} \cdots 
\partial \sb d ^ {n\sb d}
$$
$$
n! = n\sb 1! n\sb 2! \cdots n\sb d! 
$$
and 
$$
X^n = X\sb 1 ^{n\sb 1}X\sb 2 ^{n\sb 2}\cdots X\sb d ^{n\sb d}.
$$ 
 \end{definition}
 See Proposition (1.4) in Umemura \cite{ume96.2}.
 %%%%%%%%%%%%%%%%%%%%%%%%%%%%%%%%%%%%%%%%%%%%%%%%%%%%%%%
 \subsubsection{The functor 
 $\mathcal{F} \sb{ L/k}$ of infinitesimal deformations for a differential field extension}\label{9.25e}
%Combining the universal Euler morphism 
%and 
%%%%%%%%%%%%%%%%%%%%%%%%%%%%%%%%%%%%%%%%%%%%%%%%%%%
For the partial differential 
field $L\s$, 
we have the universal Taylor morphism
\begin{equation}\label{m27.1}%m27.1
\iota \sb {L\s} \colon L\s \to L\n [[ W\sb 1, W\sb 2, \cdots , W\sb d ]]
=L\n [[W]], 
\end{equation}
 where we replaced the variables $X$'s 
 in \eqref{m28.1}
 by the variables $W$'s for a notational reason. 
 The universal Taylor morphism \eqref{m27.1} gives a 
 differential morphism 
 \begin{equation}\label{m27.2}
 L\s[[X]] \to L\n [[ W\sb 1,\, W\sb 2, \cdots , W\sb d ]][[X]]. 
 \end{equation}
Restricting the morphism \eqref{m27.2} to the differential 
sub-algebra $\eL $ of $L\s[[X]] $, we get a differential 
morphism 
$\eL \to L\n [[ W\sb 1,\, W\sb 2,\, \cdots , \, W\sb d ]][[X]] $
that we denote by $\iota $. So we have the 
differential 
morphism 
\begin{equation}\label{m27.4}
\iota \colon \eL \to 
L\n [[ W\sb 1,\, W\sb 2,\, \cdots ,\, W\sb d ]][[X]].
\end{equation}
Similarly for every commutative $L\n$-algebra $A$, 
thanks to the differential morphism 
$$
L\n[[W]] \to A[[W]]
$$
arising from the structural morphism $L\n \to A$, 
we have the 
canonical differential morphism 
\begin{equation}\label{m27.3}
\iota \colon \eL \to A [[ W\sb 1,\, W\sb 2,\, \cdots ,\, W\sb d ]] [[X]].
\end{equation}
We define the functor 
$$
\mathcal{F}\sb {L/k} \colon (Alg/ L\n) \to (Set)
$$
from the category $(Alg/L\n)$ of commutative $L\n$-algebras to the category 
$(Set)$ of sets, by associating to an $L\n$-algebra $A$, the set of 
 infinitesimal deformations of the canonical morphism \eqref{m27.4}. 
So 
\begin{multline*}
\mathcal{F}\sb{L/k}(A)
= \{ f\colon \eL \to 
 A [[ W\sb 1,\, W\sb 2,\, \cdots ,\, W\sb d ]][[X]] \, | \, 
 f \,\text{\it \/ is a differential }
 \\
 \,\text{\it \/morphism 
 congruent to the canonical morphism } \iota 
 \text{\it \/ modulo nilpotent elements} \\ 
 \text{\it \/such that } 
 f = \iota \,\text{\it \/ when restricted on the sub-algebra } \K 
\}.
 \end{multline*}
% See Definition 2.13 in \cite{mo1}, 
% for a rigorous definition. 
 \subsubsection{Group functor $\infgal(L/k)$ of infinitesimal automorphisms for a differential field extension} 
The Galois group in our Galois theory is the group functor 
$$
\infgal (L/k) \colon (Alg/L\n ) \to (Grp)
$$
 defined 
 by 
 \begin{multline*}
 \infgal (L/k) (A) =
 \{
 \, 
 f \colon \eL \hat{\otimes} \sb {L\s} A[[W]] \to 
 \eL \hat{\otimes} \sb {L\s} A[[W]]
 \, | \, 
 f 
 \,\text{\it \/
 is a differential} \\ 
 \K \hat{\otimes} \sb {L\s}A[[W]] \text{\it-automorphism 
 continuous with respect to
 the $W$-adic topology} \\
 \,\text{\it \/ and congruent to the identity modulo nilpotent elements 
 }
 \}
 \end{multline*}
 for a commutative $L\n$-algebra $A$.
 Here the completion is taken with respect to the $W$-adic topology.
 See Definition 2.19 in \cite{mori09}. 
\par
Then the group functor $\infgal (L/k)$ operates on the functor 
$\mathcal{F}\sb {L/k}$ in such a way that 
the operation $(\infgal (L/k), \, \mathcal{F}\sb {L/k} )$ is a 
torsor (Theorem (5.11), \cite{ume96.2}). 
\subsubsection{Origin of the group structure}\label{9.28a}
For the differential equations, the Galois group is a group functor. 
We are going to generalize differential Galois theory in such a way that 
the Galois group is a quantum group. Quantum group is a generalization of affine algebraic group. We can not, however, regard a quantum group as a group functor. Therefore, we have to understand the origin of the group functor $\infgal$. We illustrate it by an example.
\begin{example}\label{14.5.29b} Let us consider a differential field extension 
$$
L/k:=( \com (y), \, \gd ) /\com 
$$
such that $y$ is transcendental over the field $\com$ and 
\begin{equation}\label{9.27f}
\gd (y) = y \qquad \text{ and }\qquad \gd (\com) = 0 
\end{equation}
so that $k=\com $ is the field of constants 
of $L$.
\end{example}
The universal Taylor morphism 
$$
\iota \colon L \to L\n[[X]]
$$ 
maps $y\in L$ to 
$$
Y :=y\exp X \in L\n[[X]].
$$ 
Since the field extension $L\n /k\n = \com (y)/\com$, taking $d/dy \in \mathrm{Der}(L\n/k\n)
 $ as a basis of $1$-dimensional $L\n$-vector space 
$\mathrm{Der}(L\n/k\n)$, we get 
$L\s := (L\n , \, d/dy)$. 
As we have relations 
\begin{equation}\label{9.27a}
\frac{\partial Y}{\partial X} = Y, \qquad 
y\frac{\partial Y}{\partial y} =Y 
\end{equation}
in the power series ring 
$L\s[[X]]$ 
so that the Galois hull $\eL /\K$ is 
\begin{equation}\label{9.27b}
\eL = \K .\com( \,\exp X\, ) , \quad \K =L\s \subset L\s[[X]] 
\end{equation}
by definition of the 
Galois hull. 
\par
Now let us look at the infinitesimal deformation functor $\mathcal{F}\sb{L/k}$. To this end, 
we Taylor-expand the coefficients of the 
power series in $L\s [[X]]$
to get 
$$
\iota :L \to L\s[[X]] \to L\n[[W]][[X]]= L\n[[W, \, X]]
$$
so that 
 $$
 \iota (y) = (y+W)\exp X \in L\n 
[[W, \, X]].
$$
 We identify $L\s[[X]]$ with its image 
in $L\n[[W]][[X]]=L\n[[W, X]]$. In particular we identify 
$Y = y\exp X \in L\s[[X]] $ with $Y(W, X) =(y +W)\exp X \in L\n[[W, X]]$. Equalities \eqref{9.27a} become in $L\n[[W, X]]$
\begin{equation}\label{9.27aa}
\frac{\partial Y(W, X)}{\partial X} = Y(W, X), \qquad 
(y+W)\frac{\partial Y(W, X)}{\partial W} =Y 
\end{equation}
It follows from \eqref{9.27aa}, 
for a commutative 
$L\n $-algebra $A$, an infinitesimal deformation 
$\varphi \in \mathcal{F} \sb{L/k}(A) $ is determined by the image 
\begin{equation}\label{140610a} 
\varphi (Y(W, X)) = c Y(W, X) \in A[[W, \, X]], 
\end{equation}
where $c\in A$. Conversely any invertible element $c\in A$ 
infinitesimally close to $1$ 
defines an infinitesimal deformation so that we conclude
\begin{equation}\label{11.14a} 
\mathcal{F}\sb {L/k}(A) = \{ c\in A \, | \, c-1 \,\text{\it \/ is nilpotent} \}. 
\end{equation}
\par
{ \it Where does the group structure come from?} 
\par There are two ways of answering to this question, which are closely related.
\par 
(I) Algebraic answer.
\par 
%The Galois hull $\eL$ is generated by $\iota (y) $ over 
%$\K =L\s$. So, since $\varphi $ is identity on $L\s$, the infinitesimal deformation $\varphi$ is 
%determined by the image $\varphi (y)$. 
%Let us set 
%$$
%\iota (y) = y \exp X\in L\s[[X]]
%$$ 
%that we identify with $(y+ W)\exp X\in L^n [[X]]$, by $Y(W, X)$
By \eqref{140610a}, 
we have 
$$
\varphi (y) = c(y +W )\exp X \in A[[W, X]], 
$$
 where $c- 1\in A$ is a nilpotent element. 
Consequently we have 
\begin{equation}
\varphi (y) = Y((c-1)y + cW, X).
\end{equation}
In other words $\varphi (y)$ coincides with
$$
Y(W, X)\, | \, \sb {W = (c-1)y + cW}. 
$$
Equivalently $\varphi (y) $ is obtained by substituting $(c-1)y + cW$ for $W$ in $Y(W, X)$. 
This is quite natural in view of differential equations
 \eqref{140610a}. We only have to look at the initial condition at $X=0$ 
 of the solutions $Y(W, X)$ and $\varphi (y)=
 c(y+ W)Y(W, X)$ of the differential equation $\partial {\rm Y} / \partial X = {\rm Y}$. 
 The transformation 
 \begin{equation}\label{140610b}
 W \mapsto (c-1)y + cW \text{ where $c\in A$ and $c-1$ is nilpotent, }
 \end{equation}
 is an infinitesimal coordinate transformation of the initial condition and the multiplicative structure of $c$ is nothing but the composite of coordinate transformations 
\eqref{140610b}. 
 \par 
(II) Geometric answer.
\par
To see this geometrically, we have to look at the dynamical system defined by the differential equation \eqref{9.27f}. Geometrically the differential equation \eqref{9.27f} gives us a dynamical system
on the line $\com $. 
$$
 y \mapsto Y=y\exp X 
 $$
 describes the dynamical system. 
 Observing the dynamical system through 
 algebraic differential equations, is equivalent to considering the 
 deformations of the Galois hull. 
 So the (infinitesimal) deformation functor 
 measures the ambiguity of the observation. 
 In other words, the result due to our method 
 is \eqref{11.14a}. In terms of the initial condition, it looks as 
 $$
 y \mapsto cY\, |\sb {X=0} = cy\exp X\, |\sb{X=0}= cy.
 $$ 
Namely, 
\begin{equation}\label{9.27d}
y \mapsto cy.
\end{equation} 
If we have two transformations \eqref{9.27d} 
 $$
 y \mapsto cy, \qquad y \mapsto c^\prime y 
 $$
 the composite transformation corresponds to the product 
 $$
 y\mapsto cc^\prime y.
 $$ 
 Our generalization depends on the first answer (I).
 See Section \ref{14.5.29a}. 
\subsection{Difference
Galois theory}\label{0829a} 
\par 
If we replace the universal Taylor morphism by the universal Euler morphism, we can 
construct 
a general Galois theory of difference equations (\cite{mori09}, 
\cite{morume09}).
\par
\subsubsection{Universal Euler morphism}\label{0820b}
 Let $(R,\, \sigma )$ be a $C$-difference algebra so that $\sigma :R \to R$ is a $C$-algebra automorphism of a commutative 
$C$-algebra $R$. See Remark \ref{141225b}. When there is no danger of confusion of the automorphism 
 $\sigma$, we simply say the $C$-difference algebra $R$ without referring to the automorphism $\sigma$. We often have to talk however about the abstract ring $R$ that we denote by $R\n$. 
 For a commutative ring $S$, we denote by $F(\Z , \, S)$ the ring of functions 
 on 
 the set of integers 
 $ 
 \Z % = \{\cdots , \, -1, \, 0, \, 1, \, 2, \cdots \}
 $ 
 taking values in the ring $R$. 
 For a function $f \in F(\Z , \, S)$, we define the shifted function 
 $\Sigma f \in F(\Z , \, S)$ by 
 $$
( \Sigma f)(n) = f( n + 1) \quad \,\text{\it \/ for every } n \in \Z.
 $$
 Hence 
 the shift operator 
 $$
 \Sigma :F(\Z , \, S) \to F(\Z , \, S)
 $$
 is an automorphism of the ring $F(\Z , \, S)$ 
 so that $(F(\Z , \, S), \, \Sigma )$ is a difference ring. 
\begin{remark}
In this Paragraph \ref{0820b}
 and the next \ref{9.19g}, 
 in particular for the existence of the universal Euler morphism, 
 we do not need the 
 commutativity assumption of the 
underlying ring. 
\end{remark}
\par 
 Let $(R,\, \sigma)$ be a difference ring and $S$ a ring. 
 An Euler morphism is a difference morphism 
 \begin{equation}\label{1a2.1} 
 ( R, \, \sigma ) \to ( F ( \Z , \, S), \, \Sigma ). 
 \end{equation}
 Given a difference ring $(R,\, \sigma )$, 
 among the Euler morphisms \eqref{1a2.1}, there exists the universal one. 
 In fact, for an element $a\in R$, we define the function $u[a] \in F (\Z , \, R\n)$ 
 by 
 $$
u[ a](n) = \sigma ^n (a) \quad \,\text{\it \/ for } n \in \Z .  
 $$ 
 Then the map
 \begin{equation}\label{9.19e}
\iota \colon (R, \, \sigma ) \to 
( F(\Z , \, R\n ), \, \Sigma ) \qquad a \mapsto u[a] 
\end{equation}
is the universal Euler morphism (Proposition 2.5, \cite{mori09}). 
\subsubsection{Galois hull $\eL /\K$ for a difference field extension $L/k$}\label{9.19g}
Let $(L, \,\sigma )/(k,\, \sigma )$ be a difference field extension such that 
the abstract field $L\n $ is finitely
 generated over the abstract base field $k\n$. We constructed the Galois hull 
 $\eL / \K$ 
 as 
 in the differential case. Namely, 
we take a mutually commutative basis 
 $$
 \{ D \sb 1, \, D \sb 2 , \cdots 
 , D\sb d \}
 $$
 of the $L\n$-vector space $\mathrm{Der} \, (L\n / k\n )$ 
 of $k\n$-derivations of the abstract field $L\n$. 
 %%%%%%%%%%%%%
%%%%%%%%%%%%%%%%%%%%%%%%%%%%%%%%%%%%%%%%%%%%%%%%%%%%%%%%%%%%%%%%%%%%%%%%
 We introduce the partial differential field 
$$
L\s := ( L\n , \, \{ D\sb 1 , \, D\sb 2, \, \cdots , D\sb d \} ). 
$$ 
Similarly we define a differential structure on the ring 
$F(\Z , \, L\n)$ of functions taking values in $L\n$ by considering  
the derivations $$\{D\sb 1 , \, D\sb 2, \,
 \cdots , D\sb d \} .$$ In other words, 
 we work with the differential ring 
 $F(\Z ,\, L\s )$. 
 So the ring 
 $F( \Z , \, L\n )$ has a difference-differential structure 
 defined by the shift operator $\Sigma$ and the set 
 $$\{ D\sb 1 , \, D\sb 2, \, \cdots , D\sb d \}$$ 
 of derivations. Since there is no danger of confusion of the choice of the difference operator $\Sigma$, 
 we denote this difference-differential ring by 
 $$
 F(\Z , \, L\s) . 
 $$
 %%%%%%%%%%%%%%%%%%%%%%%%%%%%%%%%%%%%%%%%%%%%%%%%
We have the universal Euler morphism 
\begin{equation}\label{1m27.0.5}%m27.0.5
\iota \colon L \to F (\Z , \, L \n)
\end{equation}
that is a difference morphism. We added further the
$\{ D\sb 1 , \, D\sb 2, \, \cdots , D\sb d \}$-differential structure 
on $F(\Z , \, L\n )$ or we replace the target space 
$F(\Z , \, L\n )$ of the universal Euler morphism \eqref{1m27.0.5}
 by $F(\Z , \, L\s )$ so that we have 
 $$
 \iota \colon L \to F ( \Z , \, L\s ). 
 $$
\par 
In Definition \ref{1a4.6} below, we work in the difference-differential ring 
$F(\Z , \, L\s)$ with difference operator 
$\Sigma $ and differential operators 
$
\{ D\sb 1 , \, D\sb 2, \, \cdots , D\sb d \}.
$
We identify with $L\s$ the set of constant functions on $\Z$. 
Namely, 
$$
L \s = \{ f \in F (\Z , L\s ) \, | \, f(0 ) = f(\pm 1) = f(\pm 2) = \cdots \in L\s \}.
$$
Therefore $L\s$ is a difference-differential sub-field of the 
difference-differential 
ring $F ( \Z , \, L\s )$. The action of the shift operator 
on $L\s$ 
 being trivial, the notation is adequate. Similarly, 
we set 
$$
k \s := \{ f \in F (\Z , L\s ) \, | \, f(0 ) = f(\pm 1) = f(\pm 2) = \cdots \in k \subset L\s \}.
$$
So both the shift operator and the derivations act trivially on 
$k\s$ and so $k\s$ is a difference-differential sub-field of $L\s$ and hence 
of the difference-differential algebra $F(\Z , \, L\s )$. 
\begin{definition}\label{1a4.6}
The Galois hull $\eL / \K $ is a difference-differential sub-algebra extension 
of $F( \Z , \, L\s)$, where 
$\eL$ is the difference-differential sub-algebra generated by the 
image $\iota (L)$ and $L\s$ 
and $\K $ is the sub-algebra generated by the image $\iota (k ) $ 
and $ L\s$. So $\eL / \K$ is a difference-differential algebra extension with difference operator $\Sigma $ and derivations 
$
\{ D\sb 1 , \, D\sb 2, \, \cdots , D\sb d \}.
$
\end{definition} 
%%%%%%%%%%%%%%%%%
 %%%%%%%%%%%%%%%%%%%%%%%%%%%%%%%%%%%%%%%%%%%%%%%%%%%%%%%
 \subsubsection{The functor $\mathcal{F} \sb{L/k}$ of infinitesimal deformations
 for a difference field extension}\label{9.25d}
%Combining the universal Euler morphism 
%and 
%%%%%%%%%%%%%%%%%%%%%%%%%%%%%%%%%%%%%%%%%%%%%%%%%%%
For the partial differential 
field $L\s$, 
we have the universal Taylor morphism
\begin{equation}\label{1m27.1}%m27.1
\iota \sb {L\s} \colon L\s \to L\n [[ W\sb 1, W\sb 2, \cdots , W\sb d ]]
=L\n [[W]]. 
\end{equation}
 The universal Taylor morphism \eqref{1m27.1} gives a 
 difference-differential morphism 
 \begin{equation}\label{1m27.2}
 F(\Z , L\s ) \to F(\Z , L\n [[ W\sb 1, W\sb 2, \cdots , W\sb d ]] ). 
 \end{equation}
Restricting the morphism \eqref{1m27.2} to the difference-differential 
sub-algebra $\eL $ of $F(\Z , L\s )$, we get a difference-differential 
morphism 
$\eL \to F(\Z , L\n [[ W\sb 1, W\sb 2, \cdots , W\sb d ]] )$
that we denote by $\iota $. So we have the 
difference-differential 
morphism 
\begin{equation}\label{1m27.4}
\iota \colon \eL \to 
F(\Z , L\n [[ W\sb 1, W\sb 2, \cdots , W\sb d ]] ).
\end{equation}
Similarly for every commutative $L\n$-algebra $A$, 
thanks to the differential morphism 
$$
L\n[[W]] \to A[[W]], 
$$
arising from the structural morphism $L\n \to A$, 
we have the 
canonical difference-differential morphism 
\begin{equation}\label{1m27.3}
\iota \colon \eL \to F(\Z , A [[ W\sb 1, W\sb 2, \cdots , W\sb d ]] ).
\end{equation}
We define the functor 
$$
\mathcal{F}\sb {L/k} \colon (Alg/ L\n) \to (Set)
$$
from the category $(Alg/L\n)$ of commutative $L\n$-algebras to the category 
$(Set)$ of sets, by associating to a commutative $L\n$-algebra $A$, the set of 
 infinitesimal deformations of the canonical morphism \eqref{1m27.4}. 
So 
\begin{multline*}
\mathcal{F}\sb{L/k}(A)
= \{ f\colon \eL \to 
F(\Z, \, A [[ W\sb 1, W\sb 2, \cdots , W\sb d ]] )\, | \, 
 f \text{\it is a difference-differential }
 \\
 \,\text{\it \/morphism 
 congruent to the canonical morphism } \iota 
 \,\text{\it \/ modulo nilpotent elements} \\ 
 \,\text{\it \/such that \/} 
 f = \iota \,\text{\it \/ when restricted on the sub-algebra } \K 
\}.
 \end{multline*}
See Definition 2.13 in \cite{mori09}, 
 for a rigorous definition. 
 \subsubsection{Group functor $\infgal(L/k)$ of infinitesimal automorphisms for a difference field extension} 
The Galois group in our Galois theory is the group functor 
$$
\infgal (L/k) \colon (Alg/L\n ) \to (Grp)
$$
 defined 
 by 
 \begin{multline*}
 \infgal (L/k) (A) =
 \{
 \, 
 f \colon \eL \hat{\otimes} \sb {L\s} A[[W]] \to 
 \eL \hat{\otimes} \sb {L\s} A[[W]]
 \, | \, 
 f 
 \,\text{\it \/
 is a difference-differential} \\ 
 \K \hat{\otimes} \sb {L\s}A[[W]] \,\text{\it \/-automorphism 
 continuous with respect to
 the $W$-adic topology} \\
 \,\text{\it \/ and congruent to the identity modulo nilpotent elements 
 }
 \}
 \end{multline*}
 for a commutative $L\n$-algebra $A$. 
 Here the completion is taken with respect to the $W$-adic topology.
 See Definition 2.19 in \cite{mori09}. 
\par
Then the group functor $\infgal (L/k)$ operates on the functor 
$\mathcal{F}\sb {L/k}$ in such a way that 
the operation $(\infgal (L/k), \, \mathcal{F}\sb {L/k} )$ is a 
torsor (Theorem2.20, \cite{mori09}). 
\par
The group functor $\infgal (L/k)$ arises from the same origin as in the 
differential case, namely from 
the automorphism of the initial conditions as we explained in 
\ref{9.28a}. 
In the quantum case too, where in Hopf Galois theory, the Galois hull $\eL$ is non-commutative.
 We we are going to see that we can apply this principle to define the Galois group that is a quantum group, in the quantum case. 
See Section \ref{10.4a}, The \ First Example, Section \ref{10.4b}, The Second Example and Section 
\ref{10.4c}, The 
Third Example.
\subsection{Introduction of more precise notations}
\label{9.25a}
So far, we explained general differential 
Galois theory and general difference Galois theory. 
To go further, we have to make our notations more precise.
\par
For example, we defined the Galois hull for a differential field extension 
in Definition \ref{a4.6} and the
Galois hull for a difference field extension in Definition \ref{1a4.6}. 
Since they are defined by the same principle, we denoted both of them by $\eL / \K$. 
 We have to, however, distinguish them. 
 \begin{definition} 
We denote the Galois hull
for a differential field extension by $\eL \sb \gd /\K \sb\gd$
and we use the symbol $\eL \sb \sigma / \K \sb \sigma $ 
for the Galois hull of a difference field extension. 
\end{definition}
We also have to distinguish the functors $\mathcal{F}\sb {L/k}$ 
and $\infgal 
(L/k)$
in the differential case and in the difference case: we add the suffix $\gd$ for the differential case 
and the suffix $\sigma$ for the difference case:
\begin{enumerate}
\renewcommand{\labelenumi}{(\arabic{enumi})}
\item 
We use 
$\mathcal{F}\sb {\gd L/k}$ and $\infgal \sb \gd
(L/k)$, when we deal with differential algebras.
\item 
We use $\mathcal{F}\sb {\sigma L/k}$ and $\infgal \sb \sigma 
(L/k)$ for difference algebras.
\end{enumerate}
\par 
We denoted, according to our convention, 
for a commutative algebra $A$ the category of commutative $A$-algebras by 
$(Alg/A)$. As we are going to consider the category of not necessarily commutative 
$A$-algebras. This notation is confusing. 
So we clarify the notation.
%\begin{definition} 
% We often denote the category of commutative 
%$A$-algebras by $(CAlg/A)$.
%\end{definition}
%%%%%%%%%%%%%%%%%
%%%%%%%%%%%%%%%%%
\section{Hopf Galois theory}\label{4.5.5a}
%So far we treated difference equations and differential equations. 
Picard-Vessiot theory is a Galois theory of linear differential or difference equations. 
The idea of introducing Hopf algebra in Picard-Vessiot theory is traced back to Sweedler \cite{swe69}. 
Specialists in Hopf algebra succeeded in unifying 
Picard-Vessiot theories for differential equations and difference equations \cite{amaetal09}. They further succeeded in generalizing the Picard-Vessiot 
theory for difference-differential equations, where the operators are 
not necessarily commutative. 
 Heiderich \cite{hei10} combined the idea of 
Picard-Vessiot theory via Hopf algebra with our general 
Galois theory for non-linear equations \cite{ume96.2}, \cite{mori09}. 
This is a wonderful idea. 
After our Examples, it becomes, however, apparent that 
his results require a certain modification
in the non-co-commutative case.
His general theory includes a wide class of difference and differential algebras. %It seems, however, that some algebras with operators are excluded of his theory. 
%Sesquilinear difference algebra in Andr\'e \cite{and} is such an example.  
\par
%%%%%%%%%%%%%%%%%%%%%%47
There are two major advantages in his theory. 
\begin{enumerate}
\renewcommand{\labelenumi}{(\arabic{enumi})}
\item Unified study of
differential equations and difference equations in non-linear case.
\item Generalization of universal Euler morphism and Taylor morphism.
\end{enumerate}
\par
$C$ being the field, 
 for $C$-vector spaces $M, \, N$, we denote 
by $\sb C \mathbf{M} (M, N)$ the set of $C$-linear maps from $M$ to $N$. 
%%%%%%%%%%%%%%%%%%%%%%%%%%46n
\begin{example}\label{9.18a}
Let $\calH : = C[\G\sb {a C}] = C [t]$ be the $C$-Hopf algebra of the coordinate 
ring of the additive group scheme $\G\sb {a C}$ over the field $C$. Let $A$ be a commutative $C$-algebra 
and 
$$
\Psi\in \, \sb C\M ( A\otimes\sb C \calH, A)  
= \, \sb C \M (A, \, \sb C\M (\calH , A))
$$ 
so that $\Psi$ defines two $C$-linear maps 
\begin{enumerate}
\renewcommand{\labelenumi}{(\arabic{enumi})}
\item $\Psi\sb 1 \colon A\otimes\sb C \calH \to A$,
\item $\Psi \sb 2 :A \to \, \sb C\M (\calH , A)$.
\end{enumerate}
\begin{definition}\label{141227a} 
We keep the notation of Example \ref{9.18a}
We say that $(A, \Psi )$ is an $\calH$-module algebra if the following equivalent conditions are satisfied. 
\begin{enumerate}
\renewcommand{\labelenumi}{(\arabic{enumi})}
\item The $C$-linear map $\Psi\sb 1 \colon A\otimes\sb C \calH \to A$
makes $A$ into a left $\calH$-module in such a way that we have in the algebra $A$, 
$$
h(ab) = \sum (h\sb {(1)}a) (h\sb{(2)}b) \in A, 
$$
for every element $h\in \mathcal{H}$ and $a, \, b \in A$, where 
we
use the sigma notation so that 
$$
\Delta (h) = \sum h\sb {(1)} \otimes h\sb{(2)}, 
$$
$\Delta :\calH \to \calH \otimes \calH$
being the co-multiplication of the Hopf algebra $\mathcal{H}$. 
\item The $C$-linear map $$\Psi \sb 2 :A \to \, \sb C\M (\calH , A)$$
%$\Delta : \mathcal{H} \to \mathcal{H} \otimes \mathcal{H}$ being yhe co-multiplication,.
is a $C$-algebra morphism, the dual 
$\sb C\M (\calH , A )$ 
of co-algebra $\calH$ being a $C$-algebra. 
\end{enumerate} 
cf. p.153 of Sweedler \cite{swe69}.
\end{definition}
Concretely the dual algebra $\sb C\M (\calH , A)$ is the formal power series ring $A[[X]]$.
\par It is a comfortable exercise to examine
that $( A, \Psi )$ is an 
$\calH$-module algebra if and only if 
$A$ is a differential algebra with derivation $\gd$ such 
that $\gd (C) =0$. 
 When the equivalent conditions are satisfied, for every element $a$ in the algebra $A$, 
$\Psi (a\otimes t) = \gd (a)$ and 
the $C$-algebra morphism 
$$
\Psi \sb 2 :A \to \, \sb C\M (\calH , A) =A[[X]]
$$
is the universal Taylor morphism. 
So 
$$
\Psi\sb{2}(a) = \sum\sb{n=0}^{\infty} \frac{1}{n!}\delta^{n}(a)X^{n}\,\, \in A[[X]]
$$ 
for every $a \in A$. 
See Heiderich \cite{hei10}, 2.3.4.
\end{example}
In Example \ref{9.18a}, we explained the differential case.
If we take 
the Hopf algebra $C[\G\sb {m C}]$ of the coordinate ring of the multiplicative group $\G\sb {m C}$ 
 for $\calH$, we get difference structure and the universal 
Euler morphism. See \cite{hei10}, 2.3.1. 
More generally we can take any Hopf algebra $\mathcal{H}$ to get 
an algebra $A$ with operation of the algebra $\mathcal{H}$ and a 
morphism 
$$
\Psi \sb 2 :A \to \, \sb C\M (\calH , A) 
$$
generalizing the universal 
Taylor morphism and Euler morphism. So we can 
define the Galois hull $\eL /\K $ and develop 
a general Galois theory for a field extension $L/k$ 
with operation of the algebra $\mathcal{H}$. 
In the differential case as well as in the difference case, the corresponding Hopf algebra $\mathcal{H}$ is co-commutative so that the dual algebra $ \sb C \mathrm {M}(\mathcal H, \, A)$
is a commutative algebra. Consequently the 
Galois hull $\eL /\K$ that are sub-algebras in the 
commutative algebra 
 $ \sb C \M(\mathcal H, \, A)$.
 In these cases, the Galois hull is an algebraic counterpart of 
the geometric object, algebraic Lie groupoid. See Malgrange \cite{mal01}. Therefore 
the most fascinating question is 
 \begin{question} %\label{14,6,29a}
Let us consider a non-co-commutative bi-algebra 
$\mathcal{H}$ and assume that the Galois hull $\eL/\K$ that 
is a sub-algebra of the dual algebra 
$ \sb C \mathrm {M}(\mathcal H, \, A)$, is not a commutative algebra. 
Does the Galois hull $\eL / \K$ quantize the 
algebraic Lie groupoid?
\end{question}  
We answer affirmatively the question by analyzing examples in 
\QSI field extensions. 
\begin{remark}
Looking at the works of Hardouin \cite{har10} and Masuoka and Yanagawa \cite{masy}, even if we consider a twisted Hopf algebra $\mathcal{H}$, so far as we consider linear difference-differential equations, the Galois hull $\eL $ often happens 
to be a commutative sub-algebra of the non-commutative algebra $\sb{C} M(\mathcal{H} ,\, A)$ and the Galois group is a linear algebraic group. See also \cite{ume11}. 
We show by examples that quantization of Galois theory really occurs for non-linear equations. We prove further that the first of our Examples reduces to a linear equation giving us 
the First Example of linear equation where quantization of Galois theory takes place. 
%How non-co-commutative the bialgebra $\mathcal{H}$ may be, so far as one considers linear equations, the Galois hull $\eL / \K$ is a commutative sub-algebra of the 
%non-commutative algebra 
%$ \sb C \mathrm {M}(\mathcal H, \, A)$. 
%Hence one does not encounter quantum groups, except for linear algebraic groups, studying generalized 
%Picard-Vessiot theories. See Hardouin \cite{har10} and Umemura \cite{ume11}. 
\end{remark}
%%%%%%%%%%%%%%%%%%%%%%%%%%%%%%%%%%%%%%%%%%%%%%%%%%%%%%%%%%%%%%%%%%41
%We consider mostly commutative algebras. So when we speak of an 
%algebra, without 
%mentioning that it is non-commutative, it is a commutative algebra. 
 % 

Let $q$ an element of the field $C$.
We use a standard notation of $q$-binomial coefficients. 
To this end, let $Q$ be a variable over the field $C$. 
\par 
We set $[n]\sb Q = \sum \sb{i=0} ^ {n-1} Q^i \in C[Q]$ for positive integer 
$n$. We need also $q$-factorial 
$$
[n]\sb Q ! :=\prod \sb{i=1} ^n [i]\sb Q \qquad 
\text{\it \/ for a positive integer $n \qquad$ and } \qquad %\qquad 
[0]\sb Q ! := 1 . 
$$ 
So $[n]\sb Q \in C[Q]$.
The $Q$-binomial coefficient is defined for $m, n \in \N $ by 
$$
\binom{m}{n}\sb Q = \begin{cases}
\frac{[m]\sb Q!}{[m-n]\sb Q ![n]\sb Q!} & \text{\it \/ if } m \ge n, \\
0 & \text{\it \/ if } m < n.
\end{cases}
$$
Then we can show that the rational function 
$$
\binom{m}{n}\sb Q \in C(Q)
$$
is in fact a polynomial or 
$$
\binom{m}{n}\sb Q \in C[Q].
$$
We have a ring morphism 
\begin{equation}\label{9.17c}
 C[Q] \to C[q], \qquad Q \mapsto q
 \end{equation}
 over $C$ and we denote 
the image of the polynomial 
$$\binom{m}{n}\sb Q $$ under morphism \eqref{9.17c} by 
$$
\binom{m}{n}\sb q.
$$ 
\par

 %%%%%%%%%%%%%%%%%%%%%%%%%%%%%%44
%%%%%%%%%%%%%%%%%%%%%%%%%%%%%%%%%%%%%%%%%%%%%%%%%%%%%%%%%%%%%%%%%%%%%%%%%%%
\subsection{$q$-skew iterative $\sigma$-differential algebra \cite{har10}, \cite{hay08} }%%%%%%%%%%%%%%%%%%%%
%%%%%%%%%%%%%%%%%%%%%%%%%%%%%%%%%%%%%%%%%%
The first non-trivial example of a Hopf Galois theory dependent on 
a non-co-commutative Hopf algebra is Galois theory of $q$-skew iterative $\sigma$-differential field extensions, 
abbreviated as \QSI field extensions. 
Furthermore we simply call them \itqsi field extensions.
\subsubsection{Definition of \QSI algebra}
\begin{definition}\label{a3.3} 
Let 
$q\not= 0$ be an element of the field $C$. 
A $q$-skew iterative $\sigma $-differential algebra 
$( A, \,\sigma ,\, \sigi , \, \theta ^* ) = ( A, \sigma , \{ \theta ^{(i)}\} \sb{i\in \N} )$, a 
%qSI$\sigma$-
\QSI algebra or a \itqsi qlgebra for short, 
 consists of a 
$C$-algebra $A$ that is eventually non-commutative,
a $C$-automorphism $\sigma :A \to A$ 
of the $C$-algebra $A$ and a family 
$$
\theta ^{(i)}\colon A \to A \qquad \text{\it \/ for $i \in \N$} 
$$
of $C$-linear maps, called derivations, 
satisfying the following conditions.
%\end{definition}
%%%%%%%%%%%%%%%%%%%%%%
\begin{enumerate}
\renewcommand{\labelenumi}{(\arabic{enumi})}
\item $\theta ^{(0)} = \Id\sb A$, 
\item $\theta ^{(i)}\sigma = q^i \sigma \theta ^{(i)} \qquad \text{ for every } i \in \N$, 
\item $\theta ^{(i)}(ab) =
\sum\sb{l+m= i, \, l, m \ge 0} \,
\sigma^{m}
(
\theta ^{(l)}(a) 
) 
\theta ^{(m)}(b)$ for every $i \in \N$ and $a,\, b \in A$,
\item $\theta ^{(i)}\circ \theta ^{(j)} =\binom{i+j}{i} \sb q \theta ^{(i+j)}$ for every $i, \, j \in \N$.
\end{enumerate}
%We denote the qSI$\sigma$-algebra 
%$(A, \sigma , \gd ^*) =
%(A, \sigma , \{ \gd ^{(i)}\}\sb {i \in \N})$ 
%or simply by $A$.
We say that 
an element $a$ of the \QSI algebra $A$ is a constant if 
$\sigma (a) = a$ and $\theta ^{(i)}(a) = 0$ for every $i\ge 1$. 
\par
A morphism of \QSI $C$-algebras is a 
$C$-algebra morphism compatible with the automorphisms $\sigma$ and the derivations 
$\theta ^ *$. 
\end{definition} 
\par 
Both differential algebras and difference algebras are \QSI algebras as we see below.
\begin{remark}\label{141225a}
There is also a weaker version of \QSI differential algebra, in which we do not require that $\sigma$ is a $C$-linear automorphism of $A$. 
\end{remark}

%\begin{align*}
%(1) \, & \theta ^{(0)} = \Id, \\ 
%(2) \, & \theta ^{(i)}\sigma = q^i \sigma \theta ^{(i)} \qquad \text{ for every } i \in \N, \\
%(3) \, & \theta ^{(i)}(ab) =
%\sum\sb{l+m= i} 
%\sigma^{m}
%(
%\theta ^{(l)}(a) 
%) 
%\theta ^{(m)}(b), \\
%(4)\, &\theta ^{(i)}\circ \theta ^{(j)} =\binom{i+j}{i}\sb q \theta ^{(i+j) 
%\end{align*}
%%%%%%%%%%%%%%%%%%%%%%%%%%%411
\subsubsection{Difference algebra and a \QSI algebra}\label{9.19c}
 Let $A$ be a commutative $C$-algebra and $\sigma :A \to A$ be a $C$-automorphism of the ring $A$. So $( A , \sigma )$ is a difference algebra. 
 If we set $\theta ^{(0)} = \Id \sb A$ and 
 $$
 \theta ^{(i)} (a) = 0 \text{\it \/ for every element $a\in A$ and for } i =1,\,2, \, 3, \, \cdots .
 $$ 
 Then $(A, \, \sigma , \, \sigi , \, \theta ^ *)$ is a 
 \QSI algebra. 
 \par 
 Namely we have a functor of the category $(Dif\!f'ce Alg/C)$ of $C$-difference algebras to the category $(q\text{-}SI \sigma \text{-}dif\!f'ial Alg/C)$ of 
 \QSI algebras over $C$: 
\begin{equation*}
(Dif\!f'ce Alg/C) \rightarrow (q\text{-SI} \sigma \text{-}dif\!f'ial Alg/C). 
\end{equation*} 
 \par 
 Let $t$ be a variable over the field $C$ and 
 let us now assume
\begin{equation}\label{8.8a}
q^n \not= 1 \qquad \text{\it \/for every positive integer } n.
\end{equation}
 We denote by $\sigma \colon C(t) \to C (t)$ the $C$-automorphism of the rational function field $C(t)$ sending the variable $t$ to $qt$. 
We consider a difference algebra extension $(A, \, \sigma )/(C(t), \, \sigma )$. 
If we set 
$$
\theta ^{(1)}(a)= \frac{\sigma (a) - a}{(q-1)t} \qquad 
\text{\it \/ for every element }a \in A
$$ 
and 
$$
\theta ^{(i)}=\frac{1}{[i]\sb q !}\{\theta ^{(1)}\}^i 
\qquad 
\text{\it \/ for } i=2, \, 3, \, \cdots .
$$
 Then $(A, \sigma , \theta^{*} )$ is a \QSI algebra. Therefore if $q \in C$ satisfies \eqref{8.8a}, then we have a functor 
\begin{equation}
(Dif\!f'ce Alg/(C(t),\sigma)) \rightarrow (q\text{-SI} \sigma \text{-}dif\!f'ial Alg). 
\end{equation}
\begin{remark}\label{141225b}
In coherence with Remark \ref{141225a}, when we speak of 
difference $C$-algebra $(A, \, \sigma )$, we assume that 
$\sigma :A \to A$ is a $C$-linear automorphism. 
\end{remark}

 \subsubsection{Differential algebra and \QSI algebra}
 \label{9.19b}
 Let $(A,\, \theta )$ be a $C$-differential algebra such that 
 the derivation $\theta : A \to A$ is $C$-linear. 
 % the field $C$ is a subfield of the ring $C\sb A$ of constants of the differential algebra $A$.
 We set 
\begin{align*}
\theta ^{(0)} &= \Id \sb A, \\
 \theta ^{(i)}&= \frac{1}{i!}\theta^ i\qquad \text{\it \/ for } i = 1, \, 2, \, 3, \cdots .
\end{align*}
 Then $(A, \, \Id \sb A ,\, \theta ^ *)$ is a \QSI algebra for $q=1$. 
 In other words, we have a functor 
 $$
 (Diff'ialAlg/C) \to (q\text{-}SI \sigma \text{-}diff'ial Alg/C)
 $$
 of the category of (\,commutative\,) differential $C$-algebras to the category of 
 \QSI algebras over $C$. 
 We have shown that both difference algebras and differential algebras are 
 particular instances of \QSI algebra.
 
 \subsubsection{Example of \QSI 
 algebra \cite{hei10}}\label{10.3a}
 We are going to see that 
 \QSI algebras live on the border between commutative algebras and non-commutative algebras. The example below seems to suggest 
 that 
 it looks natural to seek \QSI algebras in the category of non-commutative algebras. 
 \par
 An example of \QSI algebra arises from a commutative 
 $C$-difference algebra $(S, \, \sigma)$. We need, however, a non-commutative ring, 
 the twisted power series ring 
 $(S, \, \sigma )[[X]]$ over the difference ring $(S, \sigma )$ that has a natural \QSI algebra structure. 
 \par
 Namely, let $(S, \sigma )$ be the $C$-difference ring so that 
 $\sigma :S \to S$ is a $C$-algebra automorphism of the commutative ring $S$. 
 We introduce the following twisted formal power series ring 
 $(S, \, \sigma ) [[X]]$
 with coefficients in $S$ 
 that is the formal power series ring $S[[X]]$ 
 as an additive group 
 with the following commutation relation 
 $$
 aX = X \sigma (a) \qquad 
 and \qquad Xa = \sigma ^{-1} (a)X \qquad 
 \text {\it for every $a\in S$}. 
 $$
 So more generally 
 \begin{equation}\label{a11.1}
 a X^n = X^ n \sigma ^n (a) \qquad \text{\it and } \qquad 
 X^na = \sigma ^{-n}(a) X^n 
 \end{equation}
 for every $n \in \N$. 
 The multiplication of two formal power series is defined by extending \eqref{a11.1} by linearity. 
 Therefore the twisted formal power series ring $(S, \sigma )[[X]])$ is non-commutative in general. 
 By commutation relation \eqref{a11.1}, 
 we can identify 
 $$
 (S, \sigma )[[X]] = \{\sum \sb{i=0}^\infty X^i a\sb i \, |\, 
 a\sb i \in S \text{\it \/ for every } i\in \N \}
 $$
 as additive groups. 
 \par 
 We are going to see that 
 the twisted formal 
 power series ring 
 has a natural 
 \QSI structure. We define first a ring automorphism 
$$
\hat{\Sigma} \colon (S, \, \sigma ) [[X]] \to 
 (S, \, \sigma )[[X]]
 $$
 by setting 
 \begin{equation}\label{a5.1}
 \hat{\Sigma} (\sum \sb{i=0}^\infty X^i a\sb i ) = \sum 
 \sb{i=0}^\infty X^i q ^i\sigma (a\sb i) \qquad \text{\it \/ for every } i \in \N , 
 \end{equation} 
 for every element 
 $$
 \sum \sb{i=0}^\infty X^i a\sb i \in (S,
 \sigma )[[X]]. 
 $$ 
 As we assume that $\sigma : A \to A$ is an isomorphism, 
 the $C$-linear map, 
$$ \hat{\Sigma}: (A, \sigma )[[X]] \to (A, \sigma )[[X]]$$   is an automorphism of the $C$-linear space.  
The operators 
 $\Theta ^* =\{ \Theta ^{(l)} \}\sb{
 l\in \N}$ are defined by 
 \begin{equation}\label{a4.5}
 \Theta ^{(l)}(\sum \sb{i=0}^\infty X^ia\sb i ) = 
 \sum \sb{i=0}^\infty X^i \binom{i+l}{l}\sb q a\sb {i+ l}
 \qquad \text{\it \/ for every }l \in \N. 
 \end{equation}
 Hence the twisted formal power series ring $((S, \, \sigma ) [[X]], \hat{\Sigma} , \Theta ^* ) $ is a non-commutative \QSI 
 ring. We denote this \QSI ring simply by $(S, \sigma )[[X]]$.
 See \cite{hei10}, 2.3. 
 %%%%%%%%%%%%%
 %%%%%%%%%%%%%%%%%%%%%%%%%%%%%%%%%%%%%%%%%%%%%%%%%%%%%%%%%%%%%%%%%%%%%%%%%%%%%%%%%%%%%%%%%%%%%%%%%%%%%%%%%%%%%%%%%%%%%%%%%%%%%
 In particular, if we take as the coefficient difference ring $S$ 
 the difference ring 
$$(F(\Z , A), \, \Sigma )$$
 of functions on $\Z$ taking values in a ring $A$ defined in 
 \ref{0820b}, %%?
 where 
 $$
 \Sigma :F(\Z , \, A ) \to F(\Z , \, A)
 $$ 
 is the shift operator, we obtain the \QSI ring 
 \begin{equation}\label{10.1a}
 \left( ( F( \Z , A ), \Sigma ) [[X]], \, \hat{\Sigma}
, \, \Theta ^* \right) .
 \end{equation}
\begin{remark}
We assumed that the coefficient difference ring 
$(S, \, \sigma)$ is commutative. The commutativity assumption on the ring $S$ is not necessary. 
Consequently we can use non-commutative ring $A$ in \eqref{10.1a}. 
\end{remark}
\subsubsection{Hopf algebra for \QSI structures}
 As we explained for differential algebras in Definition \ref{141227a}, 
a \QSI structure is nothing but a $\cH \sb q$-module algebra structure for a Hopf 
algebra $\mathcal{H}\sb q$. 
\begin{definition}%??
Let $q \neq 0$ be an element of the field $C$. 
Let $\cH \sb q$ is a $C$-algebra generated over the field $C$ 
 by $s$, $s^{-1}$ and the $t\sb i$'s for $i\in \N$ subject to the relations %??
$$
%\begin{multline}
t\sb 0 =1, \qquad 
ss^{-1} = s^{-1} s = 1, \qquad 
t\sb i s=q^i s t\sb i, \qquad 
q^i t\sb i s^{-1}= s^{-1} t\sb i, \qquad 
t\sb i t\sb j = \binom{i+j}{i}\sb q t\sb {i+j}
$$ 
%\text{ 
for every 
%} 
$i, 
\, j \in \N$.
%%\end{multline}
We define a co-algebra structure $\Delta :\mathcal{H}\sb q \to 
\mathcal{H}\sb q \otimes \sb C \mathcal{H}\sb q$
 by 
$$
\Delta (s) = s\otimes s , \qquad 
\Delta (s^{-1}) = s^{-1}\otimes s^{-1} , \qquad 
\Delta (t\sb l )=\sum\sb {i= 0}^ l s^i t\sb {l-i}\otimes t\sb i 
$$
for every $l \in \N$.
In fact $\mathcal {H}\sb q $ is a Hopf algebra with 
co-unit $\epsilon : \mathcal{H}\sb q \to C$ defined by 
$$
\epsilon (s) = \epsilon (s^{-1}) =1, \qquad \epsilon (t\sb i) = 0
$$
for every $i \in \N$.
Antipode is an anti-automorphism 
$ S:\mathcal {H}\sb q \to \mathcal{H}\sb q $ of the 
$C$-algebra 
$\cH \sb q$
given by 
$$
S(s ) = s^{-1}, \qquad S ( s^{-1}) =s , \qquad 
S(t\sb i )=(-1) ^ i q^{i (i+1)/2} t\sb i s ^{-i} 
$$
for every $i \in \N$.
\end{definition}
\begin{proposition} 
For a not necessarily commutative $C$-algebra $A$, 
there exists a $1:1$ correspondence between the elements of the following two sets. 
\begin{enumerate}
\renewcommand{\labelenumi}{(\arabic{enumi})}
\item 
The set of \QSI algebra structures on the $C$-algebra $A$. 
\item 
The set of $\cH \sb q$-module algebra structures on the $C$-algebra $A$.
 \end{enumerate}
\end{proposition}
This result is well-known. See Heiderich \cite{hei10}. We recall 
for a \QSI algebra $A$, the corresponding 
left $\mathcal{H}\sb q $-module structure is given by 
$$
s \mapsto \sigma, \qquad s^{-1} \mapsto \sigma ^{-1}, \qquad 
 t\sb i \mapsto 
\theta ^{(i)} \text{ for every } i \in \N.
$$

\subsubsection{Universal Hopf morphism for a \QSI algebra}\label{4.5.5b} 
We introduced in \ref{0820b} the 
difference 
ring of functions $(F(\Z , \, A), \,\Sigma )$ on the set $\Z$ taking values in a ring $A$. 
It is useful to denote the function $f$ by a matrix
\begin{equation*}
\left[\begin{array}{ccccccc}
\cdots &-2 & -1 & 0 & 1 & 2 & \cdots \\
\cdots & f(-2) & f(-1) & f(0) & f(1) & f(2) & \cdots
\end{array}\right] .
\end{equation*}
\par 
For an element $b$ 
of a 
difference algebra $(R,\, \sigma )$ or a 
\QSI algebra 
$( R, \sigma , \, \theta ^* )$, we denote 
 by 
$u[b]$ a function on $\Z$ taking values in the abstract ring $R\n$
such that 
$$
u[b](n) = \sigma ^ n (b) \qquad \text{\it \/ for every } n \in \Z 
$$
so that 
\begin{equation*}
u[b] = \left[\begin{array}{ccccccc}
\cdots & -2& -1 & 0 & 1 & 2 & \cdots \\
\cdots & \sigma ^{-2} (b) & \sigi (b) &  b & \sigma^1 (b) & \sigma^2 (b) & \cdots
\end{array}\right] .
\end{equation*}

Therefore $u[b] \in F(\Z , R\n )$.

 \begin{proposition}[Proposition 2.3.17, Heiderich \cite{hei10}]
 \label{a4.1}
 For a \QSI algebra $(R, \, \sigma , \, \theta ^* )$, %hence in particular for an iterative $q$-difference ring $R$, 
 there exists 
 a canonical morphism, which we call the universal Hopf 
 morphism 
\begin{equation}\label{9.19a}
\iota \colon (R, \, \sigma , \, \theta ^* ) 
\to 
 \left((F(\Z , R\n ), \Sigma )[[X]],\, \hat{\Sigma}, \, \hat{\Theta} ^* \right), 
 \qquad a \mapsto \sum \sb{i=0}
^\infty X^i u[\theta ^{(i)}(a)] 
\end{equation}
of \QSI algebras. 
 \end{proposition} %
% Here is the definition of the universal twisted Taylor morphism. 
% $$
% \iota (a) = \sum \sb {i=0}^\infty u[\delta ^ {(i)}(a)]X^i,
% $$
% where we use the following notation. 
% For an element $b \in R$, 
% we denote by $u(b)$ a function on $\N$ taking values in $R$
 %such that 
 %$$
% u[b](n) = \sigma ^n(b) \qquad \text{ for every } n \in \N 
% $$ 
 %so that $u[b] \in F(\N , R\n)$. 
%%%%%%%%%%%%%%%%%%%%%%%%%%%%%%%%%%%%%%%%%%%%%%%%%%%%
\par 
 We can also characterize the universal Hopf morphism as the 
 solution of a universal mapping property. 
 \par 
 When $q=1$ and $\sigma =\Id \sb R$ 
 and $R$ is commutative 
 so that 
 the \QSI ring $(R, \, \Id\sb R , \, \theta ^* )$ is simply a 
 differential algebra as we have seen in \ref{9.19b}, the universal Hopf morphism \eqref{9.19a} is the universal 
 Taylor morphism in \eqref{9.19d}. 
 Similarly a commutative difference ring is a \QSI algebra 
 with trivial derivations as we noticed in \ref{9.19c}. In this case the universal Hopf morphism \eqref{9.19a} is nothing but 
 the universal 
 Euler morphism \eqref{9.19e}. 
%$$ 
%u(\delta ^ {(i)}(a))\in F(\N , R\n )$$
% is a function on $\N$ 
%taking values in $R\n$ such that 
%$$ u(\delta ^ {(i)}(a))(n)= \sigma ^n (\delta ^ {(i)}(a))$.
 Therefore the
 universal Hopf morphism unifies the universal Taylor morphism and the Universal Euler morphism. 
 \par 
 Let us recall the following fact. 
 \begin{lemma}
 Let $(R, \, \sigma , \, \theta ^* )$ be a \QSI domain. 
 If the endomorphism $\sigma \colon R \to R$ is an automorphism, then 
 the field $Q(R)$ of fractions of $R$ has the unique structure of 
 \QSI field extending that of $R$.
 \par 
 %If moreover $R$ is an \iqd algebra, %%%419
 %then the field $Q(R)$ of fractions of $R$ is also an \iqd field.  
 \end{lemma}
 \begin{proof}
 See for example, Proposition 2.5 of \cite{hay08}. 
 \end{proof}
\par 
We can interpret the Example in \ref{10.3a} from another 
view point. We constructed there from a difference ring 
$(S, \, \sigma )$ a \QSI algebra $((S, \, \sigma )[[X]], \, 
\hat{\Sigma} , \, \hat{\Theta} ^* ) $. We notice that this procedure 
is a particular case of Proposition \ref{a4.1}. 
In fact, given a difference ring $(S, \, \sigma) $. So as in \ref{9.19c}, by adding the trivial derivations, we 
get the \QSI algebra $(S, \, \sigma , \, \theta ^*)$, where 
\begin{align*} 
\theta ^{(0)} &= \Id \sb S, \\ 
\theta ^{(i)} &= 0 \qquad \text{\it \/ for }i \ge 1. 
\end{align*}
Therefore we have the universal Hopf morphism 
$$
(S, \, \sigma , \, \theta ^*) \to \left(F(\Z , \, S\n )[[X]], \hat{\Sigma} , \, \hat{\Theta} ^* \right) 
$$
by Proposition \ref{a4.1}.
So we obtained the \QSI algebra $ \left(F(\Z , \, S\n )[[X]], \, \hat{\Sigma}, \, \hat{\Theta} ^* \right) $ as a result of composite of 
two functors. Namely, 
\begin{enumerate}
\renewcommand{\labelenumi}{(\arabic{enumi})}
\item 
The functor$\colon $(\,Category of Difference algebras\,) $\to$
 (\,Category of \QSI algebras\,) of adding trivial derivations
\item
The functor $\colon$ (\,Category of \QSI algebras\,) $\to $ 
(\,Category of 
\QSI algebras\,), $A \mapsto B$ if there exists the universal
Hopf morphism $A \to B$.
\end{enumerate}

 \subsubsection{Galois hull 
$ \eL /\K$ for a \QSI field extension}\label{9.25c} 
 We can develop a general Galois theory for 
 \QSI
 field extensions analogous to our theories in \cite{ume96.1}, \cite{ume06} and \cite{ume07} thanks to the universal Hopf morphism. 
 Let $L/k$ be an extension of 
 \QSI 
 fields such that 
 the abstract field $L\n$ is finitely generated over the abstract field $k\n$. 
 Let us assume that we are in characteristic $0$. General theory in 
 \cite{hei10} %and 
 %\cite{hei2} 
 works, however, also in characteristic $p\ge 0$. 
%%%%%%%%%%%%%%%%%%%%%%%%%%%%%%%%%%%%%%%%%%%%%%%%%%%%%%%%%%%%%%%%%%%%%%%% 
%%%%%%%%%%%%%%%%%%%%%%%%%%%%%%%%%%%%
We have by Proposition \ref{a4.1} the universal Hopf morphism 
\begin{equation} 
 \iota :(L, \, \sigma , \, \theta ^* ) \to \left( (F(\Z , L\n ), \Sigma )[[X]], 
\, \hat{\Sigma},\, \hat{\Theta} ^* \right)
 \end{equation}
 so that the image $\iota (L)$ is a copy of 
 the \QSI field $L$. We have another copy of $L\n$. 
The set 
%$\{ f = \sum \sb {i=0} a\sb i X^i \in F(\N , L\n )[[X]] \, | \, 
\begin{multline}\label{a1.1}
 \{ f = \sum \sb {i=0}^\infty X^i a\sb i \in F(\Z , L\n )[[X]] \, | \, 
a\sb i = 0 \text{\it \/ for every } i \ge 1 \text{\it \/ and }
 \, \Sigma (a\sb 0 )= a\sb 0
 \} \\
 =\{ f \in F(\Z , L\n )[[X]] \, | \, \hat{\Sigma}(f) = f , \, \hat{\Theta}^{(i)} (f) = 0 \text{\it \/ for every } i \ge 1 \} 
\end{multline}
forms the sub-ring 
of constants in the
\QSI algebra of the 
 twisted power series 
 $$\left(
 (
 F(\Z , \, L\n ) , \Sigma 
 )
[[X]],\, \hat{\Sigma}, \, \hat{\Theta}^* 
\right) .
$$ 
We identify $L\n$ with the ring of constants through
the following morphism. 
For an element $a\in L\n$, we denote the constant function $f\sb a$ on $\Z$ taking the value $a \in L\n$ 
so that 
\begin{equation}\label{9.24a}
L\n \to \left(
 (
 F(\Z , \, L\n ) , \Sigma 
 )
[[X]],\, \hat{\Sigma}, \, \hat{\Theta}^* 
\right) 
 , \qquad a \mapsto f\sb a
\end{equation}
is an injective ring morphism.
We may denote the sub-ring in \eqref{a1.1} by $L\n$. In fact, as an abstract ring 
it is isomorphic to the abstract field $L\n$ and the endomorphism $\hat{\Sigma}$ and 
the derivations $\hat{\Theta} ^{(i)}, \, (i \ge 1 )$ operate trivially on the 
sub-ring. %So we have chosen the sub-ring $L\n$ in the twisted formal power series ring $F(\N , L\n ), \Sigma )[[X]].$ 
 \par 
 We are now exactly in the same situation as in 
\ref{9.19h} of the differential case and in \ref{9.19g} of the difference case.
%%%%%%%%%%%%%%%%%%%%%%%%%
 We choose a mutually commutative basis $\{ D\sb 1, D\sb 2, \cdots , D\sb d \}$ of the 
 $L\n$-vector space $\mathrm{Der}(L\n / k\n )$ of $k\n$-derivations. 
 So $L\s := (L\n, \{ D\sb 1, D\sb 2, \cdots , D\sb d \})$ is a 
differential field. 
%%%%%%%%%%%%%%%%%%%%%%%%%%%%%%%%%%%%%%%%%%%%%%%%%%%%%%%%%%%%%%%%%%%%%%41\par 
\par 
So we introduce derivations $D\sb 1,\, D\sb 2 , \, \cdots,\, D\sb d $ 
operating on the coefficient ring $F(\Z , \, L\n )$. In other words,
we replace the target space 
$F(\Z , \, L\n )[[X]]$ by 
 $F(\Z ,\, L\s )[[X]]$. 
Hence the universal Hopf morphism in Proposition \ref{a4.1} becomes 
$$
\iota :L \to F(\Z , \, L\s ) [[X]]. 
 $$
In the twisted formal power series ring
$(F(\Z , L\s)[[X]], \hat{\Sigma}, \, \hat{\Theta} ^* )$, 
we add differential operators 
$$
D\sb 1, D\sb 2 , \cdots, D\sb d .
$$
So we have a set $\mathcal{D}$ of the following operators on the ring 
$(F(\Z , L\s ), \Sigma )[[X]] $.
\begin{enumerate}
\renewcommand{\labelenumi}{(\arabic{enumi})}
%\arabic{en\renewcommand{\labelenumi}{(umi})}
\item The endomorphism $\hat{\Sigma}$. 
$$%%45
\hat{\Sigma} (\sum \sb {i=0}^\infty X^ i a\sb i ) = \sum \sb {i=0}^
\infty 
 X^i q^ i (\Sigma ( a\sb i )),
$$
$\Sigma :F(\Z , L\s ) \to F(\Z , L\s )$ being the shift operator of the ring of functions on $\Z$. 
\item 
The $q$-skew $\hat{\Sigma}$-derivations $\hat{\Theta} ^{(i)}$'s in \eqref{a4.5}. 
\begin{equation*}%\label{a4.5}
 \hat{\Theta} ^{(l)}(\sum \sb{i=0}^\infty X^i a\sb i ) = 
 \sum \sb{i=0}^\infty 
 X^i \binom{l+ i}{l}\sb q a\sb {i+ l} \qquad \text{\it \/ for every } l \in \N .
 \end{equation*}
\item 
The derivations $D\sb 1,\, D\sb 2 ,\, \cdots,\, D\sb d $ 
operating 
through the coefficient ring $F(\Z , L\s )$ as in \eqref{9.19a}. 
 \end{enumerate} 
 \par 
 Hence we may write $( F(\Z , L\s ), \, \mathcal{D} )$, where 
 $$
 \mathcal{D} = \{ \hat{\Sigma} , \, D\sb 1 ,\ D\sb 2,\, \cdots , \, D\sb d, \, \hat{\Theta} ^ * \} \text{\it \/ and } \hat{\Theta} ^ *
 = \{ 
 \hat{\Theta} ^{(i)} \} \sb {i \in \N }. 
 $$
 \par
% For an element $a\in L\s$, we denote by $a\s $ the constant function on $\N $ taking the value $a$ so that 
%$$
%a\s (n) = a \qquad \text{ for every } n \in \Z. 
%$$
We identify using inclusion \eqref{9.24a}
$$
L\s \to F(\Z ,\, L\s )[[X]]. 
$$
%Therefore $a\s \in F(\N , L\s)$. 
% The latter is a sub-ring of the ring of twisted power series ring $F(\N, L\s)[[X]]$. 
% We can canonically identify the ring 
 %$$
 %F(\N , L\s) =\{ \sum\sb{i=0}^\infty X^i a\sb i \, |\, a\sb i = 0 \text{ for every } i \in \N^ *\}.  
 %$$
 % Namely, we have canonical inclusions 
 %$$ 
 %L\s \to F(\N , L\s ) \to F(\N , L\s )[[X]].
 %$$
 We sometimes denote the image $f\sb a$ of an element $a\in L\s$ by 
 $a\s$. 
 \par
 We are ready to define Galois hull as in Definition \ref{a4.6}. 
\begin{definition}\label{4.5.4b}
The Galois hull $\eL / \K $ is a $\mathcal{D}$-invariant sub-algebra 
of $F( \Z , \, L\s)[[X]]$, 
 where 
$\eL$ is the $\mathcal{D}$-invariant sub-algebra generated by the 
image $\iota (L)$ and $L\s$ 
and $\K $ is the 
$\mathcal{D}$-invariant sub-algebra generated by the image $\iota (k ) $ 
and $ L\s$. So $\eL / \K$ is a $\mathcal{D}$-algebra extension. \par 
As in \ref{9.25a}, if we have to emphasize that we deal with \QSI algebras, we denote the Galois hull by 
$\eL\sb {\sigma\, \theta }/
\K \sb{\sigma \, \theta}$.
\end{definition} 
 \par 
We notice that we are now in a totally new situation. 
In the differential case, the universal 
Taylor morphism maps the given fields to the commutative algebra 
of the formal power series ring so that 
the 
Galois hull is an extension of commutative algebras. 
Similarly for the universal Euler morphism of a difference rings. The commutativity of the Galois hull comes from the fact in the differential and the difference case, the theory depends on the co-commutative Hopf algebras. 
When we treat the \QSI algebras, the Hopf algebra 
$\mathcal{H}$ 
is not co-commutative so that the 
Galois hull $\eL /\K$ that is an algebra extension in 
the non-commutative algebra of 
twisted formal power series algebra, 
the dual algebra of $\mathcal{H}$. 
So even if we start from a (\,commutative\,) field extension $L/k$, the 
Galois hull can be non-commutative. See the Examples in Sections 
\ref{10.4a}, \ref{10.4b} and \ref{10.4c}.
We also notice that when $L/k$ is a Picard-Vessiot 
extension fields in \QSI algebra, the 
Galois hull is commutative \cite{ume11}.
\par 
As the Galois hull is a non-commutative, 
if we limit ourselves to the category of commutative $L\n$-algebras $(Alg/L\n)$, 
we can not detect non-commutative nature of the \QSI 
field extension. So it is quite natural to extend the 
functors over the category of not necessarily commutative algebras. 
%%%%%%%%%%%%%%%%%%%%%%%%%%%%%%% 
\subsubsection{Infinitesimal deformation functor 
$\mathcal{F} \sb{L/k}$ for a \QSI field extension.} 
 We pass to the task of defining the infinitesimal deformation functor $\mathcal{F}\sb {L/k}$ and 
 the Galois group functor. The latter is a subtle object and we postpone discussing it until Section \ref{14.5.29a}. Instead we define 
 naively the 
 infinitesimal automorphism functor $\infgal (L/k)$, which does not seem useful in general. 
 \par 
%%%%%%%%%%%%%%%%%%%%%%%%%%%%%%%%%%%%%%%%%%%%%%%%%
 We have the universal Taylor morphism 
 \begin{equation}\label{a4.10}
 \iota \sb{L\s}\colon L\s \to (
 L\n[[W\sb 1,\, W\sb 2, \, \cdots ,\, W\sb d ]], 
 \{ 
 \frac{\partial}{\partial W\sb 1}, \, \frac{\partial}{\partial W\sb 2} ,\, \cdots , \, \frac{\partial}{\partial W\sb d} 
 \}
 )   
\end{equation}
 as in \eqref{m27.1}.
 So by \eqref{a4.10}, we have the canonical morphism 
\begin{equation}\label{a4.11}
 (F(\Z , L\s )[[X]], \, \mathcal{D} ) \to (F(\Z , \, L\n [[W]])[[X]],\, \mathcal{D} ), 
 \end{equation}
 where in the target space 
 $$
 \mathcal{D} = 
 \{ \hat{\Sigma} , 
 \frac{\partial}{\partial W\sb 1},\, \frac{\partial}{\partial W\sb 2} , \, \cdots ,\, \frac{\partial}{\partial W\sb d}, \, 
 \hat{\Theta}^ * \} 
 $$
 by abuse of notation. 
 \par
 For an $L\n$-algebra $A$, the structure morphism 
 $ L\n \to A$ induces 
 the canonical morphism 
 \begin{equation}\label{a4.12}
 (F(\Z ,\, L\n [[W]])[[X]],\, \mathcal{D} ) \to (F(\Z , \, A [[W]])[[X]],\, \mathcal{D} ).
 \end{equation}
The composite of the $\mathcal{D}$-morphisms \eqref{a4.11} and \eqref{a4.12} 
gives us the canonical morphism 
\begin{equation}\label{a4.13}
 (F(\Z , \, L\s )[[X]],\, \mathcal{D} ) \to 
 (F(\Z , \, A [[W]])[[X]],\, \mathcal{D} ).
\end{equation}
 The restriction of the morphism 
 \eqref{a4.13} to the $\D$-invariant sub-algebra $\eL$  
 gives us the canonical morphism 
\begin{equation}
\iota \colon (\eL,\, \mathcal{D} ) \to (F(\Z , \, A [[W]])[[X]], \, \mathcal{D} ).
\end{equation}
 We can define the functors exactly as in Paragraphs \ref{9.25e} for the 
 differential case and \ref{9.25d} for the difference case.
 \begin{definition}[Introductory definition]\label{9.25f}
 We define the functor 
$$
\mathcal{F}\sb {L/k} \colon (Alg/ L\n) \to (Set)
$$
from the category $(Alg/L\n)$ of commutative $L\n$-algebras to the category 
$(Set)$ of sets, by associating to an $L\n$-algebra $A$, the set of 
 infinitesimal deformations of the canonical morphism \eqref{a4.13}. 
\par
Hence 
\begin{multline*}
\mathcal{F}\sb{L/k}(A)
= \{ f\colon (\eL ,\, \D ) \to 
( F(\Z ,\, A [[ W\sb 1, \, W\sb 2, \, \cdots ,\, W\sb d ]] )[[X]],\, \D ) \, | \, 
 f \text{ is an algebra }\\
\text{ morphism compatible with }
 \mathcal{D}, 
 \text{ congruent to}\\
 \text{the canonical morphism } \iota 
 \text{ modulo nilpotent elements} \\ 
 \text{such that } 
 f = \iota \text{ when restricted to the sub-algebra } \K 
\}.
 \end{multline*}
 \end{definition} % Definition in an appropriate notations.} 
The introductory definition \ref{9.25f} is exact, 
analogous to Definitions in \ref{9.25e} 
and \ref{9.25d}, and 
easy to understand.
As we explained in \ref{9.25c}, we,
however, 
 have to consider also deformations over non-commutative algebras, the notation is confusing. 
\par 
We have to treat both the category of 
commutative $L\n$-algebras and that of not necessarily commutative $L\n$-algebras. 
\begin{definition}\label{4.5.7b}
All the associative algebras that we consider 
are unitary and the morphisms between them are assumed to be unitary. For a commutative algebra $R$, we denote by $(CAlg/R )$ the category of 
associative commutative $R$-algebras.
We consider also the category $(NCAlg/R)$ of not necessarily commutative $R$-algebras %To be more precise 
%we denote by $(NCalg/R)$ the category of associative $R$-algebras 
$A$ such that (\,the image in $A$ of\,) $R$ is in the center of $A$. When there is no danger of confusion the category of commutative algebras is denoted simply by $(Alg/R)$. 
\end{definition}
Let us come back to the \QSI field extension $L/k$. We can now give the infinitesimal deformation functors in an appropriate language. 
\begin{definition}
%From now on, we denote 
%the category of commutative $L\n$-algebras by 
 %$(Calg/L\n)$ {\rm when we want to emphasize that 
 %we study only commutative algebras,} 
 % and the category of 
%non-commutative $L\n$-algebras 
%by $(NCAlg/L\n)$. 
The functor $\mathcal{F}\sb {L/k}$ defined in \ref{9.25f} will be denoted by 
$\mathcal{CF}\sb {L/k}$. So we have 
$$
\mathcal{CF}\sb{L/k}\colon (CAlg/L\n) \to (Set). 
 $$
 We extend formally the functor $\mathcal{CF}\sb {L/k}$ in \ref{9.25f}
from the category $(CAlg /L\n)$ to the category 
$(NCAlg/L\n)$.
Namely, we define the functor 
\begin{equation*}
\NCF \colon \NCA \to (Set)
\end{equation*}
by setting 
%%%%%%%%%%%%%%%%%%%%%%%%%%%%%%%%%%%%%%%%%%%%%%%%%%%%%% 
\begin{multline*}
\mathcal{F}\sb{L/k}(A)
= \{ f\colon (\eL ,\, \D ) \to 
( F(\Z ,\, A [[ W\sb 1, \, W\sb 2, \, \cdots ,\, W\sb d ]] )[[X]],\, \D ) \, | \, 
 f \text{ is an algebra }\\
\text{ morphism compatible with }
 \mathcal{D}, 
 \text{ congruent to}\\
 \text{the canonical morphism } \iota 
 \text{ modulo nilpotent elements} \\ 
 \text{such that } 
 f = \iota \text{ when restricted to the sub-algebra } \K 
\}
 \end{multline*}
%%%%%%%%%%%%%%%%%%%%%%%
for 
$A \in ob\, (NCAlg\sb {L/k})$. 
%$$
%\mathcal{NCF}\sb {L/k} \, |\, {(CAlg/L\n)} = \mathcal{CF}\sb{L/k}.
%$$
\end{definition}
In the examples, we consider \QSI structure, differential structure and difference structure of a given field extension $L/k$ and we study Galois groups with respect to the structures. 
So we have to clarify which structure is  
in question. 
For this reason, when we treat \QSI structure, we sometimes 
add suffix $\sigma \, \theta^{*} $ to indicate 
that we treat the \QSI structure 
as in \ref{9.25a}.
For example 
$\mathcal{NCF} \sb {\sigma \theta ^* L/k}$.
\subsubsection{Definition of commutative Galois group functor $\cinfgal(L/k)$}
%The definition of the group functor $\infgal (L/k)$ is similar.
Similarly to the Galois group functor $\infgal (L/k)$ in the differential and the difference cases, 
we may introduce the group functor $\cinfgal (L/k)$ called commutative 
Galois group functor, on the category $(CAlg/L\n)$. 
 \begin{definition}
 In the differential case and in the difference case, the Galois group in our Galois theory is the group functor 
$$
\cinfgal (L/k) \colon (CAlg/L\n ) \to (Grp)
$$
 defined 
 by 
 \begin{multline*}
 \cinfgal (L/k) (A) =
 \{
 \, 
 f \colon \eL \hat{\otimes} \sb {L\s} A[[W]] \to 
 \eL \hat{\otimes} \sb {L\s} A[[W]]
 \, | \, \\
 f 
 \text{\it \/ 
 is a } 
 \K \otimes \sb {L\s}A[[W]] 
 \text{-automorphism} 
 \text{\it \/ compatible with 
 $\mathcal{D}$, } \\
\text{\it \/ continuous with respect to
 the $W$\!-adic topology} \\
 \text{\it \/ and congruent to the identity modulo nilpotent elements 
 }
 \}
 \end{multline*}
 for a commutative $L\n$-algebra $A$. See Definition 2.19 in \cite{mori09}. 
\end{definition}
\par
Then the group functor $\cinfgal (L/k)$ would operates on the functor 
$\mathcal{CF}\sb {L/k}$ in such a way that 
the operation $(\cinfgal (L/k), \mathcal{F}\sb {L/k} )$ is a 
torsor. 
\begin{remark}
For a \QSI field extension $L/k$, the Galois hull 
$\eL/ \K$ is, in general, a non-commutative
algebra extension so that the commutative Galois group functor 
$\cinfgal (L/k)$
 on the category $(CAlg/L\n)$ is not adequate for the following two reasons.
\begin{enumerate}
\renewcommand{\labelenumi}{(\arabic{enumi})}
%\arabic{en\renewcommand{\labelenumi}{(umi})}
\item 
If we measure the extension $\eL /\K$ over the category $(CAlg /L\n)$
by the commutative Galois group functor $\cinfgal (L/k)$, 
the non-commutative data of the extension $\eL /\K$ are lost. 
\item
We hope to get a quantum group as a Galois group. A quantum group is, however, in any sense not a group functor on the category $(NCAlg/L\n)$ of non-commutative $L\n$-algebras.
\end{enumerate}
\end{remark}
In the three coming Sections, we settle these points for three concrete Examples. Looking at these Examples, we are led to a general Definition in Section \ref{14.5.29a}. 
The idea is to look at the coordinate transformations of initial conditions. 
As it is easier to understand it with examples, we explain the definition there. See Questions \ref{14.5.31a}. 
\par 
%%%%%%%%%%%
%%%%%%%
%%%%%%%%%%
%%%%%%%%%%%%%%%%%%%45n
 %%%%%%%%%%%%%%%%%%%%%%%%%%%%%%%%%%%%%%%%%%%%%%%%%%%%%%%%%%%%%%%%%%%%%%%%%%%%%%%%%%%%%%%%%%%%
\section{The First Example, the field extension $\com(t)/\com$}
\label{10.4a}
From now on, we assume $C =\com$. The arguments below work for an algebraically closed field $C$ of characteristic $0$. So $q$ is a non-zero complex number. 
\subsection{Analysis of the example}\label{10.10a}
Let $t$ be a variable over $\com$. The field $\com(t) $ of rational functions has various structures: 
the differential field structure, the $q$-difference field structure and the \QSI field structure that we are going to define. 
We are interested in the Galois group of the field extension $\com(t)/\com$ with respect to these structures. 
Let $\sigma \colon \com(t) \rightarrow \com(t)$ be the $\com$-automorphism of the rational function field $\com(t)$ sending $t$ to $qt$. 
So $(\com(t), \,\sigma)$ is a difference field. We assume $q^{n}\neq 1 $ fore every positive integer $n$. We define a $\com$-linear map 
$\theta^{(1)}\colon \com(t) \rightarrow \com(t)$ by 
\[
 \theta^{(1)}(f(t)) := \frac{\sigma(f)-f}{\sigma(t) - t} = \frac{f(qt) -f(t)}{(q-1)t} \qquad \text{\it \/ for }f(t)\in \com(t).
\]
For an integer $n \geq 2$, we set
\[
 \theta^{(n)}:= \frac{1}{[n]\sb{q}!}\left(\theta^{(1)}\right)^{n}.
\]
It is convenient to define 
\[
 \theta^{(0)} = \Id\sb{\com(t)}. 
\]
It is well-known and easy to check that $(\com(t),\, \sigma, \, \theta^{*}) = (\com(t), \, \sigma, \, \{\theta^{(i)}\}\sb{i \in \N})$ is a \QSI algebra. 
\par 
We have to clarify a notation.
For an algebra $R$, a sub-algebra $S$ of $R$ and a sub-set $T$ of $R$, we denote by 
$S \langle T \rangle\sb{alg}$ the sub-algebra of $R$ generated over $S$ by $T$.
\begin{lemma}\label{8.30a}
The difference field extension $(\com(t), \, \sigma)/ (\com, \, \Id\sb{\com})$ is a Picard-Vessiot extension. Its Galois group is the multiplicative group $G\sb{m\, \com}$. 
%The Galois group of the difference Picard-Vessiot field extension $(\com(t), \, \sigma)/ (\com, \, \Id\sb{\com})$ is the multiplicative group $\G\sb{m\,\com}$. 
\end{lemma}
\begin{proof}
 Since $t$ satisfies the linear difference equation $\sigma(t) = qt$ over $\com$ and the field $C\sb{\com(t)}$ of constant of $\com(t)$ is $\com$, the extension
 $(\com(t),\, \sigma )/ (\com, \, \Id\sb{\com})$ is a difference Picard-Vessiot extension. 
The result follows from the definition of the Galois group. \end{proof} 
When $q\rightarrow 1$, the limit of the \QSI ring $(\com(t), \, \sigma, \, \theta^{*})$ is the differential algebra $(\com(t), \, d/dt)$. 
We denote by $AF\sb {1\, k}$, the algebraic group of affine transformations of the affine line so that 
\[
 AF\sb{1\, \com } = \left\{ \left.
 \begin{bmatrix}
 a & b \\
 0 & 1
 \end{bmatrix} \right| 
 a, \, b \in \com ,\, a \neq 0
 \right\} . 
\] Then
\[
 AF\sb{1\, \com }\simeq \G\sb{m\, \com} \ltimes \G\sb{a\,\com}, 
\]
where 
\begin{align*}
 \G\sb{m\,\com} &\simeq \left\{ \left.
 \begin{bmatrix}
 a & 0\\
 0 & 1
 \end{bmatrix} \in AF\sb{1\, \com } \right| a\in \com ^{*} \right\}, \vspace{1ex}\\
 \G\sb{a\,\com} &\simeq \left\{ \left.
 \begin{bmatrix}
 1 & b\\
 0 & 1
 \end{bmatrix} \in AF\sb{1\,\com} \right| b\in \com \right\}.
\end{align*}
\begin{lemma}\label{8.30b}
The Galois group of differential Picard-Vessiot 
extension $(\com(t), \, d/dt)/\com$ is $\G\sb{a\,\com}$.
\end{lemma}
\begin{proof}
We consider the linear differential equation
\begin{equation}\label{8.10a}
Y^\prime = \left[ \begin{array}{cc}
0 & 1 \\
0 & 0
\end{array} \right] Y,
\end{equation}
where $Y$ is a $2 \times 2$-matrix with entries in a differential extension field of $\com$. 
Then $\com(t)/\com $ is the Picard-Vessiot extension for \eqref{8.10a}, 
$$
Y = \left[ \begin{array}{cc}
1 & t \\
0 & 1
\end{array} \right]
$$
being a fundamental solution of \eqref{8.10a}. 
The result is well-known and follows from, the definition of Galois group. 
\end{proof}
The \QSI field extension $(\com(t), \, \sigma, \, \theta^{*})/\com$ is not a Picard-Vessiot extension 
in the sense of Hardouin \cite{har10} and Masuoka and Yanagawa \cite{masy} 
so that we can not treat it in the framework of Picard-Vessiot theory. We can apply, however, Hopf Galois theory of Heiderich \cite{hei10}. 
\begin{proposition}\label{5.14a}
The commutative Galois group 
$\cinfgal ((\com (t),\, \sigma,\, \theta^{*})/\com )$
 of the extension 
$(\com(t),\,\sigma, \, \theta^{*})/\com$
 is isomorphic to the formal completion $\hat{\G}\sb{m\,\com}$ of the multiplicative group $\G\sb{m\,\com}$. 
\end{proposition}
Before we start the proof, we explain the behavior of the Galois group under specializations. 
Theory of Umemura \cite{ume96.2} and Heiderich\cite{hei10} single out only the Lie algebra. 
Proposition \ref{5.14a} should be understood in the following manner. 
We have two specializations of the \QSI field extension $(\com(t),\, \sigma, \, \theta^{\ast})/\com$. 
\begin{enumerate}
\renewcommand{\labelenumi}{(\roman{enumi})}
\item $q \rightarrow 1$ giving the differential field extension $(\com(t), \, d/dt)/\com$. 
See \ref{9.19h}. 
\item Forgetting $\theta ^*$, 
or equivalently specializing 
$$
\theta ^{(i)} \to 0 \qquad \text{\it \/ for } 
i \ge 1,
$$
 we get the difference field extension $(\com(t),\, \sigma)/\com$. See \ref{9.19c}.
\end{enumerate}
We can summarize the behavior of the 
Galois group under the specializations.
\begin{enumerate}
\renewcommand{\labelenumi}{(\arabic{enumi})}
\item Proposition \ref{5.14a} says that the commutative Galois group
$${\rm C}\infgal\sb{\sigma\, \theta^*}(L/k)$$
 of $(\com(t),\, \sigma, \, \theta^{*})/\com$ is the formal completion of the multiplicative group $\G\sb{m\,\com}$. This describes the 
Galois group at the generic point. 
\item By Lemma \ref{8.30a}, the Galois group of the specialization (i) is 
the formal completion % $\hat{\G}\sb{m\,\com}$ of $\hat{\G}\sb{m\,\com}$.
of 
the multiplicative group. 
\item The Galois group of the specialization (ii) is the additive group $\G\sb{a\,\com}$ by Lemma \ref{8.30b}.
\end{enumerate}
\begin{proof}[Proof of Proposition \ref{5.14a}]
Let us set $L = (\com(t),\, \sigma, \, \theta^{*})$ and $k = (\com, \, \sigma, \, \theta^{*})$. 
By definition of the universal Hopf morphism 
\eqref{9.19a}, 
\[
 \iota \colon (L, \, \sigma , \, \theta^* ) 
 \rightarrow 
 \left( F(\Z,L\n)[[X]],\, \hat{\Sigma}, \, \hat{\Theta}^* \right), \qquad \iota(t) = tQ + X \in F(\Z,L\n)[[X]], 
 \]
where 
\[
 Q \in F(\Z, L\n)
\]
is a function on $\Z$ taking values in $\com \subset L\n$ such that 
\[
 Q(n) = q^{n} \qquad \text{\it \/ for } n \in \Z.
\]
We denote the function $Q$ by the matrix
\[
 Q = \begin{bmatrix}
  \cdots & -2 & -1 & 0 & 1 & 2 & \cdots \\
  \cdots & q^{-2} & q^{-1} & 1 & q & q^{2} & \cdots
  \end{bmatrix}
\]
according to the convention. 
We take the derivation $\partial/\partial t \in \text{Der}(L\n/k\n)$ as a basis of the 1-dimensional $L\n$-vector space $\text{Der}(L\n/k\n)$ of $k\n$-derivations of $L\n$.
So $(\partial/\partial t) (\iota(t)) = Q$ is an element of the Galois hull $\eL$. Therefore 
\[
 \eL \supset \eL ^o := L\s\langle X,\, Q \rangle\sb{alg}, 
\]
which is the $L\s$-sub-algebra of 
$F(\Z, \, L\s )[[X]]$ generated by $X$ and $Q$. So the algebra $\eL^o$ is invariant under $\hat{\Sigma} , \, \hat{\Theta} ^*$ and $ \partial /\partial t$. 
Since $QX = q XQ$, the Galois hull $\eL$ is a non-commutative $L\n$-algebra. 
Now we consider the universal Taylor expansion 
\[
 (L\n, \partial/\partial t) \rightarrow L\n[[W]]
\]
and consequently we get the canonical morphism 
\begin{equation}\label{10.3d}
 \iota \colon \eL \rightarrow F(\Z, L\s)[[X]] \rightarrow F(\Z,L\n[[W]])[[X]]. 
\end{equation}
We study infinitesimal deformations of $\iota$ in \eqref{10.3d} 
over the category $(CAlg/L\n )$ of commutative $L\n$-algebras. 
Let A be a commutative $L\n$-algebra and 
\[
 \varphi \colon \eL \rightarrow F(\Z,\, A[[W]])[[X]]
\]
be an infinitesimal deformation of the canonical morphism 
$$
\iota :\eL \to F(\Z , \, A[[W]])[[X]]. 
$$
%for a commutative $L\n$-algebra $A$ so that 
%$$
%\varphi \in \mathcal{F} \sb {L/k}(A).
%$$

\begin{sublemma}\label{5.15} %Let $A$ be a commutative algebra in $Ob(CAlg/L\n$. 
We keep the notation above. 
\begin{enumerate}
\renewcommand{\labelenumi}{(\arabic{enumi})} 
\item There exists a nilpotent element $n \in A$ such that $\varphi(Q) = (1+ n)Q$ and $\varphi(X) = X$. 
\item 
The commutative infinitesimal deformation $\varphi$ is determined by the nilpotent element $n$ such that 
$\varphi (Q) = (1+n)Q$. 
\item
 Conversely, for every nilpotent element $n \in A$, there exists a unique commutative infinitesimal deformation $\varphi \sb e \in \mathcal{F}\sb {L/k}(A)$ such that 
 %$%\varphi \sb e (X) =X$, %and 
$\varphi \sb e (Q) = eQ$, 
where we set $e= 1+n $. 
\end{enumerate}
\end{sublemma}
Sublemma proves Proposition\ref{5.14a}. 
\end{proof}
\begin{proof}[Proof of Sublemma.] The elements 
$X, Q \in \eL$ satisfy the following equation. 
\begin{align*}
 &\frac{\partial X}{\partial W} = \frac{\partial Q}{\partial W} = 0,\\
 &\hat{\Sigma}(X) = qX,\qquad\hat{\Sigma}(Q) = qQ, \\
 &\hat{\Theta}^{(1)}(X) = 1, \qquad \hat{\Theta}^{(i)}(X) = 0 \qquad \text{\it \/ for } i\geq 2,\\
 &\hat{\Theta}^{(i)}(Q) = 0 \,\,\text{\it \/ for } i \geq 1. 
\end{align*}
So $\varphi(X), \, \varphi(Q)$ satisfy the same equations as above, which shows %[us $X$ and $Q$ ******* with respect to these operators. ]%where?
\begin{align*}
\varphi(X) &= X + fQ \in F(\Z,A[[W]])[[X]],\\
\varphi(Q) &= eQ \in F(\Z,A[[W]])[[X]],\\ 
\end{align*}
where $f,e \in A$. Since $\varphi$ is an infinitesimal deformation of $\iota$, $f$ and $e-1$ are nilpotent elements in $A$. 
We show the first $f = 0$. In fact, it follows from the equation
\[
 QX = qXQ
\]
that 
\[
 \varphi(Q)\varphi(X) = q\varphi(X)\varphi(Q)
\]
or
\[
 eQ(X + fQ) = q(X + fQ)eQ.
\]
So we have 
\[
 eQfQ = qfQeQ
\]
and so 
\[
 efQ^{2} = qfeQ^{2}.
\]
Therefore
\[
 ef = qfe.
\]
Since $e$ is a unit, $e-1$ being nilpotent in $A$,
\[
f - qf = 0,
\]
so that 
\[
(1-q)f = 0.
\]
As $1-q$ is a non-zero complex number, $f = 0$. So we proved (1). In other words, we determined the restriction of $\varphi$ to the sub-algebra $\eL ^o = L\s \langle X, \, Q\rangle\sb{alg}\subset \eL$. To prove (2), we have to show that $\varphi$ is determined by its restriction on $\eL ^o$. To this end, we take two commutative infinitesimal deformations 
$\varphi , \, \psi \in \mathcal{F}\sb {L/k}(A)$ such that 
$$
\varphi (Q) = eQ \text{ and } \psi (Q) = eQ,
 $$
where $n$ is a nilpotent element in $A$ and we set $e= 1+n$. Since 
$$
\eL = L\s . \iota(\com (t))\langle X, \, Q\rangle\sb{alg} =
L\s \langle X,\, Q, \, \iota ((t + c)^{-1})\rangle\sb { c\in \com \, alg},
$$
and since $\varphi$ is a $\K=L\s$-morphism, 
it is sufficient to show that
$$
\varphi (( t + c )^{-1}) = \psi (( t + c )^{-1})
$$
for every complex number $c\in \com$. 
Since $\iota (t +c) \in \eL ^o$, 
$\varphi (t+c)= \psi (t+ c) $ and so 
$$
\varphi (( t + c )^{-1}) =\varphi ( t + c )^{-1} = \psi ( t + c )^{-1}
= \psi ( (t + c )^{-1}).
$$
This is what we had to show. 
\par 
Now we prove (3).
We introduce another sub-algebra 
$$
\tilde{\eL}:=\{ 
\sum \sb {n=0}^\infty X^n a\sb n \in F(\Z , \, L\s)[[X]] \, | \, a\sb n \in L\s (Q) \text{ for every $n \in \N$}
\}  
$$ 
so that, by commutation relation \eqref{a11.1}, 
 $\tilde{\eL}$ is a sub-algebra of 
$
F(\Z , \, L\s)[[X]]
$
invariant under $\hat{\Sigma}, \, \hat{\Theta}^{*}$ and the derivation $\partial /\partial t$. 
We show $\eL \subset \tilde{\eL}$. Since the sub-algebra $\eL$ is generated by $\iota (L)$ and $L\s$ along with operators 
$\hat{\Sigma}, \, \hat{\Theta}^{*}$ and $\partial /\partial t$. So it is sufficient to notice 
$L\s $ and $\iota (L)$ are sub-algebras of $\tilde{\eL}$. 
The first inclusion $L\s\subset \tilde{\eL}$ being trivial, it remains to show the second inclusion: 
$$\iota(L)= \iota (\com (t)) \subset \tilde{\eL}.$$
 We have to show 
that (i) $\iota (t) \in \tilde{\eL}$, and (ii) $\iota (t+ c) ^{-1} \in \tilde{\eL}$ for every complex number $c \in \com$. The first assertion (i) follows from 
the equality $\iota (t) = tQ + X$. As for the assertion (ii), we notice 
\begin{align}
\iota ((t+c)^{-1})&= \iota (t+c) ^{-1} \notag \\
     &=(tQ+X + c)^{-1} \notag \\
     &= (tQ + c)^{-1}( 1 + (tQ +c)^{-1}X)^{-1} \notag \\
     &= (tQ + c)^{-1}(1- A)^{-1} \notag \\
     &= (tQ + c)^{-1}\sum\sb {n=0} ^\infty A^n,\label{146.6a}
\end{align}
where we set $A=-(tQ +c)^{-1}X$.
Upon writing $a(Q):=-(tQ +c)^{-1}$, we have 
$$
A = Xa(qQ), \ A^2 = X^2a(qQ)a(q^2Q),\ \cdots , \ 
A^n = X^n\prod \sb{i=1}^{n}a(q^iQ),\ \cdots 
$$
by commutation relation \eqref{a11.1}. 
%$$
%A^n = (-1)^ nX^n (tq^nQ^n +c).
%$$
Hence, by \eqref{146.6a}, $\iota (t+c)^{-1} \in \tilde{\eL}$. 
Thus we proved the inclusion 
$
\eL ^o \subset \tilde{\eL}.
$
 \par 
 To complete the proof of (3), 
a nilpotent element $n$ of the
 algebra $A$ being given, 
we set $e= 1+n$. 
As we have $qXeQ =eQX$, 
by the commutation relation \eqref{a11.1}, there exists an infinitesimal deformation 
$$
\psi \sb e : \tilde{\eL} \to F(\Z , \, A[[W]])[[X]]
$$ 
such that $\psi\sb e (X) = X$ and $\psi \sb e (Q) = eQ$ and continuous with respect to the $X$-adic topology. 
Therefore to be more concrete $\psi \sb e$ maps an element 
of the algebra $\tilde {\eL}$ 
$$
%\psi \sb e
\sum \sb{n=0}^\infty X^n a\sb n (Q) \text{ {\it with } $a\sb n(Q) \in L\s(Q)$ {\it for every } $n\in \N$ }
$$
%of $\tilde{\eL}$
to
 an element 
$$
\sum \sb{n=0}^\infty X^n a\sb n (eQ) \in F(\Z , \, A[[W]])[[X]].
$$ 
If we denote the restriction $\psi \sb e \, |\sb{\eL} $ to 
$\eL$ by $\varphi \sb e$, then $\varphi \sb e$ satisfies all the required conditions except for the uniqueness. The uniqueness follows from (2) that we have already proved above. 
%%%%%%%%%%%%%%%%%%%%%%%%%%%%%%%%%%%%%%%%%%%%%%%%
%we consider the infinitesimal deformation 
%$$
%\varphi \sb e :\eL^o \to F( \N , \, A[[w]])[[X]]
%$$
 %of $\iota | \eL ^o$ 
%sending $X$ to $X$ and $Q$ to $eQ$. We show that 
%$\varphi\sb e$ extends uniquely to an 
% infinitesimal deformation $ \varphi :\eL \to F(\N \, A[[W]])[[X]]$.
%First, we notice that we can uniquely extend $\varphi \sb n $ to 
%$L\n (\iota (t))<X, \, Q>\sb{alg}$. 
%To this end, it is sufficient
% to 
%notice that if it is extended to an infinitesimal deformation
%$$
%\varphi :L\n (\iota (t))<X, \, Q>\sb{alg} \to F(\N , \, A[[W]])[[X]],
%$$ 
%then we have, for every $c\in \com$, 
 %\begin{align*}
%\varphi ((t - c)^{-1}) &= (\varphi (t)-c )^{-1} \\
%&= (e(t + W)Q + X -c )^{-1} 
%\end{align*}
%that is a well-determined element of $F(\N , A[[W]])[[X]]$. 
%This shows the extension over 
%the sub-algebra $L\n (\iota (t) <Q, X>\sb{alg}$ 
%uniquely exists.
%Now, for 
%every element 
%$$h\in L\n \iota ( (t)) <X,\, Q>\sb{alg},$$ since 
%\begin{align*}
%\varphi (
%\%frac{
 %\partial h
 %}
 %{
 %\partial W 
 %}
% ) &= 
%  \frac{\partial \varphi (h)
 % }
%  {\partial W }, \\
%\varphi (
%\hat{\Theta}^{(i)}(h) 
%)
%&=
% \hat{\Theta}^{(i}(\varphi (h)) \text{ for every $i \in %\N$}, \\
%\varphi (\hat{\Sigma}(h) &=  
 % \hat{\Sigma}(\varphi (h)).
 % \end{align*}
% Hence once $\varphi$ is determined on $\L\n <Q, X>\sb%{alg}$, 
 % it is determined al over $\eL$. 
 \end{proof}
%%%%%%%%%%%%%%%%%%%%%%%%%%%%%%%%%%%%%%%%
We have shown that the functor
\[
\mathcal{F}\sb{L/k} \colon (Alg/L\n) \rightarrow (Set)
\]
is a torsor of the group functor $\hat{G}\sb{m\,\com}$. 
 For origin of the group structure, see Paragraph \ref{9.28a} as well as Paragraph \ref{9.28e} below. 
%\begin{definition}
%For a commutative algebra $R$, we denote by $(NCAlg/R)$ the category of not necessarily commutative $R$-algebras $A$ such that the image of $R$ is in the center of $A$. %for an algebra R. 
%To avoid confusions, we denote the category $(Alg/R)$ of commutative $R$-algebras by $(CAlg/R)$ and the infinitesimal deformation functor $\mathcal{F}\sb{L/k}$ by $C\mathcal{F}\sb{L/k}$. 
%\end{definition}
%%%%%%%%%%%%%%%%%%%
%We investigate non-commutative infinitesimal deformations of the Galois hull $\eL/\K$. So we define a functor 
%\[
 %\NCF \colon (NCAlg) \rightarrow (Set)
%\]
%by 
%\begin{align*}
% \NCF (A) = \left\{ \! \begin{array}{cc}
%\varphi \colon \eL \rightarrow F(\N, A[[W]])[[X]] & 
%\left| \begin{array}{l} 
%\text{\it \/ \it $\varphi$ is a \QSI }\\
%\text{\it \/ \it algebra morphism compatible with} \\ 
%\text{\it the derivation $\partial/\partial W$ and $\varphi|\sb{\K} = \iota|\sb{\K}$.}
%\end{array} \right. 
%\end{array} \! \right\}.
%\end{align*}
\par
In the course of the proof of Proposition \ref{5.14a}, we
have proved the following 
\begin{proposition}\label{10.17a}
The Galois hull $\eL$ coincides with the sub-algebra 
$$
L\n\langle X, \, Q, \, (c+ tQ + X)^{-1}\rangle\sb {c\in \com\, alg}
$$
of $F(\Z, \, L\s)[[X]]$ generated by $L\s , \, X ,\, Q$ 
and the set $\{ (c + tQ + X)^{-1}\, |\, c\in \com \}$. The commutation relation of $X$ and $Q$ is 
$$
QX=qXQ. 
$$ 
In particular, if $q\not= 1$, then the Galois hull is non-commutative.
\end{proposition}

\subsection{Non-commutative deformation functor $\mathcal{NCF}\sb {\sigma \, \theta^*\,L/k}$
for $L/k=\com (t) /\com$} 
We are ready to describe the non-commutative deformations. 
Let $A\in ob(NCAlg/L\n)$. 
\begin{lemma}\label{8.30c}
If $q \neq 1$, we have 
\[
 \NCF (A)= \left\{ \, 
( e, \, f) \in A ^2 \, | \, qfe=ef \text{\it \/ and } e-1,\, f \text{\it \/ are nilpotent \,}
 \right\}
 \]
 for every $A\in ob (NCAlg/L\n)$.
\end{lemma}
\begin{proof}
Since $q \not= 1$, 
it follows from the argument of the proof of Sublemma \ref{5.15} that if we take 
$$
\varphi \in 
\mathcal{NCF}\sb {L/k} (A) \text{\it \/ for } A \in
ob \, \NCA , 
$$ 
then $\varphi(X) = X + fQ$ and $\varphi(Q) = eQ$ for $f,e \in A$. \par
Since $\varphi$ is an infinitesimal deformation of $\iota$, $f$ and $e-1$ are nilpotent. \par
It follows from $QX = qXQ$ that 
\[
eQ(X+fQ) = q(X + fQ)eQ
\]
so $ef = qfe$. 
\par 
Suppose conversely that elements $e,\, f \in A$ such
that $e-1, \, f$ are nilpotent and such that 
$ef =qfe$ are given. Then the argument of the proof of Sublemma \ref{5.15} allows us to show the unique existence of 
the infinitesimal deformation $\varphi \in \mathcal{NCF}\sb{L/k}(A)$ such that 
$$
\varphi (X)= X + fQ, \, \varphi (Q) = eQ. 
$$ 
\end{proof}
We are going to see in \ref{9.28e} that
theoretically, 
 we can identify 
%the set $\mathcal{NCF}\sb {L/k}$ with the set 
\begin{equation}
 \NCF(A) = \left\{ \begin{array}{cc}
 \left. \begin{bmatrix}
 e & f\\
 0 & 1
 \end{bmatrix} \right| 
 e ,\ f \in A,\ qfe=ef \text{\it \/ and } e-1, f \text{\it \/ are nilpotent}
 \end{array}
 \right\} . 
\end{equation}
%%%%%%%%%%%%%%%%%%%%%%%%%%%%%%%%%%%%%%%%%%%%%%%%%%%%%
\begin{cor}
[Corollary to the proof of Lemma \ref{8.30c}]
When $q = 1$ that is the case excluded in our general study, we consider the \QSI field 
$$
(\com(t),\, \Id , \, \theta^{\ast})
$$ 
as in \ref{9.19b}. So $\theta^{\ast}$ is the iterative derivation; 
\begin{align*}
\theta^{(0)} &= \Id, \\
 \theta^{(i)} &= \dfrac{1}{i!}\dfrac{d^i}{dt^i}\qquad \text{\it \/ for } i\geq 1 .
\end{align*}
%%%%%%%%%%%%%%%%%%
Then we have 
\begin{align}
 & \eL\sb{\Id \, \theta ^*} \simeq \eL\sb{d/dt}, 
 \label{9.26a} \\ 
& \mathcal{NCF}
\sb{\Id \, \theta^{\ast}\, \com(t)/\com
}(A)
= \{ f \in A \, | \, f \text{\it \/ is a nilpotent element }
\} \label{9.26b}
\end{align}
 for $ A \in ob \NCA$. 
\end{cor}
\begin{proof}
In fact, if $q = 1$, then 
\begin{equation*}
Q = 
\left[\begin{array}{ccccccc}
\cdots & -2 & -1 & 0 & 1 & 2 & \cdots \\
\cdots & 1 & 1 & 1 & 1 & 1 & \cdots
\end{array} \right] = 1 \in \com.
\end{equation*}
So $\eL\sb{\Id \, \theta^{\ast}}$ is generated by $X$ over $\K$. Therefore $\eL\sb{\Id \, \theta^*} \simeq \eL\sb{d/dt}$. 
Since $Q = 1 \in \K, \, \varphi(Q) = Q$ for an infinitesimal deformation 
 $$
 \varphi \in \mathcal{NCF}\sb{\Id \, \theta^{\ast}\, \com(t)/\com}(A)
 $$ 
 and we get \eqref{9.26b}. 
\end{proof}
\subsubsection{Quantum group enters}\label{9.28b}
To understand Lemma \ref{8.30c}, it is convenient to introduce a quantum group. 
\begin{definition}
We work in the category $(NCAlg/\mathbb{C})$. Let $A$ be a not necessarily commutative $\mathbb{C}$-algebra. 
We say that two sub-sets $S,\, T$ of $A$ are mutually commutative if for every $s\in S,\, 
 t \in T$, we have $[s,t]=st-ts = 0$.
\end{definition}
For $A \in ob \NCA$, we set
\[
H\sb{q}(A) =\left\{ \left.
 \begin{bmatrix}
e & f \\
0 & 1 
\end{bmatrix} \, \right| \, %\begin{array}{l}
e,\, f \in A, \text{\it \/ $e$ is invertible in $A$, } ef=qfe\right\}
. \]
\begin{lemma} For two matrices
\[
Z\sb{1} = \begin{bmatrix}
e\sb{1} &f\sb{1} \\
 0 & 1 
\end{bmatrix},\quad 
Z\sb{2} = \begin{bmatrix}
e\sb{2} & f\sb{2}\\
0  & 1 
\end{bmatrix} \in H\sb{q}(A), 
\]
if $\{e\sb{1},\, f\sb{1}\}$ and $\{e\sb{2},\, 
f\sb{2}\}$ are mutually commutative, then the product matrix 
\[
Z\sb{1} Z\sb{2} \in H\sb{q}(A).
\]
\end{lemma}
\begin{proof}
Since 
\[
Z\sb{1} Z\sb{2} = \begin{bmatrix}
e\sb{1}e\sb{2} & e\sb{1}f\sb{2} + f\sb{1}\\
 0 & 1
\end{bmatrix}, 
\]
we have to prove 
\[
e\sb{1}e\sb{2}(e\sb{1}f\sb{2}+f\sb{1}) = q(
e\sb{1}f\sb{2} + f\sb{1})e\sb{1}e\sb{2}
. \]
This follows from the mutual commutativity of $\{e\sb{1},\, f\sb{1}\}$,and $\{e\sb{2},\, f\sb{2}\}$, 
and the conditions $e\sb{1}f\sb{1} = qf\sb{1}e\sb{1},\, e\sb{2}f\sb{2} = qf\sb{2}e\sb{2}$. 
\end{proof}
\begin{lemma}
For a matrix
\[
Z = \begin{bmatrix}
e & f\\
0 & 1
\end{bmatrix} \in H\sb{q}(A),
\]
if we set
\[
\tilde{Z} = \begin{bmatrix}
e^{-1} & -e^{-1}f \\
 0 & 1
\end{bmatrix} \in M\sb{2},
\]
then 
\[
\tilde{Z} \in H\sb{q^{-1}}(A) \text{\it \/ and } 
\tilde{Z}Z = Z\tilde{Z} = I\sb{2}. 
 \]
\end{lemma}
\begin{proof}
We can check it by a simple calculation. See also Remark \ref{140224a}, where the first assertion is proved.
\end{proof}
\begin{remark}\label{140224a}
If $q^{2}\neq 1$, for $f \neq 0$, $\tilde{Z} \not\in H\sb{q}(A)$. In fact, let us set
\[
\tilde{Z} = \begin{bmatrix}
\tilde{e} & \tilde{f}\\
0 & 1
\end{bmatrix}
\]
so that, 
%thanks to commutation relation \eqref{a11.11},
 $\tilde{e}=e^{-1}, \tilde{f} = -e^{-1}f$. 
Then $\tilde{e}\tilde{f} = e^{-1}(-e^{-1}f) = -e^{-2}f$ and $\tilde{f}\tilde{e}= -e^{-1}fe^{-1} = -qe^{-2}f$. 
So 
\begin{equation}\label{140224b}
\tilde{e}\tilde{f} = -e^{-2}f =q^{-1} \tilde{f}\tilde{e}
\end{equation}
showing 
$$
\tilde{Z}\in H\sb{q^{-1}}(A).
$$
%and hence if $q^{2} \neq 1$ and if $f\not= 0$,
Now we assume to the contrary that $\tilde{Z}\in H\sb q (A)$. We show that it would lead us to a contradiction. 
The assumption would imply that we have 
\begin{equation}\label{140224c}
\tilde{e}\tilde{f} = q \tilde{f}\tilde{e}.
\end{equation}
It follows from \eqref{140224b} and \eqref{140224c}
\begin{equation}
q\tilde{f}\tilde{e} = q^{-1} \tilde{f}\tilde{e}. 
\end{equation}
so that we would have 
\begin{equation}
q^2 \tilde{f}\tilde{e} = \tilde{f}\tilde{e}. 
\end{equation}
Since $\tilde{e}= e^{-1}$ is invertible in $A$,
$$
(q^2 -1)\tilde{f} = 0.
$$
As the algebra $A$ is a $\com$-vector space and $\tilde{f} \not= 0$, 
the complex number $q^2 -1 =0$ which is a contradiction. 
% This implies $q^2 = 1$ which contradicts the assumption on the number $q$. 
\end{remark}
%%%%%%%

%%%%%%%%%%%%%%%%127
\begin{lemma}
Let $u$ and $v$ be symbols over $\com$.
We have shown that we find a $\mathbb{C}$-Hopf algebra 
\begin{equation} \label{hopfa1}
\gH\sb q = \mathbb{C}\langle u,\, u^{-1}, \, v\rangle \sb{alg}/(uv -qvu)
\end{equation}
as an algebra so that
$$
uu^{-1} = u^{-1} u = 1 , \qquad u^{-1} v = q^{-1}vu^{-1}.
$$
Definition of the algebra $\gH\sb q$ determines the multiplication
\[
m \colon \gH \sb q\otimes\sb{\mathbb{C}}\gH\sb q \to \gH\sb q, 
\]
the unit 
\[
\eta \colon \mathbb{C} \to \gH\sb q, 
\]
that is the composition of natural
 morphisms $$\com \rightarrow \com \langle u, \, u^{-1}, \, v\rangle\sb{alg}$$ and $$\com\langle u, \, u^{-1}, \, v\rangle\sb{alg} \rightarrow \com\langle u, \, u^{-1}, \, v \rangle\sb{alg}/(uv-qvu)=\gH\sb q.$$ 
 The product of matrices gives 
the co-multiplication
\[
\Delta \colon \gH\sb q \to \gH\sb q\otimes\sb{\mathbb{C}}\gH\sb q, 
\] 
that is a $\com$-algebra morphism 
defined by
\[
\Delta (u)=u\otimes u, \qquad \Delta (u^{-1})=
u^{-1}\otimes u^{-1}, 
\qquad \Delta (v)=u\otimes v + v\otimes 1,
\]
for the generators $u,\, u^{-1}, \, v$ of the algebra ${\gH\sb q}$, 
the co-unit is a $\com$-algebra morphism 
\[
\epsilon \colon \gH\sb q \to \mathbb{C}, \qquad \epsilon(u)=\epsilon (u^{-1})=1, \, \epsilon(v)=0 
 \]
 for the generators $u, \, u^{-1}, \, v$ of the algebra 
${\gH\sb q}$.
The antipode 
$$
S\colon {\gH\sb q} \to {\gH\sb q}
$$
is a $\com$-anti-algebra morphism given by 
$$
S(u)=u^{-1}, \qquad S(u^{-1})= u, \qquad S(v) = -u^{-1}v. 
$$
\end{lemma}
Let us set
$$
\gH\sb {q\, L\n} := \gH\sb q\otimes \sb{\com} L\n
$$
so that $\gH\sb{q \, L\n}$ is an $L\n$-Hopf algebra. We 
notice that for an $L\n$-algebra $A$ 
$$
\begin{array}{rl}
\gH\sb{q\, L\n}(A) &:= \Hom\sb{L\n\text{\it -algebra}}
(\gH\sb{q \,L\n}, \, A) \\ 
&= 
 \left\{ \left. 
\begin{bmatrix}
\begin{array}{cc}
e & f \\
0 & 1
\end{array} 
\end{bmatrix} \, \right| \, e, f \in A, \, ef=q fe, \text{\it \/ $e$ is invertible}
 \right\}.
\end{array}
$$

\begin{remark}
 We know by general theory that 
 the antipode $S\colon H \to H$ that is a linear map 
 making a few diagrams commutative, is necessarily an anti-endomorphism of the algebra $H$ so
 that 
 $$
 S (ab) = S(b) S(a) 
 \text{\it\/ for all elements $a, \, b \in H$ and \/ }
S(1) =1.
 $$
\end{remark} 
See Manin \cite{man88}, Section 1,\, 2. 
\par 
The Hopf algebra $\gH\sb q$ is a $q$-deformation of the affine algebraic group $AF\sb {1\, \com}$ of affine transformations of the affine line. 
\par 
Anyhow, we notice that the quantum group 
appears in this very simple example showing that 
quantum groups are indispensable for a Galois theory of \QSI field extensions.

\subsection{Observations on the Galois structures 
of the field extension $\com (t)/ \com $}
\label{151007a}
\subsubsection{Where does quantum group structure come from?}\label{9.28e} 
%\label{10.5a}
%%%%%%%%%%%%%%%%%%%%%%%%%%%%%%%%%%%%%%%%%%%%%%%%%%%%%%%%%%%%%%%%%%%%%%%%%%%%%%%%%%%%%%%%%
Let us now examine that the group structure in \ref{9.28a} 
arising from the variation of initial conditions coincides with the 
quantum group structure defined in \ref{9.28b}.
\par 
To see this, we have to clearly understand 
the initial condition of 
a formal series 
$$
f(W, \, X) = \sum\sb {i= 0}^ \infty X^i a\sb i (W) \in 
F(\Z , \, A[[W]] )[[X]]
$$
so that the coefficients $a\sb i$'s, which are elements of 
$F(\Z ,\, A[[W]])$, 
are 
functions on $\Z$ taking values in the formal power series ring $A[[W]]$.
The initial condition of $f(W,\, X)$ is the value of 
the function $f(W, \, 0)=a\sb{0} (W) \in F(\Z, \, A)$ at $n=0$ which we may denote by 
$$
f(W, \, 0)|\sb {n=0}\in A[[W]]. 
$$
\par
As in Example \ref{14.5.29b}, we set 
$$
T(W, X):=\iota (t) = (t+W)Q + X \in F(\Z , L\n [[W]])[[X]].
$$
For $A \in ob(NCAlg/L\n)$, we take an 
infinitesimal deformation 
$
\varphi \in \NCF (A)
$
so that the morphism   
$
\varphi \colon \eL \to F(\Z,\, A[[W]])[[X]]
$
is determined by the image 
$$
\tilde{T}(W, X) := \varphi (t) \in F(\Z , A[[W]])[[X]]
$$ 
of 
$t \in L \subset \eL$, the \QSI field 
$L$ being a sub-algebra of $\eL$ 
by the universal Hopf morphism. 
It follows from Lemma \ref{8.30c} that there exist $e,\, f \in A$
such that $ef=qfe$, 
 the elements $e-1, \, f$ are nilpotent and such that 
\begin{equation}\label{9.28d}
\varphi (t) = \left( e(t + W) +f\right) Q +X.
\end{equation}
Therefore, 
$$
\tilde{T}(W, X) = T((t(e-1) +f) +eW, X).
$$
Since $T(W, X)$ and $\tilde{T}(W, X)$ satisfy 
$$
\hat{\Sigma} ({\mathrm T})=q{\mathrm T}\text { and } \hat{\Theta}^{(1)}({\mathrm T})= 1,
$$
their difference is measured at the initial conditions. 
The initial condition of $T(W, X)$ is $t + W$ and that of 
$\tilde{T}(W, X)$ is $et +f + W$. Namely, the infinitesimal 
deformation $\varphi$ arises from the coordinate transformation 
$$
t + W \mapsto et + f +eW
$$ 
or equivalently 
$$
W \mapsto t(e-1) +f +eW.
$$
%The composition of two coordinate transformations of this type coincides 
%with the product of two triangular matrices.

%The above equality \eqref{9.28d} says 
%that in the level of the initial condition, 
%the dynamical system 
%$ 
%y \mapsto \varphi (y)
%$
%looks 
%as
%\begin{equation}\label{9.28e}
%t \mapsto \text{\it \/ the initial condition of }\varphi (y) =
%et + f.
%\end{equation}
%The composition of two mutually commutative transformations 
%of the form \eqref{9.28e} is nothing but the multiplication of $2\times 2$ matrices. 
%Therefore the quantum group structure is the same as in 
%the group structure in 
%\ref{9.28a}.
%\par
%%%%%%%%%
We answer the question above in Observation \ref{obs2}. 
\subsubsection{Quantum Galois group $\mathrm {NC}\infgal
\sb {\sigma \, \theta^*} (\com (t)/\com )$}
The Hopf algebra $\gH_{q}$ in \ref{9.28b} defines a 
functor 
$$
\hat{\gH}\sb{q \, L\n}\colon (NCAlg/L\n )\to (Set)
$$
such that 
\begin{multline*}
\hat{\gH}\sb{q\, L\n}(A) =\{ \psi :\gH\sb q\otimes \sb\com L\n \to A \, |\, \psi \text{\it \/ is a $L\n$-algebra morphism }\\
\text{\it \/ such that }\psi (u) -1, \, 
\psi (v) \text{\it \/ are nilpotent} \} 
\end{multline*}
for every $A \in (NCAlg/L\n)$.
In other words $\hat{\gH}\sb{q\, L\n}$ is the formal completion of the quantum group $\gH\sb q \otimes \sb \com L\n=
\gH\sb{q\, L\n}$. 
We can summarize our results in the following form. 
\begin{theorem}\label{14.6.23a}
The quantum formal group $\hat{\gH}\sb{q\, L\n}$ operates on the functor $\NCF$ in such a way that there exists a functorial isomorphisms 
\begin{equation}\label{140220a}
\NCF \simeq \hat{\gH}\sb{q\, L\n}. 
\end{equation}
The restriction of the functor $\NCF$ on the sub-category $(CAlg/L\n)$ gives the functorial isomorphism 
\begin{equation*}
\NCF\, |_{(CAlg/L\n)} \simeq \hat{\mathbb{G}}\sb{m\, L\n}.
\end{equation*}
Or equivalently, 
\begin{enumerate}
\renewcommand{\labelenumi}{(\arabic{enumi})}
\item %%%%%
We have not only isomorphism \eqref{140220a} of functors on the category $(NCAlg/L\n )$ taking 
values in the category $(Set)$ of sets, but also we can identify, 
by this isomorphism, the co-product of the quantum formal group 
$\hat{\gH} \sb{q\, L\n}$ 
arising from the multiplication of triangular matrices in 
 \ref{9.28b} with composition of coordinate transformations of the initial condition in \ref{9.28e}. 
For these two reasons, we say that the quantum infinitesimal Galois group of the \QSI 
field extension $(\com (t),\, \sigma , \, \theta^{\ast} ) / \com$ is the quantum formal group $\hat{\gH}\sb {q\, L\n}$. 
Namely, 
$$
\ncinfgal ((\com (t), \sigma , \, \theta^*)/\com) \simeq \hat{\gH}\sb{q \,L\n }.
$$
 %%%%%%% 
%Thanks to \ref{9.28e} and \eqref{140220a}, we may say that $\hat{\gH}\sb {q\, L\n}$ is the quantum infinitesimal Galois group of the \QSI 
%field extension $(\com (t),\, \sigma , \, \theta^{\ast} ) / \com$. %
%In other words, 
%the quantum Galois group functor of the \QSI extension 
%$(\com (t),\, \sigma , \, \theta^{\ast} ) / \com$ on the category 
%$(NCAlg/L\n )$ 
% of not necessarily commutative $L\n$-algebra is isomorphic to the 
 %quantum formal group $\hat{\gH}\sb{q \, L\n}$. 
 %Through this isomorphism, the co-product of the quantum formal group arising from the multiplication of triangular matrices in 
 %\ref{9.28b} coincides with composition of coordinate transformations of the initial condition.
\item The commutative Galois group functor 
$\cinfgal\sb{\sigma \theta^*}(L/k)$ of the \QSI extension $(\com (t), \sigma , \theta^{\ast} ) / \com$ on the category $(Alg / L\n )$ of commutative $L\n$-algebras 
 is isomorphic to the formal group 
$\hat{\G}\sb{m}$. 
\end{enumerate}
\end{theorem}
\par 
The operation of quantum formal group requires a 
precision. 
\begin{remark}\label{10.24a}
We should be careful about the 
operation of quantum formal group. To be more precise, for $\varphi \in 
\mathcal{F}\sb{L/k}(A)$ and
$
 \psi \in \hat{\gH}\sb{q\, L\n}(A)
$ 
so that 
we have 
$$
\varphi (t) = (e(t + W) +f)Q +X \in F(\Z , \,A[[W]])[[X]]
$$
with $e, \, f\in A$ and 
 we imagine the
matrix 
$$
\left[ \begin{array}{cc}
\psi(u) & \psi(v) \\
 0 & 1 
\end{array} \right] \in M\sb 2 (A)
$$
corresponding to $\psi$. 
If 
the sub-sets of the algebra $A$, $\{ \psi(u), \, \psi(v) \}$ and $\{ e, \, f \}$ are mutually commutative,
the product 
$$
\psi \cdot \varphi = \varpi \in \mathcal{F}\sb{L/k}(A)
$$
is defined 
to be
$$
\varpi (t) = (\psi(u) e (t+ W)+ (\psi(u)f + \psi (v) )Q + X \in F(\Z , \, A[[W]])[[X]].
$$ 
\end{remark}
\subsubsection{Non-commutative Picard-Vessiot ring}
So far we analyzed the First Example, which is a non-linear \QSI equation, 
according to general principle of Hopf Galois theory. We finally arrived at Theorem \ref{14.6.23a} that shows a quantum formal group appears as a Galois group. 
Our experiences of dealing Picard-Vessiot theory in our general 
framework done in our previous works \cite{ume11}, \cite{ume96.2}, 
teach us that 
we discovered here 
a new phenomenon, a non-commutative Picard-Vessiot extension.
\par 
We work in the \QSI ring 
$(F(\Z, \com (t) )[[X]], \, \hat{\Sigma} , \hat{\Theta}^{*})$. 
We are delighted to assert that a non-commutative \QSI ring extension 
\begin{equation}\label{14.6.24c}
(\com \langle Q,\, Q^{-1}, \, X \rangle\sb{alg}, \, \hat{\Sigma} , \, \hat{\Theta}^* )
/\com )
\end{equation}
is a non-commutative Picard Vessiot ring with quantum Galois group 
$\gH \sb q$. 
%where we work in \QSI algebra $F(\N, \com (t) \n )[[X]]$.
We consider the fundamental system 
$$
Y := 
\begin{bmatrix}
\begin{array}{cc}
Q & X \\
0 & 1
\end{array} 
\end{bmatrix} \in M\sb 2 (\com \langle Q, \, Q^{-1}, \, X\rangle\sb {alg})
$$
so that the homogeneous linear \QSI equations 
is 
\begin{equation}\label{14.6.23b}
\hat{\Sigma} (Y) = 
\begin{bmatrix}
\begin{array}{cc}
q & 0 \\
0 & 1
\end{array} 
\end{bmatrix}Y, \qquad \hat{\Theta}^{(1)}(Y) =
\begin{bmatrix}
\begin{array}{cc}
0 & 1 \\
0 & 0
\end{array} 
\end{bmatrix}Y.
\end{equation}
In fact, we can check the first equation in \eqref{14.6.23b}:
$$
\hat{\Sigma} (Y) = 
\begin{bmatrix}
\begin{array}{cc}
\hat{\Sigma} (Q) & \hat{\Sigma} (X) \\
\hat{\Sigma} (0) & \hat{\Sigma} (1)
\end{array} 
\end{bmatrix}
=
\begin{bmatrix}
\begin{array}{cc}
qQ & qX \\
0 & 1
\end{array} 
\end{bmatrix}
=
\begin{bmatrix}
\begin{array}{cc}
q & 0 \\
0 & 1
\end{array} 
\end{bmatrix}
\begin{bmatrix}
\begin{array}{cc}
Q & X \\
0 & 1
\end{array} 
\end{bmatrix} 
=
\begin{bmatrix}
\begin{array}{cc}
q & 0 \\
0 & 1
\end{array} 
\end{bmatrix}Y.
$$
The second equality of \eqref{14.6.23b} is also checked easily. 
%%%%%%%%%%%%%
\par 
Leaving heuristic reasoning totally aside, 
we study the Picard-Vessiot extension 
\eqref{14.6.24c}
in detail in Sections \ref{14.6.25b} and \ref{141223b}.
\section{The Second Example, the \QSI field extension $(\com(t,\, t^{\alpha}) , \,\sigma,\, \theta^{\ast})/\com$ }\label{10.4b}
\subsection{Commutative deformations}\label{1027a}
As in the previous Section, let $t$ be a variable over $\com$ 
and we assume that the complex number $q$ is not a root of unity if we do not mention other assumptions on $q$. Sometimes we write the condition that 
$q$ is not a root of unity, simply to 
recall it. 
We work under the condition that $\alpha $ is an irrational complex number so that $t$ and $t^\alpha$ are algebraically independent over $\com$. 
Therefore the field $\com (t, \, t^\alpha )$
is isomorphic to the rational function field of two variables over $\com$. We denote 
by 
$\sigma$ the $\com$-automorphism of the field $\com (t, \, t^\alpha )$ such that 
$$
 \sigma (t) = qt \ \text{\it and }\ 
 \sigma (t^\alpha ) =q^\alpha t^\alpha .
$$
Let us set $\theta^{(0)} := \Id\sb{\com(t,t^{\alpha})}$, the map
\[
 \theta^{(1)}: = \frac{\sigma - \Id_{\com(t,\, t^{\alpha})}}{(q-1)t} \colon \com(t,\, t^{\alpha}) \rightarrow \com(t,\, t^{\alpha})
\]
and
\[
 \theta^{(n)} = \frac{1}{[n]\sb{q}!} \left( \theta^{(1)} \right)^{n}\qquad \text{ for}\qquad n = 2, 3, \cdots .
\]
So the $\theta^{(i)}$'s are $\com$-linear operators on $\com(t,\, t^{\alpha})$ and 
\[
 L \colon = (\com(t,\, t^{\alpha}), \,\sigma,\, \theta^{\ast})
\]
is a \QSI field. The restriction of $\sigma$ and $\theta^{\ast}$ to the subfield $\com$ are trivial. 
We denote the \QSI field extension $L/\com$ by $L/k$. We denote $t^{\alpha}$ by $y$ so that 
as we mentioned above, 
the abstract field $\com(t,t^{\alpha}) = \com(t,\, y)$ is isomorphic to the rational function field of $2$ variables over $\com$. We take the derivations $\partial/ \partial t$ and $\partial / \partial y$ as a basis of the $L\n$-vector space $\mathrm{Der}(L\n/k\n)$ of $k\n$-derivations of $L\n$. Hence $L\s= (L\n,\{\partial/\partial t, \partial/\partial y \})$ as in \cite{ume11}. \par
Let us list the fundamental equations.
%\begin{equation}
\begin{align}
&\sigma(t) = qt, \qquad \sigma(y) = q^{\alpha}y, \label{5.24a}\\
&\theta^{(1)}(t) = 1, \qquad 
      \theta^{(1)}(y) = [\alpha]\sb{q} \frac{y}{t}. \label{5.24d}
\end{align}
We explain below the notation $[\alpha ]\sb q$.
%\end{equation}
%$$ \cum $$
We are going to determine the Galois group 
$$\mathrm{NC}\infgal(L/k).$$
Before we start, we notice
that 
by Proposition \ref{10.17a}, the 
Galois hull of the extension $(\com (t), \, \sigma , \, \theta ^* )/\com $ is not a commutative algebra and 
since $ \com (t) $ is a sub-field of $\com ( t, \, t^\alpha )$, the Galois hull of the \QSI field extension 
$(\com (t, \, t^\alpha ), \, \sigma ,\,\theta ^*)/\com $ is not a commutative algebra either. 
Consequently the \QSI field extension 
$\com (t, \, t^\alpha )/\com $ is not a Picard-Vessiot extension (See \cite{har10}, \cite{masy}, \cite{ume11}).
So we have to go beyond the general theory of Heiderich \cite{hei10}, Umemura \cite{ume11} 
 for the definition of the Galois group ${\rm NC}\infgal(L/k)$. \par
It follows from general definition that the universal Hopf morphism 
\[
\iota \colon L \rightarrow F(\Z, L\n)[[X]]
\]
is given by 
\[
 \iota(a) = \sum\sb{n=0}^{\infty}X^{n}u[\theta^{(n)}(a)] \in F(\Z,L\n)[[X]]
\]
for $a \in L$. Here for $b \in L$, we denote by $u[b]$ the element
\[
u[b] = \left[ \begin{array}{ccccccc}
\cdots & -2 & -1 & 0 & 1 & 2 & \cdots\\
 \cdots & \sigma^{-2}(b) & \sigma^{-1}(b) & b & \sigma(b) & \sigma^{2}(b) & \cdots
\end{array} \right]
\in F(\Z,L\n)
. \]
It follows from the definition above of the universal Hopf morphism $\iota$, 
\[
 \iota(y) = \sum\sb{n=0}^{\infty}X^{n}{\binom{\alpha}{n}}\sb{q}t^{-n}Q^{\alpha - n}y, 
\]
where we use the following notations. For a complex number $\beta \in \alpha + \Z$ , 
\[
 [\beta]\sb{q} = \frac{q^{\beta} - 1}{q-1}
\]
and
\[
 \binom{\alpha}{n}\sb{q} = \frac{[\alpha]\sb{q}[\alpha -1]\sb{q} \cdots [\alpha -n + 1]\sb{q}}{[n]\sb{q}!}. 
\]
\[
Q = \left[ \begin{array}{ccccccc}
\cdots & -2 & -1 & 0 & 1 & 2 & \cdots\\
 \cdots & q^{-2} & q^{-1}&1 & q & q^{2} & \cdots
\end{array} \right]
%\qquad
 \text{ and }%\qquad 
Q^{\alpha} = \left[ \begin{array}{ccccccc}
\cdots & -2 & -1 & 0 & 1 & 2 & \cdots\\
 \cdots & q^{-2\alpha} & q^{-\alpha} & 1 & q^{\alpha} & q^{2\alpha} & \cdots
\end{array} \right]
. \]
We set
\[
 Y\sb{0} := \sum\sb{n=0}^{\infty}X^{n}\binom{\alpha}{n}\sb{q} t^{-n}Q^{\alpha -n}
\]
so that 
\begin{equation}\label{5.21b}
 \iota(y)= Y\sb{0}y \qquad \text{ in } F(\Z,L\n)[[X]]. 
\end{equation}
Considering $k\n$-derivations $\partial/\partial t, \partial /\partial y $ in $L\n$ and therefore in $F(\Z,L\n)$ or in $F(\Z,L\n)[[X]]$, 
we generate the Galois hull $\eL$ by $\iota(L)$ and $L\n$ so that $\eL \subset F(\Z,L\n)[[X]]$ is invariant under 
$\hat{\Sigma}$, 
the $\hat{\Theta}^{(i)}$'s and $\{ \partial/\partial t, \partial/ \partial y \}$. 
We may thus consider
\[
 \eL \hookrightarrow F(\Z, L\s)[[X]]
. \]
By the universal Taylor morphism 
\[
 L\s = (L\n, \{\partial / \partial t, \partial/ \partial y \}) \rightarrow L\s[[W\sb{1},W\sb{2}]]
, \]
we identify $\eL$ by the canonical morphism
\[
 \iota \colon \eL \rightarrow F(\Z, L\s)[[X]] \rightarrow F(\Z, L\n[[W\sb{1},W\sb{2}]])[[X]]. 
\]
We study first the infinitesimal deformations $\mathcal{CF}\sb{L/k}$ of $\iota$ on the category $(CAlg/L\n)$ of commutative $L\n$-algebras and then generalize the argument to the category $\NCA$ of not necessarily commutative $L\n$-algebras. \par
For a commutative $L\n$-algebra $A$, let $\varphi \colon \eL \to F(\Z,A[[W\sb{1},W\sb{2}]])[[X]]$ 
be an infinitesimal deformation of the canonical morphism $\iota \colon \eL \rightarrow F(\Z, L\n[[W\sb{1},W\sb{2}]])[[X]]$ 
so that both $\iota$ and $\varphi$ are compatible with operators $\{ \hat{\Sigma}, \hat{\Theta}^{\ast}, \partial /\partial W\sb{1},\partial / \partial W\sb{2} \}$. 
\begin{lemma}\label{402.3}
The infinitesimal deformation $\varphi$ is determined by the images $\varphi(Y\sb{0}),\, \varphi(Q)$ and $\varphi(X)$. 
\end{lemma}
\begin{proof}
The Galois hull $\eL/\K$ is generated over $\K= L\s$ by $\iota(t) = tQ + X$ and $\iota(y) = Y\sb{0}y$ with operators $\hat{\Theta}^{\ast}, \, \hat{\Sigma} $ and $\partial / \partial t, \, \partial /\partial y$ along with localizations. This proves the Lemma. 
\end{proof}
Let us set $Z\sb{0} := \varphi(Y\sb{0}) \in F(\Z, A[[W\sb{1},W\sb{2}]])[[X]]$ and expand it into a formal power series in $X$:
\[
 Z\sb{0} = \sum\sb{n = 0}^{\infty}X^{n}a\sb{n}, \,\, \text{\it \/ with \/} a\sb{n}\in F(\Z,A[[W\sb{1},W\sb{2}]]) 
\qquad \text{\it \/ for every \/ } n \in \N 
. \]
It follows from \eqref{5.24a} and \eqref{5.21b} 
\[
 \hat{\Sigma} (Z\sb{0}) = q^{\alpha}Z\sb{0}
\]
so that 
\begin{equation}\label{6.2a}
 \sum\sb{n=0}^{\infty}X^{n}q^{n}\hat{\Sigma} (a\sb{n}) = q^{\alpha} \sum\sb{n=0}^{\infty} X^{n}a\sb{n}. 
\end{equation}
Comparing the 
coefficient of the $X^{n}$'s in \eqref{6.2a} we get
\[
 \hat{\Sigma} (a\sb{n}) = q^{\alpha -n}a\sb{n} \,\, \text{\it \/ for } n \in \N.
\]
So $a\sb{n} = b\sb{n}Q^{\alpha - n}$ with $b\sb{n} \in A[[W\sb{1},W\sb{2}]]$ for $n \in \N$. Namely we have 
\begin{equation}\label{5.24g}
 Z\sb{0} = \sum\sb{n=0}^{\infty}X^{n}b\sb{n}Q^{\alpha -n} \qquad \text{\it \/ with}\qquad b\sb{n} \in A[[W\sb{1},W\sb{2}]] .
\end{equation}
It follows from \eqref{5.24d}, 
\[
 \sigma(y) - y = \theta^{(1)}(y)(q-1)t
\]
and so by \eqref{5.24a}
\[
(q^{\alpha} - 1)y = \theta^{(1)}(y)(q-1)t.
\]
Applying the canonical morphism $\iota$ and the deformation $\varphi$, we get 
\begin{equation}
(q^{\alpha}-1)Y\sb{0}= \hat{\Theta}^{(1)} (Y\sb{0})(q-1)(tQ+X) \label{5.24e}
\end{equation}
as well as by the argument of First Example, 
\begin{equation}
(q^{\alpha}-1)Z\sb{0}= 
\hat{\Theta}^{(1)} (Z\sb{0})(q-1)(teQ+X) \label{5.24f}
\end{equation}
where $e \in A$ is an invertible element congruent to $1$ modulo nilpotent elements. 

Substituting \eqref{5.24g} into \eqref{5.24f}, we get a recurrence relation among the $b\sb{m}$'s; 
\[
 b\sb{m+1} = \frac{[\alpha - m]\sb{q}}{[m+1]\sb{q}(e(t + W\sb{1}))}b\sb{m}. 
\]
Hence 
\begin{equation}\label{11.16a}
 b\sb{m} = \binom{\alpha}{m}\sb{q}(e(t+W\sb{1}))^{-m}b\sb{0} \qquad \text{\it \/ for every \/} m \in \N,
\end{equation}
where $b\sb{0} \in A[[W\sb{1}, W\sb{2}]]$ and every coefficient of the power series $b\sb{0}-1$ are nilpotent. \\
Since 
\[
 \frac{\partial Y\sb{0}}{\partial y} 
= \frac{\partial}{\partial W\sb 2}\left( \sum\sb{n=0}^{\infty} X^{n} \binom{\alpha}{n}\sb{q}(t+W\sb{1})^{-n} Q^{\alpha -n} \right) = 0,
\]
we must have 
$$
0 =\varphi \left( \frac{\partial Y\sb{0}}{\partial y}\right)
=\frac{\partial \varphi(Y\sb{0})}{\partial W\sb 2}=
\frac{\partial Z\sb{0}}{\partial W\sb 2}
$$
 and consequently 
\[
 \frac{\partial b\sb{0}}{\partial W\sb{2}} = 0
\]
so that 
\[
 b\sb{0} \in A[[W\sb{1}]]
. \]
 by \eqref{5.24g}.
Therefore, we have determined the image 
\begin{equation}\label{402.3b}
Z\sb 0 = \varphi(Y\sb{0}) = \sum\sb{n=0}^{\infty} X^{n} \binom{\alpha}{n}\sb{q}(e(t+W\sb{1}))^{-n} Q^{\alpha -n}b\sb{0}
\end{equation}
by \eqref{11.16a}, 
where all the coefficients of the power series $b\sb{0}(W\sb 1) -1$ are nilpotent. 
\subsection{Commutative deformation functor
$\mathcal{CF}\sb{L/k}$ for $\com (t,\, t^{\alpha})/\com$}\label{4.5.1d}
%$\mathcal{CF} \sb{ L/k}$ of infinitesimal deformations on the category $(CAlg/k)$ of commutative $L\n$-algebras for \QSI field exten sion} 
%and
%$\infgal\sb{\sigma \theta ^*} (L/k)$ }
In the Second Example, when we deal with the \QSI field extension $L/k$, the Galois 
hull $\eL / \K$ is a non-commutative algebra extension. 
So we have to consider the functor $\mathcal{NCF}\sb {L/k}$
on the category $(NCAlg/ L\n)$ of not 
necessarily commutative $L\n$-algebras. 
It is, however, easier to understand first the commutative deformation functor 
$\mathcal{CF}\sb{L/k}$ that is the restriction on the sub-category $(CAlg/ L\n)$ of the functor $\mathcal{NCF}\sb{L/k}$. 
%The restricted functor $\mathcal{NCF}\sb {L/k}\, | \sb{(CAlg/L^n )}$ is by definition $\mathcal{CF}\sb{L/k} = \mathcal{F}\sb {L/k} $.
We using the notation of Lemma \ref{402.3}, it follows from 
\eqref{402.3b} the following Proposition. 
\begin{proposition}\label{402b}
We set 
\begin{align}
Y\sb 1 (W\sb 1, \, W\sb 2\, ; \, X \, ) &:= (t+ W\sb 1)Q + X, \\
 Y\sb 2 (W\sb 1, \, W\sb 2\, ; \, X \, ) &:= \sum\sb{n=0}^{\infty} X^{n} 
 \binom{\alpha}{n}\sb{q}(t+W\sb{1})^{-n} Q^{\alpha -n}(y+ W\sb 2 ).
\end{align}
Then we have
\begin{align}
 \iota (t) & = Y \sb 1 (W\sb 1, \, W\sb 2\, ; \, X \, ) , \\
 %&:= Y \sb 1 ((e-1)t+ eW\sb 1, \, ( b\sb 0(W\sb 1) -1)y + b\sb 0 (W\sb 1) W\sb 2) \\
 \iota (y) & = Y\sb 2 (W\sb 1, \, W\sb 2\, ; \, X \, ) % &:= Y \sb 1 ((e-1)t+ eW\sb 1, \, ( b\sb 0(W\sb 1) -1)y + b\sb 0 (W\sb 1) W\sb 2) 
\end{align} 
and 
\begin{align}
\varphi (Y \sb 1 (W\sb 1, \, W\sb 2; \, X)) &:= Y \sb 1 ((e-1)t+ eW\sb 1, \, [ b\sb 0(W\sb 1) -1 ]y + b\sb 0 (W\sb 1) W\sb 2 ; \, X), \\
 \varphi (Y\sb 2 (W\sb 1, \, W\sb 2; \, X)) &:= Y \sb 2 ((e-1)t+ eW\sb 1, \, [ b\sb 0(W\sb 1) -1]y + b\sb 0 (W\sb 1) W\sb 2;\, X). 
\end{align}
In other words the infinitesimal deformation $\varphi$ is given by the coordinate transformation of the initial conditions 
$$
(W\sb 1 , \, W\sb 2) \mapsto (\varphi \sb 1 (W\sb 1,\, W\sb 2), \, \varphi \sb 2 (W\sb 1,\, W\sb 2)), 
$$ 
where 
\begin{align}\label{140226a}
\varphi \sb 1 (W\sb 1 , \, W\sb 2)&= (e-1)t + eW\sb 1 , \\ 
\label{140226b}
\varphi \sb 2 (W\sb 1 , \, W\sb 2)&= [ b\sb 0(W\sb 1) -1]y + b\sb 0 (W\sb 1) W\sb 2.
\end{align}
\end{proposition}
The set of transformations in the form of \eqref{140226a}, \eqref{140226b}
forms a group.
\begin{lemma}\label{40206a}
For a commutative $L\n$-algebra $A$, we set 
\begin{multline}
\hat{G}\sb{II}(A) :=
\{ 
 (( e-1)t + eW\sb 1, \, [b(W\sb 1) -1]y + b(W\sb 1)W\sb 2 ) 
 \in A[[W\sb 1, \, W\sb 2 ]]\times A[[W\sb 1. W\sb 2]]\, |\, \\
 e\in A, \, b(W\sb 1) \in A[[W\sb 1]], 
  \text{\it \/ all the coefficients of } b(W\sb 1) -1 
  \text{\it \/ and } e-1 \text{\it \/ are nilpotent } \}. 
\end{multline}
Then the set $\hat{G}\sb{II}(A)$
 is a group, the group law being the composition of coordinate transformations. 
\end{lemma}
\begin{proof} 
We have shown in Umemura \cite{ume96.2} that the set of coordinate transformations of $n$-variables with coefficients in a commutative ring that are congruent to the identity 
modulo nilpotent elements forms a group under the composite of 
transformations. 
So it is sufficient to show: 
\begin{enumerate}
\renewcommand{\labelenumi}{(\arabic{enumi})}
\item The set
 $\hat{G}\sb{II}(A)$ is closed under the composition. 
 \item The identity is in $\hat{G}\sb{II}(A)$.
 \item The inverse of every element in $\hat{G}\sb{II}(A)$ is in $\hat{G}\sb{II}(A)$.
 \end{enumerate}
 In fact, let 
 $$ 
%\begin{align*}
( ( e-1)t + eW\sb 1, \, [b(W\sb 1) -1]y + b(W\sb 1)W\sb 2 ), \quad
 (( f-1)t + fW\sb 1, \, [c(W\sb 1) -1]y + c(W\sb 1)W\sb 2 ) 
% \end{align*}
 $$
 be two elements of $\hat{G}\sb{II}(A)$. 
 We mean by their composite 
 \begin{equation}\label{140226e}
( (ef-1)t + efW\sb 1,\, [ 
b((f-1)t +fW\sb1)c(W\sb 1) -1 ]y + b((f-1)t +fW\sb1)c(W\sb 1) W\sb 2)
%( (ef-1)t + efW\sb 1,\, [b(W\sb1)c((e-1)t +eW\sb1) -1 ]y + [b(W\sb1)c((e-1)t +eW\sb1)] W\sb 2)
 \end{equation}
 that is an element of $\hat{G}\sb {II}(A)$. Certainly the identity 
 $(W\sb 1, \, W\sb 2)$ is expressed for $e=1$ and $b(W\sb 1) =1$.
As for the inverse 
$$
(( e-1)t + eW\sb 1, \, [b(W\sb 1) -1]y + b(W\sb 1)W\sb 2 ) ^{-1} 
=((e^{-1}-1)t + e^{-1} W\sb 1, \, [c(W\sb 1) -1]y + c(W\sb 1)W\sb 2),
$$
where 
$$ 
c(W\sb 1)=\frac{1}
{b((e^{-1}-1)t + e^{-1}W\sb 1 )}. 
$$
\end{proof}
We can summarize what we have proved as follows. 
\begin{proposition}\label{5.24i}
There exists a functorial inclusion on the category 
$(CAlg/L\n)$ of commutative $L\n$-algebras 
$$
\mathcal{CF}\sb{L/k}(A) :=\mathcal{NCF} 
\sb{L/k} 
|\sb{(CAlg/L\n)}
 \left( A\right) 
 \hookrightarrow \hat{G}\sb{II}(A) 
$$
that sends infinitesimal deformation $\varphi$ 
to 
$$
\left( (e-1)t + eW\sb 1, \, [b\sb 0 (W\sb 1) -1]y + b\sb 0(W\sb 1) W\sb 2 \right) \in \hat{G}\sb{II}(A)
$$ 
 for every commutative $L\n$-algebra $A$. 
\end{proposition}
In the Definition of the group functor $\hat{G}\sb{II}$ in Lemma \ref{40206a}, 
we can eliminate the variable $W\sb 2$. 
\begin{lemma}\label{40208a}
We introduce a group functor 
$$
\hat{G}\sb 2\, :\, (CAlg/L\n) \to (Grp),
$$
setting 
\begin{multline*}
\hat{G}\sb 2 \, (A)=\, \{ \, (e, \, b(\,W\sb 1\, )) \in A \times A[[W\sb 1]] \, | \\
 \, \text{ \it \/ All the coefficients of } b(\, W\sb 1 \,) -1 \,\text{\it \/ and $e-1$ are nilpotent} \}   
\end{multline*}
for every $A \in ob(CAlg/L\n)$.
The group law, the identity and the inverse are given as below. 
\par
For two elements $(e, \, b(W\sb 1), \, (f, \, c(W\sb 1))$, their product is by definition 

\begin{equation} \label{140226c} 
\left( ef, \,b((f-1)t+ fW\sb 1)c(W\sb 1)\right) .
\end{equation}
 The identity 
is $(1,\, 1)$ and the inverse 
$$
(e, \, b(W\sb 1))^{-1} = \left( \frac{1}{e}, \, \frac{1}{b((e^{-1} -1)t + e^{-1}W\sb 1)} \right). 
$$
Then the there exists an isomorphism of group functors.
$$
\hat{G}\sb{II} \simeq \hat{G}\sb 2. 
$$
\end{lemma}
\begin {proof}
In fact, for a every commutative algebra $A \in ob(CAlg/L\n)$, the map 
\begin{align}
\hat{G}\sb{II}(A) &\to \hat{G}\sb 2 (A), \\
((e-1)t +eW\sb1, \, ( b(W\sb 1) -1)y + b(W\sb 1) W\sb 2) &\mapsto (e, \, b(W\sb 1) )
\end{align}
gives an isomorphism of group functors. 
\end{proof}
\begin{remark}\label{140226d}
In the composition laws 
for $\tilde{G}\sb {II}$
 \eqref{140226e} and for $\tilde{G}\sb{2}$
 \eqref{140226c}, we substitute in the variable $W\sb 1$
the linear polynomial $(e-1)t + W\sb 1$ in the power series $c(W\sb 1) $
to get $c((e-1)t + eW\sb 1)$. Since $c(W\sb 1 )$ is a power series, in order that the substitution has sense, we can not avoid the condition that $e -1 $ is nilpotent. We can neither define the global group functors $G\sb{II}$ nor 
$G\sb 2$ whose completions are $\hat{G}\sb {II}$, $\hat{G}\sb 2$ respectively. 
\end{remark}

It is natural to wonder what is the image of the inclusion map in Proposition \ref{5.24i}.
\begin{conjecture}\label{5.24j}
If $q$ is not a root of unity, 
the inclusion in Proposition \ref{5.24i} is the equality. 
\end{conjecture}
%We are guided to the following \\
\begin{proposition}
\label{9.28aa} Origin of the group structure teaches us that 
if the Conjecture \ref{5.24j} is true, then 
the group functor 
$$
\hat{G}\sb {II} :\, (CAlg/L\n) \to (Grp)
$$
 operates on the functor 
$$
\mathcal{CF}\sb {L/k}: \, (CAlg/ L\n) \to (Set)
$$
through the transformations of the initial conditions $(W\sb 1, \, W\sb 2)$, 
in such a way that $$(\, \hat{G}\sb{II}, \, \mathcal{CF}\sb{L/k}\, )$$
is a torsor. So we may say that the Galois group functor 
 %We have an isomorphism of group functors 
$$
\cinfgal \, ((\com (t, \, t^\alpha \, \sigma , \theta^{*} )/ \com ) \simeq \hat{G}\sb{II}. 
$$
\end{proposition}
%%%%%%%%%%%%%%%
%%%%%%%%%%%%%%%%%%%%%%%%%%%%%
\begin{remark}\label{5.24l}
We explain a background of Conjecture \ref{5.24j}. 
\end{remark}
\begin{lemma}\label{5.24n} The Galois hull $\eL$ is a localization of the following ring
\[
 L\s\langle Q, X, \frac{1}{tQ + X}\rangle \sb{alg} \langle\frac{\partial^{l}}{\partial t^{l}}Y\sb{0}\rangle \sb{alg, \, l\in\N}. 
\]
\end{lemma}
\begin{proof} Since $\iota(t) = t Q + X$, as we have seen in the First Example, 
\[
 L\s \langle Q, X \rangle \sb{alg}\langle \frac{\partial^{l}}{\partial t^{l}}Y\sb{0}\rangle \sb{alg, \, l\in\N} \subset \eL.
\]
We show that the ring
\[
L\s\langle Q, X\rangle \sb{alg}\langle \frac{\partial^{l}}{\partial t^{l}}Y\sb{0}\rangle \sb{alg, \,l\in\N}
\]
is closed under the operations $\hat{\Sigma} , \hat{\Theta}^{(i)}, \partial/\partial t$ and $\partial/\partial y$ of $F(\Z, \, L\s)[[X]]$. Evidently the ring is closed under the last two operators. Since the operators $\hat{\Sigma} $ and $\partial^{n} / \partial t^{n}$ operating on $F(\Z, L\s)[[X]]$ mutually commute, it follows from \eqref{5.24a}
\[
 \hat{\Sigma} \left( \frac{\partial^{n} Y\sb{0}}{\partial t^{n}} \right) = \frac{\partial^{n}}{\partial t^{n}} \hat{\Sigma} (Y\sb{0}) = \frac{\partial^{n}}{\partial t^{n}}(q^{\alpha}Y\sb{0}) = q^{\alpha}\frac{\partial^{n} Y\sb{0}}{\partial t^{n}}
. \]
So the ring is closed under $\hat{\Sigma} $. 
Similarly since the operators $\hat{\Theta}^{(1)}$ and $\partial^{n} / \partial t^{n}$ mutually commute on $F(\Z, L\s)[[X]]$,
\begin{align*}
\hat{\Theta}^{(1)}\left( \frac{\partial^{n} Y\sb{0}}{\partial t^{n}} \right) 
&=\frac{\partial^{n}}{\partial t^{n}} \hat{\Theta}^{(1)}(Y\sb{0})\\
&=\frac{1}{y}\frac{\partial^{n}}{\partial t^{n}} \hat{\Theta}^{(1)}(Y\sb{0}y)\\
&=\frac{1}{y}\frac{\partial^{n}}{\partial t^{n}} \hat{\Theta}^{(1)}(\iota(y))\\
&=\frac{1}{y}\frac{\partial^{n}}{\partial t^{n}} \iota(\theta^{(1)}(y))\\
&=\frac{1}{y}\frac{\partial^{n}}{\partial t^{n}} \iota \left( \frac{\sigma(y) - y}{(q-1)t}\right)\\
&=\frac{1}{y}\frac{\partial^{n}}{\partial t^{n}} \left( \frac{q^{\alpha}Y\sb{0}y - Y\sb{0}y}{(q-1)(tQ+X)}\right)\\
&=\frac{1}{y}\frac{\partial^{n}}{\partial t^{n}} \left( \frac{q^{\alpha}Y\sb{0} - Y\sb{0}}{(q-1)(tQ+X)}\right),
\end{align*}
which is an element of the ring. 
\end{proof}
Conjecture \ref{5.24j} arises from {\it experience that if $q$ is not a root of unity, 
we could not find any non-trivial algebraic relations among
 the partial derivatives} 
 $$
\frac{\partial^{n} Y\sb{0}}
{\partial t^{n}}\qquad
\text{\it \/ for \/} n \in \N 
$$
{ \it over $L\s$
 so that we could guess that there would be none.} 
\par 
In fact, assume that we could prove our guess. Let $\varphi \colon \eL \rightarrow 
F(\Z, A[[W\sb{1}, W\sb{2}]])[[X]]$ be an infinitesimal deformation of $\iota$. So as we have seen
\[
 Z\sb{0} = \varphi(Y\sb{0}) = \sum\sb{n=0}^{\infty} X^{n} \binom{\alpha}{n}\sb{q}(et)^{-n}Q^{\alpha-n}b(W\sb{1})
\]
with $b(W\sb{1}) \in A[[W\sb{1}]]$. 
%Since the observation says that $\partial^{n} Y\sb{0}/\partial t^{n},\, n \in \N$ are algebraically independent over $L\s$, there is 
There would be no constraints among the partial derivatives 
$\partial^{n} b(W\sb{1})/\partial W\sb{1}^{n},\, n \in \N$ and hence we could choose any power series $b(W\sb{1}) \in A[[W\sb{1}]]$ such that every coefficient of the power series $b(W\sb 1) -1$ is nilpotent. 
\subsection{Non-commutative deformation functor $\NCF$ for $\com(t,\, t^{\alpha})/\com$}\label{11.23a}
We study the functor $\NCF(A)$ of non-commutative deformations 
\[
\NCF \colon \NCA \rightarrow (Set).
\]
For a not necessarily commutative $L\n$-algebra $A \in ob\NCA$, let
\begin{equation}\label{9.13a}
 \varphi \colon \eL \rightarrow F(\Z,
 \, A[[W\sb{1}, W\sb{2}]])[[X]] 
\end{equation}
be an infinitesimal deformation of the canonical morphism
\[
\iota \colon \eL \rightarrow F(\Z, A[[W\sb{1},W\sb{2}]])[[X]]
. \]
Both $t$ and $y$ are elements of the field $\com(t,t^{\alpha} )= \com(t,y)$ so that $[t,y] = ty -yt = 0$. So for the deformation $\varphi \in \NCF(A)$ we must have 
\begin{equation}\label{6.4a}
[\varphi(t),\varphi(y)] = \varphi(t)\varphi(y) - \varphi(y)\varphi(t) = 0.
\end{equation}
When we consider the non-commutative deformations, the commutativity \eqref{6.4a} gives a constraint for the deformation. 
To see this, we need a Lemma. 
\begin{lemma}\label{6.2g}
For every $l\in \N$, we have
\[
 q^{l}\binom{\alpha}{l}\sb{q} + \binom{\alpha}{l-1}\sb{q} = \binom{\alpha}{l}\sb{q} + 
 q^{\alpha -l +1}\binom{\alpha}{l-1}\sb{q}. 
\]
\end{lemma}
\begin{proof} 
This follows from the definition of $q$-binomial coefficient. 
\end{proof}
\begin{lemma}\label{9.13b}
Let $A$ be a not necessarily commutative $L\n$-algebra in $ob\, (NCAlg/L\n)$. Let $e, \, f \in A$ such that $e-1$ and $f$ are nilpotent. We set 
$$
\mathcal{A} : = (e(t+W\sb{1}) + f)Q + X
$$ 
and for a power series $b(W\sb{1})\in A[[W\sb{1}]]$, we also set
\[
 \mathcal{Z} := \sum\sb{n=0}^{\infty} X^{n} 
 \binom{\alpha}{n}\sb{q}(e(t+W\sb{1})+f)^{-n}
 Q^{\alpha -n}b(W\sb{1})
\]
so that $\mathcal{A}$ and $\mathcal{Z}$ are elements of $F(\Z,A[[W\sb{1}]])[[X]]$. 
The following conditions are equivalent. 
\begin{enumerate}
\renewcommand{\labelenumi}{(\arabic{enumi})}
\item $[\mathcal{A},\, \mathcal{Z}] := \mathcal{A}\mathcal{Z} - \mathcal{Z}\mathcal{A} = 0$. 
\item $[e(t + W\sb{1}) + f, \, b(W\sb{1})] = 0$. 
\end{enumerate}
\end{lemma}
\begin{proof}
We formulate condition $(1)$ in terms of coefficients of the power series in $X$. Assume condition $(1)$ holds so that we have 
\begin{equation}\label{6.2d}
\begin{split}
((e(t+W\sb{1}) + f)Q + X)\left( \sum\sb{n= 0}^{\infty}X^{n} \binom{\alpha}{n}\sb{q}(e(f+W\sb{1})+ f)^{-n}Q^{\alpha -n }b(W\sb{1}) \right)\\
=\left( \sum\sb{n= 0}^{\infty}X^{n} \binom{\alpha}{n}\sb{q}(e(f+W\sb{1})+ f)^{-n}Q^{\alpha -n }b(W\sb{1}) \right)
(( e(t +W\sb{1}) + f )Q + X ). 
\end{split}
\end{equation}
Comparing degree $l$ terms in $X$ of \eqref{6.2d}, we fined condition $(1)$ is equivalent to 
\begin{equation}\label{6.2e}
\begin{split}
&q^{l}\binom{\alpha}{l}\sb{q}(e(t+ W\sb{1})+ f)^{-l+1}Q^{\alpha -l+1}b(W\sb{1}) \\
&\hspace*{5em}
+ \binom{\alpha}{l-1}\sb{q}(e(t+W\sb{1})+f)^{-l+1}Q^{\alpha -l+1} b(W\sb{1})\\
&=
\binom{\alpha}{l}\sb{q}(e(t+ W\sb{1})+ f)^{-l}b(W\sb{1}) (e(t+ W\sb{1})+ f)Q^{\alpha -l +1}\\
&\hspace*{5em}
+ \binom{\alpha}{l-1}\sb{q}q^{\alpha -l+1}(e(t+W\sb{1})+f)^{-l+1}Q^{\alpha -l+1} b(W\sb{1}).
\end{split}
\end{equation}
So the condition $(1)$ is equivalent to 
\begin{equation}\label{6.2f}
\begin{split}
&\hspace*{2em}q^{l}\binom{\alpha}{l}\sb{q}(e(t+ W\sb{1})+ f)b(W\sb{1}) \\
&\hspace*{5em}+ \binom{\alpha}{l-1}\sb{q}(e(t+W\sb{1})+f)b(W\sb{1})\hspace{5em}\\
&=\binom{\alpha}{l}\sb{q}b(W\sb{1}) (e(t+ W\sb{1})+ f)\\
&\hspace*{5em}+ \binom{\alpha}{l-1}\sb{q}q^{\alpha -l+1}(e(t+W\sb{1})+f)b(W\sb{1})
\end{split}
\end{equation}
for every $l \in \N$. Condition \eqref{6.2f} for $l = 0$ is condition (2). Hence condition (1) implies condition (2). 
Conversely condition (1) follows from (2) in view of \eqref{6.2f} and Lemma \ref{6.2g}. 
\end{proof}
%\subsection{Non-commutative deformations $\NCF$}
Now let us come back to the infinitesimal deformation \eqref{9.13a}
%\[
 %\varphi \colon \eL \rightarrow F(\N,A[[W\sb{1},W\sb{2}]])[[X]]
%\]
of the canonical morphism $\iota$.
%\[
%\iota \colon \eL \rightarrow F(\N, L\n[[W\sb{1},W\sb{2}]])[[X]]\rightarrow F(\N, A[[W\sb{1},W\sb{2}]])[[X]].
%\]
The argument in Section \ref{10.4a} allows us to determine the restriction $\varphi$ on the sub-algebra generated by $\iota(t) = tQ + X$
 over $L\s$ 
 invariant under the $\hat{\Theta}^{(i)}$'s, 
 $\hat{\Sigma} $ and $\{\partial/ \partial t, \partial/ \partial y \}$ in $F(\Z, L\s)[[X]]$. So there exist $e, \, f \in A$ such that $ef= qfe,\, e-1, \, f$ are nilpotent and such that
\[
 \varphi(Q) = eQ \,\text{\it \/ and \/}\varphi(X) = X + fQ,
\]
that are equations in $F(\Z,A[[W\sb{1},W\sb{2}]])[[X]]$. 
In particular 
\[
 \varphi(t) = \varphi(tQ+X) = (et+f)Q+ X = (e(t+W\sb{1})+ f)Q + X,
\]
where we naturally identify rings
\[
F(\Z,L\s)[[X]] \rightarrow F(\Z, L\n[[W\sb{1},W\sb{2}]])[[X]] \rightarrow F(\Z, A[[W\sb{1},W\sb{2}]])[[X]]
\]
through the canonical maps. \par
Then the argument in 
the commutative case allows us to show that there exists a power series $b\sb{0}(W\sb{1}) \in A[[W\sb 1]]$ such that 
\[
 \varphi(Y\sb 0) = \sum\sb{n=0}^{\infty} X^{n} \binom{\alpha}{n}\sb{q}(e(t+ W\sb{1})+ f)^{-n}Q^{\alpha - n}b\sb{0}(W\sb{1}).
\]
such that all the coefficients of the power series $b\sb 0(W\sb 1) -1 $ are nilpotent.
As we deal with the not necessarily commutative algebra $A$, the commutation relation in $L$ gives a constraint. Namely since $\iota(y) = yY\sb{0}$ and $t y = y t$ in $L$ so that $\iota(t)\iota(y) = \iota(y)\iota(t)$, we get $\iota(t)(yY\sb{0}) = (yY\sb{0})\iota(t)$ in $\eL$ and 
$
\varphi (tQ + X ) \varphi (Y\sb 0) =\varphi (Y\sb 0) \varphi ( tQ + X)
$
.
So we consequently have 
\begin{equation}\label{6.8a}
\mathcal{A}\mathcal{Z}\sb 0 = \mathcal{Z}\sb{0}\mathcal{A} \qquad \text{\it \/ in \/ } F(\Z, A[[W\sb{1}, \, W\sb{2}]])[[X]], 
\end{equation}
setting 
\[
 \mathcal{A} : = (e(t+W\sb{1}) + f)Q + X, \qquad \mathcal{Z}\sb 0 := \sum\sb{n=0}^{\infty} X^{n} \binom{\alpha}{n}\sb{q}(e(t+W\sb{1})+f)^{-n}Q^{\alpha -n}b\sb 0 (W\sb{1}).
\]
\begin{lemma}\label{10.23a}
We have 
\[ [e(t + W\sb{1}) + f,\, b\sb 0 (W\sb{1})] = 0.
\]
\end{lemma}
\begin{proof}
This follows from \eqref{6.8a} and Lemma 
\ref{9.13b}. 
\end{proof}
\begin{definition}\label{10.19a}
We define a functor 
\[
 QG\sb{2\, q} \colon \NCA \rightarrow (Set)
\]
by putting 
\begin{multline*}
 QG\sb{2 \, q}(A) 
 = \{ ( \begin{bmatrix}
 e & f\\
 0 & 1
 \end{bmatrix}, \, b(W\sb{1}) 
 ) \in M\sb {2} (A)\times A[[W\sb 1]] 
\, |\, 
 e, \, f \in A,\, ef = qfe, \\
 e \text{\it \/ is invertible in } A , \, 
 b(W\sb 1 ) \in A[[W\sb 1]],\, [e(t+W\sb{1}) + f, \, \, b(W\sb{1})] = 0
 \, 
\} 
\end{multline*}
%\end{definition}
for $A\in ob\,(NCAlg/l\n)$. 
\par The functor $QG\sb {2\, q}$ is almost 
a quantum group in usual sense of the word. See Remark \ref{140226d}. 
We also need the formal completion $\widehat{QG}\sb {2\, q}$ of the quantum group functor $QG\sb{2\, q}$ so that 
$$
\widehat{QG}\sb {2\, q}:(NCAlg/L\n) \to (Set)
$$
is given by 
\begin{multline*}
\widehat{ QG}\sb{2 \, q}(A) 
= \{ ( \begin{bmatrix}
 e & f \\
 0 & 1
 \end{bmatrix}, \, b(W\sb{1})) \in QG\sb {2\, q}(A) \, \\ 
 |\, 
 e-1,\, f \text{\it \/ and all the coefficients of } 
 b(W\sb 1 ) -1 \, \text { are nilpotent}
 \} 
\end{multline*}
for $A\in ob (NCAlg/L\n)$. 
\end{definition}
Studying commutative deformations of the Galois hull 
$\eL/\K$ of $( \com (t, \, t^\alpha ), \, \sigma , \theta^ * )/ \com $, we introduced in Lemma 
\ref{40206a} the functor $\hat{G}\sb{II}$ and in 
Lemma \ref{40208a} the functor $\hat{G}\sb{2}$. 
They are isomorphic. The former involves the variable $W\sb 2$ but the latter does not. 
The functor $\widehat{QG}\sb {2\, q}$ does not involve the variable 
$W\sb 2$. As you imagine, we also have another functor $\widehat{QG}\sb {II\, q}$ equivalent to 
the functor $\widehat{QG}\sb {2 \, q}$ and 
 involving the variable $W\sb 2$. 
%%%%%%%%%%%%%
%%%%%%%%%%%%%%%%
Using Definition \ref{10.19a}, we can express what we
have shown. 
\begin{proposition}\label{6.8i}
There exists a functorial inclusion 
\[
 \NCF(A) \hookrightarrow \widehat{QG}\sb{2 \, q}(A)
\]
 sending $\varphi \in \NCF(A)$ to 
\[
 ( \begin{bmatrix}
 e & f\\
 0 & 1
 \end{bmatrix},\, b\sb 0 (W\sb{1}) )\in \widehat{QG}\sb{2 \, q}(A). 
\]
\end{proposition}
We show that $\widehat{QG}\sb{2 \, q}$ is a quantum formal group over $L\n$. In fact, we take two elements 
\[
(G,\, \xi(W\sb{1})) =
 ( \begin{bmatrix}
 e & f\\
 0 & 1
 \end{bmatrix},\, \xi (W\sb{1}) ),\qquad 
(H,\, \eta(W\sb{1})) =
 ( \begin{bmatrix}
 g & h\\
 0 & 1
 \end{bmatrix}, \, \eta (W\sb{1}) )
\]
of $\widehat{QG}\sb{2 \, q}(A)$ so that $e,\,f, g,\, h \in A$ satisfying 
\[
 ef = qfe, \qquad gh= qhg,
\]
the elements $ e - 1,\, g - 1$ and $f, \, h$ are nilpotent 
 and such that
\begin{equation}\label{6.6e}
[e(t+ W\sb 1) +f, \, \xi(W\sb{1})] = 0, \qquad [g(t+W\sb{1}) + h, \, \eta(W\sb{1})] = 0.
\end{equation}
When the following two sub-sets of the ring $A$ 
\begin{align}\label{14426a}
%\begin{enumerate}
%\renewcommand{\labelenumi}{(\arabic{enumi})}
\{e, f \} \,\cup &
\text{ 
 (the sub-set consisting of all the coefficients of the power series $\xi(W\sb{1})$)}, 
 \\ 
 \label{14426b}
 \{g,\, h \}\,\cup &
 \text{ (the sub-set consisting of 
all the coefficients of the 
power series $\eta(W\sb{1})$)}, 
\end{align}
 %?\which are two subsets of $A$, 
 are mutually commutative, 
we define the product of $(G, \, \xi(W\sb{1}))$ and $(H, \, \eta(W\sb{1}))$ by
\[
(G, \xi(W\sb{1})) * (H, \, \eta(W\sb{1})) := (GH,\, \xi((g-1)t + h +gW\sb 1)\eta(W\sb{1})).
\]
%which are subsets of $A$, are mutually commutative, 

\begin{lemma} \label{14320a}
The product $(GH,\, \xi((g-1)t + h +gW\sb{1})\eta(W\sb{1}))$ is indeed an element of $\widehat{QG}\sb{2 \, q}(A)$. 
\end{lemma}
\begin{proof}
First of all, we notice that
the constant term $(g-1)t +h$ of 
 the linear polynomial in $W\sb 1$ 
\begin{equation}\label{10.22a}
 (g-1)t + h + gW\sb{1} 
\end{equation}
is nilpotent so that we can substitute 
\eqref{10.22a}
into the power series $\xi (W\sb 1 )$. Therefore 
 $$
 \xi((g-1)t + h +gW\sb{1} )\eta(W\sb{1} )
 $$
is a well-determined element of the power series ring 
$A[[W\sb 1]]$.
We have seen in Section \ref{10.4a} that if 
 $\{ e, \, f \}$
 and $\{g, \, h \} $
 are mutually commutative, then 
the product $GH$ of matrices $G, \, H \in \gH\sb{q \, L\n}(A)$ 
 is in 
 $\gH\sb{q \, L\n}(A)$. 
 Since 
$$
GH = \begin{bmatrix}
 eg & eh + f\\
 0 & 1
\end{bmatrix}, 
$$
it remains to show 
\begin{equation}\label{6.6a}
[eg(t+W\sb{1})+ eh +f, \, \xi((g-1)t + h +gW\sb{1})\eta(W\sb{1})] = 0. 
\end{equation}
The proof of \eqref{6.6a} is done in several steps. \par
First we show 
\begin{equation}\label{6.6x}
[\xi((g-1)t + h +gW\sb{1}), \, \eta(W\sb{1})]= 0. 
\end{equation}
This follows, in fact, from the mutual commutativity 
of the sub-sets \eqref{14426a} and \eqref{14426b} above, and the second equation of \eqref{6.6e}. \par
Second, we show 
\begin{equation}\label{6.6g}
[eg(t+W\sb{1})+ eh +f,\, \xi((g-1)t + h +gW\sb{1})] = 0.
\end{equation}
%This is a consequence of 
%the firt equation of \eqref{6.6e} and the mutual commutativity of the sets (1) and (2). 
To this end, we notice
\begin{equation}\label{6.6c}
eg(t+W\sb{1})+ eh +f = e((g-1)t + h +gW\sb{1}) + et + f.
\end{equation}
So we have to show 
\begin{equation}
[e((g-1)t + h +gW\sb{1}) + et + f, \, 
\xi((g-1)t + h +gW\sb{1})]=0.
\end{equation}
This follows from 
the first equation of \eqref{6.6e}
and the mutual commutativity of the sub-sets \eqref{14426a} and \eqref{14426b}. \par
We prove third 
\begin{equation}\label{6.8d}
[eg(t+W\sb{1})+ eh +f, \, \eta(W\sb{1})]=0 .
\end{equation}
In fact, by mutual commutativity of sub-sets \eqref{14426a} and 
\eqref{14426b}, 
\begin{align*}
[e(t+ gW\sb 1)+ eh +f , \, \eta (W\sb 1)] &= [e(t+ gW\sb 1)+ eh , \, \eta (W\sb 1)] \\
%by the mutual commutativity of the sub-sets \eqref{14426a} and 
%\eqref{14426b}, 
%
%= [e(t+ gW\sb 1)+ eh , \eta (W\sb 1)]
& =[e(t+ gW\sb 1+ h) , \, \eta (W\sb 1) ],
\end{align*}
which is equal to $0$ thanks to mutual commutativity of the sub-sets \eqref{14426a} and \eqref{14426b} and the second equality of 
\eqref{6.6e}. 
%this a consequence of the second equation of 
%\eqref{6.6e} and the mutual commutativity of thesub-sets \eqref{14426a} and \eqref{14426b}.
%\eqref{6.6c} and \eqref{6.6b}. \par
%Equality \eqref{6.6a} is a consequence of \eqref{6.6x}, \eqref{6.6g} and \eqref{6.6c}. 
\end{proof}
One can check associativity for the multiplication by a direct calculation. The unit element is given by 
\[
(I\sb{2}, 1) \in \widehat{QG}\sb{2 \, q}(L\n). 
\]
The inverse is given by the formula below. 
For an element 
$$
(G, \, b(W\sb 1) ) =
( \begin{bmatrix}
e & f \\ 
0 & 1
\end{bmatrix}, \, b(W\sb 1) )
\in \widehat{QG}\sb{2\, q}(A),    
$$
 we set 
$$
(G, \, b(W\sb 1 ) )^{-1}: = ( 
\begin{bmatrix}
e^{-1} & -e^{-1}f\\
0& 1
\end{bmatrix}, \,
b(e^{-1}W\sb 1) + (e^{-1}t -e^{-1}f) ^{-1}
 ) \in \widehat{QG}\sb{2\, q^{-1}}(A), 
$$
then we have 
$$
(G, \, b(W\sb 1) )^{-1} * (G, \, b(W\sb 1) ) =
(G, \, b(W\sb 1) ) * (G, \, b(W\sb 1) )^{-1} = (I\sb 2 , \, 1 ).
$$
\begin{conjecture}\label{6.8j}
If $q$ is not a root of unity, 
the injection in Proposition \ref{6.8i} is bijective for every $A \in ob(NCAlg/L\n)$. 
\end{conjecture}
\begin{proposition}
Conjecture \ref{6.8j} implies Conjecture \ref{5.24j}. 
\end{proposition}
\begin{proof}
Let us assume Conjecture \ref{6.8j}. Take an element $(e,\, \xi(W\sb{1}))\in \hat{G}\sb{2}(A)$ for $A \in ob\, (Alg/L\n)$. 
Since $A$ is commutative, the commutation relation 
in Lemma \ref{10.23a}
imposes no condition on $\xi(W\sb{1})$, 
\[
(e, \, \xi(W\sb{1}))=(
\begin{bmatrix}
e & 0\\
0 & 1
\end{bmatrix},\, \xi(W\sb{1})
)\in \widehat{QG}\sb{2 \, q}(A). 
\]
Conjecture \ref{6.8j} says that 
if $q$ is not a root of unity, 
 $(e, \, \xi(W\sb{1}))$ arise from an infinitesimal deformation
\[
 \iota \colon \eL \rightarrow F(\Z,\, A[[W\sb{1},W\sb{2}]])[[X]]. 
\]
\end{proof}
Conjecture \ref{6.8j} says that we can identify the functor $\NCF$ with the quantum formal group $\widehat{QG}\sb{2\, q}$. 
 To be more precise, 
the argument in the First Example studied in \ref{10.4a} allows us to
define a formal $\com$-Hopf algebra $\hat{\mathfrak{I}}\sb q$ 
and hence 
$$
\hat{\mathfrak{I}}\sb{q\, L\n}:= \widehat{\mathfrak{I}}\sb q \hat{\otimes} \sb \com L\n, 
$$
%as well as its formal completion $\hat{\gI}\sb{ q\, L\n}$
which is a functor on the category $(NCAlg/L\n )$\
so that we have a functorial isomorphism 
$$
\hat{\gI}\sb {q\, L\n}(A) \simeq \widehat{QG}\sb{2\, q}(A)
\quad \text{\it \/ for every } L\n\text{\it \/ -algebra }A 
\in ob\, (NCAlg/L\n).
$$ 
%%%%%%%%%%%%%%%%%%%%%%
\begin{definition} We define a functor 
$$
\widehat{QG}\sb{II \, q} : \NCA \to (Set)
$$ 
by setting 
\begin{multline*}
\widehat{QG}\sb{II \, q}(A) := \{ \left( (e-1)t + f + eW\sb 1, \, (b(W\sb 1) -1)y + b(W\sb 1)W\sb 2 \right) \in \\ 
A[[W\sb 1, \, W\sb 2]] \times A[[W\sb 1, \, W\sb 2]] \, | \, 
e, \, f \in A \text{\it \/ and } b(W\sb 1) \in A[[W\sb 1]], \\
[(e-1)t + f + eW\sb 1, \, (b(W\sb 1)-1)y + b(W\sb 1)W\sb 2 ]= 0, 
\\
 e-1, \, f \, \text{ and all the coefficients of the power series } b(W\sb 1 ) -1\, \text{\it are nilpotent} \}
\end{multline*}
for every $A \in ob(NCAlg/L\n )$. 
\end{definition}
\begin{lemma}\label{4.5.2f} The functor $\widehat{QG}\sb{II \, q}$ is a quantum formal group. Namely, 
for two elements
\begin{align*}
u :=(u\sb 1 , u\sb 2) :=((e-1)t + f + eW\sb 1, \, &(b(W\sb 1) -1)y + b(W\sb 1)W\sb 2 ) , 
\\
v:=(v\sb 1,\, v\sb 2):= ( (g-1)t + h + gW\sb 1, \, &(c(W\sb 1) -1)y + c(W\sb 1)W\sb 2 )
\end{align*}
of $\widehat{QG}\sb{II\, q}(A)$, 
we consider the following two sub-sets of the ring $A$ :
\begin{enumerate}
\renewcommand{\labelenumi}{(\arabic{enumi})}
\item
The sub-set $ S\sb u$ of the coefficients of the two power series $u\sb 1,\, u\sb 2$ of $u$ and 
\item
the sub-set $S\sb v$ of the coefficients of the two power series $v\sb 1, \, v\sb 2$ in $v$. 
\end{enumerate}
If the sets $S \sb 1$ and $S\sb 2$ are mutually commutative, 
we define their product $u * v$ by 
$$
((eg -1)t + eh + f + egW\sb 1 , \, (b((g-1)t + h +gW\sb{1})c(W\sb{1}) -1)y + b((g-1)t + h +gW\sb{1})c(W\sb{1})W\sb2). 
$$
that is the composite of coordinate transformations 
\begin{align*}
(W\sb 1, \, W\sb 2 ) &\mapsto 
((e-1)t + f + eW\sb 1, \, (b(W\sb 1) -1)y + b(W\sb 1)W\sb 2 )
 \text{\it \/ and} \\
 (W\sb 1, \, W\sb 2) &\mapsto ( (g-1)t + h + gW\sb 1, \, (c(W\sb 1) -1)y + c(W\sb 1)W\sb 2 ), &
\end{align*}
then the product $u * v$ is an element of $\widehat{QG}\sb{II \, q}$. The co-unit is given by the identity transformation of $(W\sb 1, \, W\sb 2)$. 
\end{lemma}
The quantum formal group $\widehat{QG}\sb {II\, q}$ 
arises as symmetry of the initial conditions 
of \QSI equations.
\begin{align*}
\begin{array}{ll}
\sigma (t) = qt,& \sigma (t^\alpha ) = q^\alpha t^\alpha, \vspace{1ex}\\
\theta^{(1)}(t) = 1,& \theta^{(1)}(t^\alpha ) = [\alpha ]\sb q t^ \alpha.
\end{array}
\end{align*}
%%%%%%%%%%%%%%%%%
 \begin{proposition}\label{4.5.1c}
 For every algebra $A \in (NCAlg/L\n)$, 
we have a functorial isomorphism of quantum formal 
group
$$
\widehat{QG}\sb{2\, q}(A) \to \widehat{QG}\sb{II\, q}(A)
$$
sending an element $$( \begin{bmatrix}
e & f \\
0& 1
\end{bmatrix}, \, b(W\sb{1})
) \in \widehat{QG}\sb{2\, q}(A) \text{ to } ( (e-1)t +f+ eW\sb{1}, \, b(W\sb 1)W\sb 2 + (b(W\sb{ 1}) -1)y)\in \widehat{QG}\sb{II, q}(A).
$$
 \end{proposition}
 %%%%%%%%%%%%%%%%%%%%%%%%%%%%%%%%%%%%%%%%%%%%%%%%%%%%%%%%%%%%%%%%%%%
 Thanks to 
Propositions \ref{6.8i}, \ref{4.5.1c} and Conjecture \ref{6.8j}, we are in
 the similar situation as in the commutative deformations in \ref{4.5.1d}. 
\begin{theorem} 
We have an inclusion 
$$
\mathcal{NCF}\sb{L/k} \hookrightarrow \widehat{QG}\sb {II\, q } 
$$
of functors on the category $(NCAlg/L\n )$ taking values in the 
category of sets, where 
\begin{equation}\label{4.5.1a}
L/k = (\com (t, \, t^\alpha ), \, \sigma ,\, \theta ^* )/ \com .
\end{equation}
Let us assume Conjecture \ref{6.8j}. Then the inclusion \eqref{4.5.1a} is bijection so that we can identify the functors 
$$
\mathcal{NCF}\sb{L/k} \simeq \widehat{QG}\sb {II\, q }. 
$$
The quantum formal group $\widehat{QG}\sb {II\, q }$ 
operates on the functor 
 $\mathcal{NCF}\sb{L/k}$
in an appropriate sense, through the initial conditions. 
(cf. The commutativity condition in Lemma \ref{4.5.2f}.) 
So we may say that 
the quantum formal Galois group 
$$
{\rm NC}\infgal (L/k) \simeq \widehat{QG}\sb {II\, q}. 
$$
\end{theorem} 

%%%%%%%%%%%%%%%%%5
\subsection{Summary on the Galois structures of the field extension $\com(t,\, t^{\alpha})/\com$ }\label{151007b}
Let us summarize our results on the field extension $(\com(t, \, t^{\alpha})/\com)$. 
\begin{enumerate}
\renewcommand{\labelenumi}{(\arabic{enumi})}
\item Difference field extension $(\com(t,\, t^{\alpha}),\, \sigma)/\com$. This is a Picard-Vessiot extension with Galois group $\G\sb{m\, \com} \times \G\sb{m\,\com}$. 
\item Differential field extension $(\com(t,\, t^{\alpha}),\, d/dt)/\com$. This is not a Picard-Vessiot extension. The Galois group 
$$
\infgal(L/k)\colon (CAlg/L\n ) \to (Grp)
$$
 is isomorphic to $\hat{\G}\sb{m\, L\n} \times\sb{L\n} \hat{\G}\sb{a \,L\n}$, 
 where $\hat{\G}\sb{m \, L\n}$ and 
 $\hat{\G}\sb {a \, L\n}$ are formal completion of the multiplicative group and the additive group. 
 So as a group functor on the category 
 $(CAlg/L\n)$, we have 
 $$ 
 \hat{\G }\sb{ m \, L\n}(A) =\{ b\in A \, | \, b-1 \text{\it \/ is nilpotent } \},
 $$
 the group law being the multiplication 
 and 
 $$
 \hat{\G }\sb{ a \, L\n}(A) =\{ b\in A \, | \, b 
 \text{\it \/ is nilpotent} \}
 $$
 is the additive group 
 for a commutative $L\n$-algebra $A$.
 
\item Commutative deformation of \QSI extension $(\com(t,\, t^{\alpha}),\, \sigma, \, \theta^{\ast})/\com$. 
If $q$ is not a root of unity, 
$\infgal(L/k)$ is an infinite-dimensional formal group such that we have 
\[
 0 \rightarrow \widehat{A[[W\sb{1}]]^{\ast}} \rightarrow \infgal(L/k)(A) \rightarrow \hat{\G}\sb{m}(A) \rightarrow 0,
\]
where $\widehat{A[[W\sb{1}]]^{\ast}}$ denotes the multiplicative group 
\[
 \left\{ a \in A[[W\sb{1}]]\, \left| \, \text{\itshape all the coefficients of power series $a-1$ are nilpotent}
\right. \right\}
\]
%and $\hat{\G}\sb{m}(A)$ is the multiplicative group
%\[
%\{ a\in A \, |\, \text{\itshape $a-1$ is nilpotent} \},
%]
modulo Conjecture \ref{6.8j}. 
\item Non-commutative Galois group. If $q$ is not a root of unity, the Galois group ${\rm NC}\infgal (L/k)$ is isomorphic to the quantum formal group 
$\widehat{QG}\sb{II \, q}$:
\[
{\rm NC}\infgal (L/k) \simeq \widehat{QG}\sb{II \, q}
\]
modulo Conjecture \ref{6.8j}. 
\par
We should be careful about the group law. 
Quantum formal group structure in $\widehat{QG}\sb {II \,q }$ coincides with the group structure defined from the initial conditions as in Remark \ref{10.24a}. 
%So we might say that quantum Galois group is the quantum formal group $\widehat{QG}\sb{2 \, q}$:
%$$
%{\rm NC}\infgal (L/k) \simeq \widehat{QG}\sb{II\, q}.
%$$
\item 
Let us assume $q$ is not a root of unity.
If we have a $q$-difference field extension 
$(L, \, \sigma )/(k, \, \sigma )$ such that $t \in 
L$ with $\sigma (t) = qt$, then we 
can define the operator 
$\theta ^{(1)}\colon L \to L$ 
by setting 
$$\theta ^{(1)}(a) := \frac{\sigma (a) -a}{qt -t}.$$
We also assume the field $k$ is $\theta ^{(1)}$ invariant. 
Defining the operator 
 $
\theta ^{(n)}\colon L \to L
 $
 by 
 
\begin{align}
\theta^{(0)} & = \Id \\ 
 \theta ^{(n)} & = \frac{1}{[n]\sb q !}(\theta ^{(1)})^n
 \end{align}
 for every positive integer $n$ 
 so that we have a \QSI field extension 
 $(L, \, \sigma , \, \theta ^*)/(k, \, \sigma , \, \theta ^*)$.
 \par  
Here arises a natural question
 of comparing the Galois groups of 
 the difference field extension 
 $ (L, \, \sigma )/(k, \, \sigma )$
 and \QSI field extension 
 $(L, \, \sigma , \, \theta^* )/(k, \, \sigma , \, \theta ^* )$. 
 \par
 As the \QSI field extension is constructed from the 
 difference field extension in a more or less trivial way, 
 one might imagine that they coincide or they are not much different. 
 \par
 This contradicts Conjecture \ref{6.8j}. 
 Let us take our example $\com (t , \, t^\alpha )/\com $. 
 Assume Conjecture \ref{6.8j} is true. Then the 
 Galois group for the \QSI extension is 
 $\widehat{QG}\sb {II q L\n}$ that is infinite-dimensional,
 whereas 
the Galois group is of 
 the difference field extension is of 
 dimension $2$.
\end{enumerate}
%%%%%%%%%%%%%%%%%%%%

\section{The Third Example, the field extension 
$\com (t, \, \log\, t )/ \com $}\label{10.4c}
We assume that $q$ is a complex number not equal to $0$. 
Let us study the field extension $L/k:=\com(t, \, \log\, t)/\com$ from various view points as in Sections \ref{10.4a} and \ref{10.4b}. 
\subsection{$q$-difference field extension $\com(t,\, \log\, t)/\com$ }\label{6.8k}
We consider $q$-difference operator $\sigma\colon L \to L$ such that $\sigma$ is the $\com$-automorphism of the field $L$ satisfying 
\begin{equation}\label{6.7a}
 \sigma(t) = qt \ \text{\it \/ and } \ \sigma(\log \, t) = \log \,t + \log\, q . 
\end{equation}
\par
It follows from \eqref{6.7a} that
if $q$ is not a root of unity, then the field of constants of the difference field 
$( \com (t, \, \log\, t), \, \sigma )$ is $\com $ and hence 
 $(\com(t, \, \log\, t),\sigma)/ \com$ is a Picard-Vessiot extension with Galois group $\G\sb{m\,\com} \times\sb{\com} \G\sb{a \, \com}$. 
\subsection{Differential field extension $(\com(t, \, \log\, t), d/dt)/\com$ }
As we have
\[
 \frac{dt}{dt} = 1 \ \text{\it \/ and }\ \frac{d\log\, t}{dt} = \frac{1}{t}, 
\]
 both differential field extensions $\com(t,\, \log\, t)/\com(t)$ and $\com(t)/\com$ are Picard-Vessiot extensions with Galois group $\G\sb{a\, \com}$. The differential extension $\com(t, \, \log\, t)/\com$ is not, however, a Picard-Vessiot extension. Therefore, we need general differential Galois theory \cite{ume96.2} to speak of the Galois group of the differential field extension $\com(t, \, \log\, t)/ \com$. \\
The universal Taylor morphism 
\[
 \iota \colon L \rightarrow L\n[[X]]
\]
sends
\begin{align}
\iota(t) &= t+X,\label{6.7c}\\
\iota(\log\, t) &= \log\, t +\sum\sb{n=0}^{\infty}(-1)^{n+1}\frac{1}{n}\left( \frac{X}{t} \right)^{n} \in L\n[[X]] .\label{6.7d}
\end{align}
Writing $\log\, t$ by $y$, we take $\partial/\partial t,\, \partial/\partial y$ as a basis of $L\n=\com(t,y)\n$-vector space $\mathrm{Der}(L\n/k\n)$ of $k\n$-derivations of $L\n$. It follows from \eqref{6.7c}, \eqref{6.7d} that 
\[
 \eL = \text{ \it \/ a localization of the algebra }L\s \left[t+X, \, \sum\sb{n=1}^{\infty}(-1)^{n+1}\frac{1}{n}\left( \frac{X}{t} \right)^{n} \right] \subset L\s[[X]].
\]
We argue as in \ref{9.28e} and Section
 \ref{10.4b}.
 For a commutative algebra $A \in ob(CAlg/L\n)$ and $\varphi \in \mathcal{F}\sb{L/k}(A)$, there exist 
nilpotent elements 
$a, b \in A $ such that
\begin{align*}
 \varphi(t+X) &= t+ W\sb{1} + X + a, \\
 \varphi \left( \sum\sb{n=1}^{\infty}(-1)^{n+1}\frac{1}{n}\left( \frac{X}{t+W\sb{1}}\right)^{n} \right) 
 &= \sum\sb{n=1}^{\infty}(-1)^{n+1}\frac{1}{n}\left( \frac{X}{t+W\sb{1}+ a}\right)^{n} + b
.\end{align*}
Therefore 
we arrived at the dynamical system 
\begin{equation}\label{10.31a}
\left\{ \begin{array}{l} t, \\ y,  \end{array} \right. \mapsto 
 \left\{\begin{array} {l}
\phi (t) = t + X +W\sb 1 +a, \vspace{1ex}
 \\
\displaystyle{\phi (y) = y+ \sum\sb{n=1}^{\infty}(-1)^{n+1}\frac{1}{n}\left( \frac{X}{t+W\sb{1}+ a}\right)^{n} + b .}
 \end{array} \right. 
\end{equation}
In terms of initial conditions, dynamical system 
\eqref{10.31a} reads 
$$
\begin{pmatrix}
 t \\ y 
\end{pmatrix}  \mapsto \,
\begin{pmatrix}t +a \\ 
y+ b \end{pmatrix},
$$
where $a, \, b $ are nilpotent elements of $A$. 
So we conclude 
$$
\infgal (L/k)(A) = \hat{\G}\sb a (A) \times \hat{\G}\sb a (A)
$$
for every commutative $L\n$-algebra $A$.
Consequently we get 
\[
 \infgal(L/k) \simeq (\hat{\G}\sb{a \com} \times \hat{\G}\sb{a \com} ) \otimes\sb {\com} L\n . 
\]
\subsection{\QSI field extension $(\com(t,\, \log\, t),\, \sigma, \, 
\sigi , \, \, \theta^{\ast})/\com$ }
For the automorphism $\sigma \colon \com(t, \, \log\, t) \rightarrow \com(t, \, \log\, t)$ in Subsection \ref{6.8k} 
we set
 $$
 \theta^{(0)} = \Id\sb{\com(t, \, \log\, t)} \text{\it \/ and \/}
\theta^{(1)} = \frac{\sigma - \Id\sb{\com(t, \, \log\, t)}}{(q-1)t}
$$
so that $\theta^{(1)} \colon \com(t, \, \log\, t) \rightarrow \com(t, \, \log\, t)$ is a $\com$-linear map. We further introduce 
\[
 \theta^{(i)} := \frac{1}{[i]\sb{q}!}(\theta^{(1)})^{i}\colon \com(t, \, \log\, t) \rightarrow \com(t, \, \log\, t)
\]
that is a $\com$-linear map for $i = 1,\,2,\,3, \, \cdots$. 
Hence if we denote the set $\{ \theta^{(i)} \}\sb{i \in \N}$ by $\theta^{\ast}$, then $(\com(t,\, \log\, t),\, \sigma, \, \sigi , \, \theta^{\ast})$ is a \QSI ring. \par 
The universal Hopf morphism 
\[
\iota \colon \com(t, \, \log\, t) \rightarrow F(\Z, L\s)[[X]]
\]
sends, by Proposition \ref{a4.1}, $t$ and $y$ respectively to
\begin{align*}
 \iota(t) &= tQ + X,\\
 \iota(y) &= y + (\log\, q) Z + \frac{\log\, q }{q-1} \sum\sb{n=1}^{\infty} X^{n}(-1)^{n+1}\frac{1}{[n]\sb{q}q^{n(n-1)/2}
 }(tQ)^{-n}
\end{align*}
that we identify with 
\begin{align*}
 y + W\sb{2} + \, (\log\, q)\,Z + \frac{\log\, q }{q-1} \sum\sb{n=1}^{\infty} X^{n}(-1)^{n+1}\frac{1}{[n]\sb{q}q^{n(n-1)/2}}(t+W\sb{1})^{-n}Q^{-n}
\end{align*}
that is an element of $F(\Z, L^{\n}[[W\sb{1},W\sb{2}]])[[X]]$, where we set 
\[
Z := \begin{bmatrix}
\cdots&-1&0&1&2& \cdots \\
\cdots&-1&0&1&2& \cdots
\end{bmatrix} \in F(\Z, \Z)
. \]
In particular we have 
\begin{equation}\label{140306a}
\frac{\partial \iota (y)}{\partial W\sb 2}= 1.
\end{equation}
We identify further $t \in L^{\s}$ with $t + W\sb{1} \in L\n[[W\sb{1},W\sb{2}]]$ and hence
\[
 \sum\sb{n=1}^{\infty}X^{n}(-1)^{n+1}\frac{1}{[n]\sb{q} q^{n(n-1)/2}}(tQ)^{-n} \in F(\Z, L^{\s})[[X]]
\]
with 
\[
 \sum\sb{n=1}^{\infty}X^{n}(-1)^{n+1}\frac{1}{[n]\sb{q}q^{n(n-1)/2}}(t+W\sb{1})^{-n} Q^{-n}\in F(\Z,L\n[[W\sb{1},W\sb{2}]])[[X]]. 
\]
\subsubsection{Commutative deformations $\mathcal{CF}\sb{L/k}$ for $(\com(t,\, \log\, t),\, \sigma, \, \theta^{\ast})/\com$}
Now the argument of Section \ref{10.4b} allows us to describe infinitesimal deformations on the category of commutative $L\n$-algebras $(CAlg/L\n )$. Let $\varphi \colon \eL \rightarrow F(\Z, A[[W\sb{1},W\sb{2}]])[[X]]$ be an infinitesimal deformation of the canonical morphism $\iota \colon \eL \rightarrow F(\Z, A[[W\sb{1},W\sb{2}]])[[X]]$ for $A \in ob\, (CAlg/L\n)$. 
Then there exist $e \in A$
%and $b(W\sb{1}) \in A[[W\sb{1}]]$ 
 such that $e-1$
 is nilpotent and such that 
 % and all the coefficients of the power series $b(W\sb{1})$ are nilpotent and such that
\[
 \varphi\left( (t + W\sb{1}) Q + X \right) = e(t + W\sb{1})Q + X,
% [varphi (\sum\sb{n=1}^{\infty} X^{n}(-1)^{n+1}\frac{1}{[n]\sb{q}}(t+W\sb{1}))^{-n}Q^{-n}) &=\sum\sb{n=1}^{\infty} X^{n}(-1)^{n+1}\frac{1}{[n]\sb{q}}(e(t+W\sb{1}))^{-n}Q^{-n}) + b(W\sb{1}). 
\] 
as we learned in the First Example. 
To determine the image $\cZ:=\varphi (y)$, we argue as in the Second Example. 
We have 
\begin{align}\label{140305a}
\sigma (y) &= y + \log\, q, \\ \label{140305b}
\theta^{(1)}(y) &= \frac{\log \, q}{(q-1)t} .
\end{align}
Since the deformation $\varphi$ is \QSI morphism, the two equations above 
give us relations 
\begin{align}\label{140305c}
\hat{\Sigma} (\cZ) &= \cZ + \log\, q, \\ \label{140305d}
\hat{\Theta}^{(1)}(\cZ) &=
\frac{\log \, q}{(q-1)( (t +W\sb 1)e Q + X)} .
%%%%%%%%%%%
\end{align}
We determine the expansion of the element $\cZ$: 
$$
\cZ = \sum \sb {n=0} ^\infty X^n a \sb n \in F(\Z , \, A[[W\sb 1, \, W\sb 2 ]] )[[X]]
$$ 
so that 
$$
a\sb n \in F(\Z, \, A[[W\sb 1 , W\sb 2]]) \text{\it \/ for every \/} n\in \N.
$$
It follows from \eqref{140306a} and \eqref{140305c}
$$
a\sb 0 = y + W\sb 2 + b(W\sb 1 ) + (\log \, q) N \in F(\Z , \, A[[W\sb 1, W\sb 2 ]]) ,
$$
where $b(W\sb 1)$ is an element of $A[[W\sb 1]]$
such that all the coefficients of the power series $b(W\sb 1)$ are nilpotent.
On the other hand \eqref{140305d} tells us 
\begin{align}\label{4322a}
a\sb 1 &= \frac{\log \, q }{q-1} \, \frac{1}{(t+W\sb 1)eQ},\\ 
a\sb {n +1} &= - \frac{[n]\sb q}{[n+1]\sb q}\,\frac{1}{(t+W\sb 1)eQq^{n}}a\sb n 
\qquad \text{\it \/ for \/} n \ge 1. 
\end{align}
Hence 
\begin{equation} 
a\sb {n} = (-1)^{n+1}\left(\frac{\log q}{q-1}\right)\frac{1}{[n]\sb q q^{n(n-1)/2}}(t+ W\sb 1)^{-n}(eQ)^{-n} \qquad
\text{\it \/ for \/}n \ge 1. 
\end{equation}
So we get 
\begin{equation}\label{4322b}
\begin{split}
\cZ= &y + W\sb{2} + b(W\sb 1) +\, (\log\, q)\,N \\
&+\frac{\log q }{q-1} 
\sum\sb{n=1}^{\infty} X^{n}(-1)^{n+1}\frac{1}{[n]\sb{q} q^{n(n-1)/2}}(t+W\sb{1})^{-n}(eQ)^{-n}, 
\end{split}
\end{equation}
which is an element of $F(\Z, A[[W\sb{1},W\sb{2}]])[[X]]$.
\begin{proposition}\label{6.8m}
For every commutative $L\n$-algebra $A \in ob (CAlg/L\n )$, 
We have a functorial injection
\begin{multline*}
 \mathcal{CF}\sb{L/k}(A) \rightarrow \hat{G}\sb{3 }(A) :=
 \{ (e,\, b(W\sb{1})) \in A\times A[[W\sb 1]]\, | \, 
e-1 \text{ and } \\
\text{\it\/ all the coefficients 
 of $b(W\sb{1})$ are nilpotent}
 \} 
\end{multline*}
sending an element 
$$
\varphi \in \mathcal{CF}\sb {L/k}(A) \text{ to } (e, \, b(W
\sb 1) ) \in \hat{G}\sb{3 }(A).
$$ 
\end{proposition}
\begin{conjecture}\label{6.8l}
If $q$ is not a root of unity, then
the injection in Proposition \ref{6.8m} is a bijection. 
\end{conjecture}
$\hat{G}\sb{3 }$ is a group functor on $(CAlg/L\n)$. In fact, for $A\in ob(Alg/L\n)$, 
 we define the product of two elements
$$(e, \, b(W\sb{1})),\, (g, \, c(W\sb{1})) \in \hat{G}\sb{3 }(A)$$ 
 by 
\[
 (e,\, b(W\sb{1})) * (g, \, c(W\sb{1})) := (eg, \, b((g-1)t+gW_{1})+ c(W\sb{1})). 
\]
Then, 
the product is, in fact, an element of $\hat{G}\sb {3 }(A)$, the product is associative, 
the 
unit element of the group law is $(I\sb 2,\, 0) \in \hat{G}\sb{3 }
(A)$ and the inverse $(e, \, b(W\sb{1}))^{-1} = (e^{-1}, \, -b(e^{-1}W\sb{1}+ (e^{-1} -1)t))$. 
\par
So if Conjecture \ref{6.8l} is true, we have a splitting exact sequence 
\[
 0 \rightarrow A[[W\sb{1}]]\sb{+} \rightarrow \infgal(L/k)(A) \rightarrow \hat{\G}\sb{m\, L\n}(A) \rightarrow 1,
\]
where $A[[W\sb{1}]]\sb{+}$ denote the additive group of the power series in $A[[W\sb{1}]]$ whose coefficients are nilpotent element. 
%%%%%%%%%%%%%%%%%%%%%%%%%%%%%%%%%%%%%%%%%%%%%%%%%%%%%%%%%%%%%%%%%%%%%%%%%%%%%%%%%%%%%%%%%%%%%
\subsubsection{Non-commutative deformations $\NCF$ for $(C (t,\, \log\, t),\, \sigma,\, \sigi , \, \theta^{\ast})/\com$}

The arguments in Section \ref{10.4b} allows us to prove analogous results 
on the non-commutative deformations 
for the \QSI field extension $(\com(t,\, \log\, t),\, \sigma,\, \sigi , \, \theta^{\ast})/\com$. We write assertions without giving detailed proofs. For, since the proofs are same, it is easy to find complete proofs. 
\par 
As in the Second Example, doing calculations 
\eqref{4322a},..., \eqref{4322b} in the non-commutative case, we can determine the set
$\mathcal{NCF}\sb{(\com (t, \log t ), \sigma ,\sigi , \theta^*)/ \com}(A)
$.
 \begin{proposition}
 For a not necessarily commutative $L\n$-algebra $A
 \in ob (NCAlg/L\n)$, we can describe an infinitesimal deformation 
$$
 \varphi \in \mathcal{NCF}\sb{(\com (t, \log t ), \sigma ,\sigi , \theta^*)/ \com} (A).
 $$ 
 Namely putting $y:= \log \, t$, we have 
 \begin{align*} 
 \varphi (t) &= (e(t+ W\sb 1) + f)Q + X, \\
 \varphi (y) &= y + W\sb 2 + b(W\sb 1) + (\log q) Z \\
& + \frac{\log q}{q-1}\sum \sb {n= 1} ^\infty
 X^{n}(-1)^{n+1}\frac{1}{[n]\sb{q} q^{n(n-1)/2}}[e(t+W\sb{1})+ f]^{-n}Q^{-n}
\end{align*}
that are elements of 
$ F(\Z, A[[W\sb{1},W\sb{2}]])[[X]]$, where 
$e, \, f \in A$ and $b(W\sb 1 )\in A[[W\sb 1]] $ satisfying the following conditions.
\begin{enumerate}
\renewcommand{\labelenumi}{(\arabic{enumi})}
\item ef =qfe. 
\item $e - 1$ and $f$ are nilpotent elements of $A$.
\item All the coefficients of the power series $b(W\sb 1) $ are nilpotent. 
\item $[e( t + W\sb 1 ) + f, b(W\sb 1)] = 0$.
\end{enumerate}
 \end{proposition}
 The commutativity condition (4) comes from the commutativity relation between the elements $t$ and $y = \log t$ in the field $L$.
 \begin{definition}
We introduce a functor 
\[
\widehat{QG}\sb{3 \, q} \colon \NCA \rightarrow (Set)
\]
by setting 
\begin{multline*} 
\widehat{QG}\sb{3 \, q}(A):=\{(G,\, \xi (W\sb{1}))\in \gH\sb{q\, L\n}(A)\times A[[W\sb{1}]] \, | \, 
\text{\it\/ (1) $G = \begin{bmatrix}
e&f\\
0&1
\end{bmatrix} \in \widehat{\gH}\sb q (A)$ so that } \\
ef = qfe$, \, $e-1, f \in A \text{ are nilpotent. }
\text{\it\/(2) All the coefficients of $\xi (W\sb{1})$}\\
\text{\it\/ are nilpotent. }
\text{\it\/(3) $[e(t + W\sb{1})+f, \, \xi (W\sb{1})] = 0$.}
\}
\end{multline*}
\end{definition}
$\widehat{QG}\sb{3 \, q}$ is a quantum formal group. Namely, for 
$$
(G,\, \xi(W\sb{1})),\, (H, \, \eta(W\sb{1})) \in \widehat{QG}\sb{3 \, q}(A)
$$
 such that the two sub-sets 
\begin{align*}
&\text{\it 
\{all the entries of matrix $G$, all the coefficients of the power series $\xi (W\sb{1})$\},
} \\
&\text{\it
 \{all the entries of matrix $H$, all the coefficients of the power series $\eta (W\sb{1})$\}
 }
 \end{align*}
 of $A$ are mutually commutative, we define their product by 
\[
(G, \, \xi (W\sb{1}) ) * (H, \, \eta (W\sb{1})) := (GH, \, \xi ((g-1)t + h+gW_{1}) + \eta (W\sb{1})), 
%\widehat{QG}\sb{III \, q}(A),
 \]
 where 
 $$
 H=\begin{bmatrix} 
 g & h \\ 
 0 & 1 
  \end{bmatrix}.
 $$
Then, the argument of Lemma \ref{14320a} shows that the product of two elements is, in fact, an element in the set $\widehat{QG}\sb{3 \, q}(A)$ and the product is associative. The unit element is $(I \sb{2}, \, 0) \in \widehat{QG}\sb{3 \, q}(A)$. The inverse 
$$
(G, \, \xi (W\sb 1))^{-1}
= (G^{-1}, \, -
\xi 
(
(e^{-1} -1)t -e^{-1}f +e^{-1}W\sb 1
)
 \in \widehat{QG}\sb{3 \, q^{-1}}(A), 
$$
where
$$
G=\begin{bmatrix}
e & f \\
0 &1
\end{bmatrix}. 
$$
\begin{proposition}\label{6.8o}
We have a functorial injection
\[
\NCF(A) \rightarrow \widehat{QG}\sb{3 \, q}(A)
\]
that sends $\varphi \in \NCF(A)$ to $( \begin{bmatrix}
e & f \\
0& 1
\end{bmatrix}, \, b(W\sb{1})
)$. Here
\begin{align}
\varphi ((t+ W\sb{1})Q+X) = & (e(t+W\sb{1}) + f)Q + X,\\
\varphi (\iota (y)) = & 
\varphi (y + W\sb 2 + \log q N 
\nonumber \\
& + \frac{\log q}{q-1} \sum\sb{n=1}^{\infty}X^{n} (-1)^{n+1} \frac{1}{[n]\sb{q}q^{n(n-1)/2}}((t+W\sb{1}) ^{-n}Q^{-n} ) 
 ) \\
 \nonumber 
= & y+ W\sb 2 + b(W\sb 1) +\log q N \\
 & + \frac{\log q}{q -1}\sum\sb{n=1}^{\infty}X^{n} (-1)^{n+1} \frac{1}{[n]\sb{q}q^{n(n-1)/2}}(e(t+W\sb{1}) + f)^{-n}Q^{-n}.
\end{align}
\end{proposition}
We also have a Conjecture. 
\begin{conjecture}\label{6.8p}
If $q$ is not a root of unity, then 
the injection in Proposition \ref{6.8o} is a bijection. 
So
\[
\NCF \simeq \widehat{QG}\sb{3 \, q}.
\]
\end{conjecture}
%The Conjecture \ref{6.8p} allows us to determine the quantum 
%Galois group of the \QSI field extension 
%$\com ( t, \log t )/ \com $.
%\begin{proposition}%%%%%%%%%%%%%%
%If the conjecture above is true and if $q$ is not a root of unity, then
%the quantum Galois group of the \QSI extension 
%$\com (t, \log t)/\com $
%is the quantum formal group $\widehat{QG}\sb{3 \, q}$. 
%\end{proposition}
\begin{remark}
The argument in 
\ref{11.23a}
allows us to prove that Conjecture \ref{6.8p} implies Conjecture \ref{6.8l}.
\end{remark}
We can also define the quantum formal group $\widehat{QG}\sb {III \, q}$ in terms of non-commutative coordinate transformations as in the Second Example. 
\begin{definition}\label{4.5.2a}
We define a functor 
$$
\widehat{QG}\sb{III\, q}: (NCAlg/L\n) \to (Set) 
$$
by setting 
\begin{align*}
\widehat{QG}\sb{III\, q}(A): = \{ 
&\left( (e -1)t + f + eW\sb 1, 
\, W\sb 2 + b(W\sb 1) \right) \in A[[W\sb 1, W\sb 2 ]] \times 
 A[[W\sb 1, W\sb 2 ]] \\
 &| \, e-1, \, f, \text{ and all the coefficients of the power series }b(W\sb 1) \text{ are} \\
 &\text{ nilpotent 
 satisfying } ef =qfe, \, [ (e -1)t + f+eW\sb 1, 
\, W\sb 2 + b(W\sb 1) ]= 0 \}.
\end{align*}
We regard an element 
$$
\varphi = \left(\varphi\sb 1(W\sb 1, W\sb 2),\, 
\varphi\sb 2(W\sb 1, W\sb 2)\right) \in \widehat{QG}\sb{III\, q}(A)
$$
as an infinitesimal coordinate transformation $\Phi $
$$
(W\sb 1, \, W\sb 2) \mapsto \left(\varphi\sb 1(W\sb 1, W\sb 2),\, \varphi\sb 2(W\sb 1, W\sb 2)\right)
$$
with non-commutative coefficients. 
The product in the quantum formal group $\widehat{QG}\sb{III\, q}$ 
is the composition of coordinate transformations if they satisfy a commutation relation so that the product is defined. 
To be more concrete, 
let 
$$
\left( (e-1)t + f +eW\sb 1, 
\, W\sb 2 + b(W\sb 1) \right) \text{ and } 
\left( (g-1)t + h + gW\sb 1, 
\, W\sb 2 + c(W\sb 1) \right)
$$
be two elements of $\widehat{QG}\sb{III\, q}(A)$ such that the following two sub-sets of the ring $A$ is mutually commutative:
\begin{enumerate}
\renewcommand{\labelenumi}{(\arabic{enumi})}
\item $\{ \, e, \, f\, \}\, \cup$ the set of coefficients of the power series $b(W\sb 1)$,
\item 
$\{ \, g, \, h\, \}\, \cup$ the set of coefficients of the power series $c(W\sb 1)$,
\end{enumerate}
then the product is 
\begin{align*}
(
 (e-1)t + f +&eW\sb 1, 
\, W\sb 2 + b(W\sb 1) 
)
*
\left( 
 (g-1)t + h+gW\sb 1, 
\, W\sb 2 + c(W\sb 1) 
\right) \\
 &=\left( (eg -1)t + eh +f + egW\sb 1, 
\, W\sb 2+ b( (g-1)t + h + gW\sb 1) +c(W\sb 1) \right) 
\end{align*}
which is certainly an element of $\widehat{QG}\sb{III\, q}(A)$.
\end{definition}
 Though we reversed the procedure, the quantum formal group 
 $\widehat{QG}\sb {3\, q}$ arises from $\widehat{QG}\sb{III\, q}$ 
 and we arrived at the last object as a natural extension of Lie-Ritt functor in \cite{ume96.2} of coordinate transformations in the space of initial conditions. 
 \begin{proposition}\label{4325a}
 For every algebra $A \in (NCAlg/L\n)$, 
we have a functorial isomorphism of quantum formal 
groups
$$
\widehat{QG}\sb{3\, q}(A) \to \widehat{QG}\sb{III\, q}(A)
$$
sending an element $$( \begin{bmatrix}
e & f \\
0& 1
\end{bmatrix}, \, b(W\sb{1})
) \in \widehat{QG}\sb{3\, q}(A) \text{ to } ( (e-1)t +f + eW\sb 1, \, W\sb 2 + b(W\sb{ 1}))\in \widehat{QG}\sb{III\, q}(A).
$$
 \end{proposition}
Looking at Propositions \ref{6.8o}, \ref{4325a} and Conjecture \ref{6.8p}, 
we find that we are in the same situation 
as in 
\ref{11.23a}, 
where we studied non-commutative deformations of the 
Second Example. 
 \begin{theorem} %%
 We have an inclusion 
$$
\mathcal{NCF}\sb{L/k} \hookrightarrow \widehat{QG}\sb {III\, q } 
$$
of functors on the category $(NCAlg/L\n )$ taking values in the 
category of sets, where 
\begin{equation}\label{4.5.1x}
L/k = (\com (t, \, \log t ), \, \sigma ,\, \sigi ,\, \theta ^* )/ \com .
\end{equation}
If we assume Conjecture \ref{6.8p}, then the inclusion \eqref{4.5.1x} is bijection so that we can identify the functors 
$$
\mathcal{NCF}\sb{L/k} \simeq \widehat{QG}\sb {III\, q }. 
$$
The quantum formal group $\widehat{QG}\sb {III\, q }$ 
operates on the functor 
 $\mathcal{NCF}\sb{L/k}$
in an appropriate sense, 
through the initial conditions. 
(cf. The commutativity condition in Definition \ref{4.5.2a}.)
So we may say that 
the quantum formal Galois group 
$$
{\rm NC}\infgal (L/k) \simeq \widehat{QG}\sb {III\, q}. 
$$
 \end{theorem}
 
\subsection{Summary on the Galois structures of the field extension $\com(t,\, \log t)/\com$ }\label{151007c}
Let us summarize our results on the field extension $\com(t, \, \log t )/\com$. 
\begin{enumerate}
\renewcommand{\labelenumi}{(\arabic{enumi})}
\item Difference field extension $(\com(t,\, \log t),\, \sigma)/\com$. This is a Picard-Vessiot extension with Galois group $\G\sb{a\, \com} \times \G\sb{a\,\com}$. 
\item Differential field extension $(\com(t,\, \log t),\, d/dt)/\com$. This is not a Picard-Vessiot extension. The Galois group 
$$
\infgal(L/k)\colon (CAlg/L\n ) \to (Grp)
$$
 is isomorphic to $\hat{\G}\sb{a L\n} \times\sb{L\n} \hat{\G}\sb{a L\n}$, 
 where $\hat{\G}\sb{a L\n}$
 is the formal completion of the additive group. 
 So as a group functor on the category 
 $(CAlg/L\n)$, we have 
 $$
 \hat{\G }\sb{ a L\n}(A) =\{ b\in A \, | \, b 
 \text{ is nilpotent} \},
 $$ 
 the group law being the addition 
 and hence 
 $$
 \infgal (L/k)(A) = \{ (a,\, b )\, | \, \text{ $a, \, b$ are nilpotent elements of $A$ } \}
 $$ 
 for a commutative $L\n$-algebra $A$.
 
\item Commutative deformations of \QSI extension $(\com(t,\, \log t ), \sigma , \sigi , \theta^{\ast})/\com$. 
If $q$ is not a root of unity, 
$\infgal(L/k)$ is an infinite-dimensional formal group such that we have a splitting sequence 
\[
 0 \rightarrow A[[W\sb{1}]]\sb{+} \rightarrow \infgal(L/k)(A) \rightarrow \hat{\G}\sb{m}(A) \rightarrow 0,
\]
where $A[[W\sb{1}]]\sb{+}$ denotes the additive group 
\[
 \left\{ a \in A[[W\sb{1}]]\, \left| \, \text{ all the coefficients of power series $a$ are nilpotent}
\right. \right\}
\]
%and $\hat{\G}\sb{m}(A)$ is the multiplicative group
%\[
%\{ a\in A \, |\, \text{\itshape $a-1$ is nilpotent} \},
%\]
modulo Conjecture \ref{6.8p}. 
\item Non-commutative Galois group. If $q$ is not a root of unity, the Quantum Galois group 
${\rm NC}\infgal (L/k)$ 
 is isomorphic to a quantum formal group 
 $\widehat{QG} \sb{III\, q}$:
\[
{\rm NC}\infgal (L/k) \simeq \widehat{QG}\sb{III \, q}.
\]
modulo Conjecture \ref{6.8p}. 
\par
We should be careful about the group law. 
Quantum formal group structure in $\widehat{QG}\sb {III \,q }$ coincides with the group structure defined from the initial conditions as in Proposition \ref{4325a}. 
 %So we may say that non-commutative Galois group is the quantum formal group $\widehat{QG}\sb{3 \, q}$. 
 %%%%%%%%%%%%%%%%%%%%%%%%%%%%%%%%%%%%%%%%%
\end{enumerate}
\section{General scope of quantized Galois theory for \QSI field extensions}\label{14.5.29a}
 After we worked with three examples of \QSI field extensions 
$$
\com ( t )/\com, \quad \com ( t, \, t^\alpha )/\com \quad \text{\it \/ and \/ } \quad 
\com (t, \, \log t) /\com,
$$
there arises naturally, in our mind, the idea of formulating general quantized 
Galois theory for \QSI field extensions. 
The simplest differential Example \ref{14.5.29b} is also very inspiring. 
As we are going to see, it seems to work. 
%There is, however, 
%a subtle point: the definition of the quantum Galois group ${\rm NC}\infgal (L/k)$ for a \QSI field extension $L/k$. 
\subsection{Outline of the theory}
Let $L/k$ be a \QSI field extension such that the abstract field extension $L\n/k\n$ is of finite type. %We assume that the field of constants $C\sb k$ is the field $C$ ovver which the Hopf algebra is defined. We do not use this assumption.
 Galois theory for \QSI filed extensions is a particular case of Hopf Galois theory 
in Section \ref{4.5.5a}. So as we learned in \ref{4.5.5b},  
we have the universal Hopf morphism
$$
\iota : L \to F(\Z , \, L\n )[[X]].
$$
We choose a basis 
$$
\{ D\sb 1 , \, D\sb 2, \cdots , 
D\sb d \}
$$
of mutually commutative derivations of the $L\n$-vector space 
${\rm Der}(L\n/k\n)$ 
of $k\n$-derivations of $L\n$. We constructed the 
Galois hull $\eL /\K$ in Definition \ref{4.5.4b}. 
So we have the canonical morphism 
\begin{equation}\label{4.5.7a}
\iota : \eL \to F(\Z , \, L\n [[W\sb 1, \, W\sb 2, \cdots , W\sb d ]])[[X]].
\end{equation}
The rings $\eL$ and $\K$ are invariant under 
the set of operators 
\begin{equation}\label{4.5.5c}
\mathcal{D} := \{ \hat{\Sigma}, \, {\hat{\Theta} }^*, \, \frac{\partial}{\partial W\sb i } 
\, (1\le i \le d) \, \}
\end{equation}
on $F(\Z ,\, L\n [[W]] )[[X]]$. 
\par
In general, the Galois hull $\eL /\K$ is not commutative. 
So we measure it by infinitesimal deformations of the canonical morphism \eqref{4.5.7a} over the category $(NCAlg/L\n )$ of 
not necessarily commutative $L\n$-algebras.
We set in 
 Definition \ref{4.5.7b}
 \begin{align*}
 \mathcal{NCF}\sb {L/k}(A) = \{ \varphi :& \eL \to F(\Z , \, A[[W]])[[X]] \, |\, \varphi \text{\it \/ is an 
infinitesimal \/}\\
 &\text{\it \/ deformation over $\K$ 
 compatible with $\mathcal{D}$
 of canonical morphism \eqref{4.5.7a}} \} 
 \end{align*}
 so that we get the functor
 $$
 \mathcal{NCF}\sb {L/k} : (NCAlg/L\n ) \to (Set).
 $$
 Now we compare the differential case and \QSI case to understand their similarity and difference. 

\begin{enumerate}
\renewcommand{\labelenumi}{(\arabic{enumi})} 
\item Differential case 
\begin{enumerate}
\renewcommand{\labelenumi}{(\arabic{enumi})}
\item The Galois hull $\eL /\K$ is an extension of commutative algebras.
\item It is sufficient to consider commutative deformation functor $\mathcal{F}\sb {L/k}$ of the Galois hull $\eL/\K$ over the category $(CAlg/ L\n )$ of commutative $L\n $-algebras. 
\item The Galois group $\infgal (L/k) $ is a kind of generalization of algebraic group. In fact, it is at least a group functor on the category $(Calg/ L\n)$. 
\item Indeed the group functor 
$\infgal (L/k)$ 
 is given as the functor of automorphisms of the Galois hull $\eL /\K$. 
\end{enumerate}
\item \QSI case
\begin{enumerate}
 \renewcommand{\labelenumi}{(\arabic{enumi})}
 %\begin{enumerate}
 %\renewcommand{\labelenumi}{(\arabic{enumi})}
\item 
Galois hull $\eL/\K$ is not always an extension of commutative algebras. 
\item We have to consider the non-commutative deformation functor 
$\mathcal{NCF} \sb {L/k}$ over the category $(NCAlg/L\n )$
of not necessarily commutative $L\n$-algebras. 
\item The Galois group should be a quantum group that we can not interpret in terms of group functor. 
\end{enumerate}
\end{enumerate}
The comparison above shows that we have to find a counterpart of (d) in the \QSI case. The three examples suggest the following solution. 
\par 
{\bf Solution that we propose.} Let $y\sb 1 , \, y\sb 2 ,\cdots , y\sb d $ be a transcendence basis of the abstract field extension $L\n /k\n$. We set by 
$$
\iota (y\sb i ) = Y\sb i (W\sb 1,\, W\sb 2, \cdots , W\sb d ; X)
\in F(\Z , \, L\n [[W]])[[X]] \text{\it \/ for \/} 1\le i \le d.
$$ 
\begin{questions} \label{14.5.31a}
%\begin{emumerate}
%\item 
(1) For an $L\n$-algebra $A \in ob(NCAlg /L\n)$, let 
$$
f: \eL \to F(\Z , \, A[[W]])[[X]]
$$
be an infinitesimal deformation of the canonical morphism 
$\iota$. Then there exist an 
infinitesimal coordinate transformation 
$$
\Phi = (\varphi\sb 1 (W),\, \varphi \sb 2 (W), \, \cdots , \varphi \sb d (W) ) \in A[[W\sb 1, \, W\sb 2, \, \cdots , W\sb d ]]^d 
$$
with coefficients in the not necessarily commutative algebra $A$ 
such that 
$$
f(Y\sb i) = Y\sb i (\Phi (W); \, X) \text{ for every } 1\le i \le d.
$$
(2) Assume that Question (1) is affirmatively answered. Then we have a functorial morphism 
\begin{align}\label{4.5.8a}
\mathcal{NCF}\sb{L/k}&(A) \to \notag 
\\
\{ \Phi &\in A[[W]]^d \, | \, W \mapsto \Phi (W) \text{ is an infinitesimal coordinate transformation}\}
\end{align}
sending $f$ to $\Phi$ using the notation of (1). We set 
$$
{\rm Q}\infgal (L/k) (A) := \text{ the image of map \eqref{4.5.8a}} 
$$
%\mathcal{D}$ 
%$$
so that 
$$
{\rm Q}\infgal (L/k) : (NCAlg/L\n)
 \to \, (Set)
 $$ 
 is a functor. 
 Our second question is if the functor ${\rm Q}\infgal (L/k)$
is a quantum formal group. 

\noindent{(3)} Assume that Question (1) 
has an affirmative answer. Since the identity transformation is 
in ${\rm Q}\infgal (L/k)$, Question (2) reduces to the following 
concrete question. Let 
$f, \, g$ be elements of $\mathcal{NCF}\sb {L/k}(A)$ and let $\Phi$ and $\Psi$ be the corresponding coordinate transformations to $f$ and $g$ respectively. If the set of the coefficients of 
$\Phi$ and the set of the coefficients of $\Psi$ is mutually commutative, then does the composite coordinate transformation 
$\Phi \circ \Psi$ arise from an infinitesimal deformation 
$h \in \mathcal{NCF}\sb{L/k}(A)$?

\end{questions}
\par %
In view of Corollary \ref{140927b}, the universal deformation or the universal coaction seems to solve the Questions. It seems that we are very close to the solutions.

%\end{enumerate}

%%%%%%%%%%%%%%%%%%%%
%%%%%%%%%%%%%%%%%%%%%%%%%%%%%%%%%%%%%%%%%%%%%%%%%%%%%%%%
\part{Quantization of Picard-Vessiot theory}\label{14.6.24b}
%\section{Introduction to the second part} \label{14.6.24a} 
Keeping the notation of the Part I, we denote by $C$ a field of characteristic $0$. 
\par 
As we explained in the introduction, we believed for a long time that it was impossible to 
quantize Picard-Vessiot theory, Galois theory for linear difference or differential equations. 
Namely, there was no Galois theory for linear difference-differential equations, of which the Galois group is a quantum group that is, in general, neither commutative nor co-commutative. 
Our mistake came from a misunderstanding of preceding 
works, Hardouin \cite{har10} and Masuoka and Yanagawa \cite{masy}. 
They studied linear \QSI equations, \itqsi equations for short, 
 under two assumptions 
on \itqsi base field $K$ and 
\itqsi module $M$: 
\begin{enumerate} 
\renewcommand{\labelenumi}{(\arabic{enumi})} 
 \item 
 The base field $K$ contains $C (t)$. 
 \item 
 On the $K [\sigma , \, \sigi , \, \theta^* ]$-module $M$ the equality 
\[
\theta ^{(1)} = \frac{1}{(q-1)t}( \sigma - \Id \sb M).
\]
\end{enumerate}
holds. 
Under these conditions, 
 their Picard-Vessiot extension is 
realized in the category of commutative \itqsi algebras.
The second assumption seems too restrictive as clearly explained in \cite{masy}. 
If we drop one of these conditions, there are many linear 
\itqsi equations whose Picard-Vessiot ring is not commutative and the Galois group is a quantum group
that is neither commutative nor co-commutative. 
\par 
We analyze only one favorite example \eqref{176} over the base field $C $ in detail, 
 which is equivalent to the non-linear equation 
in Section \ref{10.4a}.
We add three more example in Section \ref{141223a}. 
Looking at these examples, the reader's imagination would go far away, as 
Cartier \cite{car13} did it for every qsi linear equation with constant coefficients. 
\par 
In the favorite example, we have a Picard-Vessiot ring
$R$ that is non-commutative, simple \itqsi ring 
(Observation \ref{obs3} and Lemma \ref{1016a}). 
The Picard-Vessiot ring $R$ is a torsor of a quantum group (Observation \ref{obs5}). 
As for equivalence of rigid tensor categories, we note Expectation \ref{160808a} and prove a modified version in Part \ref{part_III}. 
We have the imperfect Galois correspondence (Observation \ref{1018c}). 
We prove the uniqueness of the Picard-Vessiot ring for certain Examples in Section \ref{141223b}. 
Picard-Vessiot ring is, however, not unique in general as we see in Section \ref{sec:non-uniqueness}. 
 \par 
 We are grateful to K.~Amano for 
useful discussions. 
%%%%%%%%%%%%%%%%%%%%%%%%%%%%%%%%%%%%%%%%%
\section{Field extension $C (t) /C $ from classical and quantum view points}\label{14.6.25a}%\label{140116a} 
In Section \ref{10.4a}, %of the previous paper \cite{saiume13.2},
we studied a non-linear 
\QSI 
 equation, which we call 
\itqsi equation for short, 
\begin{equation}\label{171}
\theta^{(1)}(y) =1,\qquad \sigma(y) =qy, \qquad 
 \sigi (y) = q^{-1}y, 
\end{equation}
where $q$ is an element of the field $C$ not equal to $0$ nor $1$. 
Let $t$ be a variable over the constant base field $C$. 
We assume to simplify the situation that $q$ is not a root of unity. 
We denote by $\sigma \colon C (t) \to C (t)$ the $C $-automorphism of the field $C (t)$ 
of rational functions 
sending $t$ to $qt$. 
We introduce the 
$C $-linear operator $\theta^{(1)} \colon C (t) \to C (t)$
 %as a $\com$-linear map 
 %given 
 by 
\[
\theta^{(1)}\left( f(t)\right) := \frac{f(qt)-f(t)}{(q-1)t} \quad\text{\it \/ for every $f(t) \in C (t)$. }
\]
We set
\[
\theta^{(m)} := 
\begin{cases}
\Id_{C (t)}& \text{\it \/ for $m=0$}\\
\frac{1}{[m]_{q}!}\left(\theta^{(1)}\right)^{m} & \text{\it \/ for $m=1,\,2,\,\cdots$. }
\end{cases}
\]
As we assume that $q$ is not a root of unity, 
the number $[m]\sb q$ 
in the formula is 
not equal to $0$ and hence 
the formula determines the family $\theta ^{*}=
\{ \theta ^ {(i)}\, | \, i \in \N \, \}$ of operators. 
So $(C (t),\, \sigma , \, \sigi ,\, \theta^{\ast})$ is a \itqsi field. See Section \ref{10.4a} and $y=t$ is a solution for system \eqref{171}.

The system \eqref{171} is non-linear in the sense that for two solutions $y_{1},\, y_{2}$ of 
\eqref{171}, a $C $-linear combination $c_{1}y_{1}+c_{2}y_{2}\,(c_{1},\,c_{2} \in C )$ is not a solution of the system in general. 

However, the system is very close to a linear system. 
To illustrate this, let us look at the differential field extension $(C (t), \partial \sb t )/ (C , \, \partial \sb t )$, where we denote the derivation 
$d/dt$ by $\partial_{t}$ 
%We consider the differential field $(\com(t),\,\partial_{t})$, where $t$ is a variable over the complex number field $\com$ and $\partial_{t} = d/dt$ is the differential operator. 
The variable $t \in C(t)$ satisfies a non-linear differential equation
\begin{equation}\label{172}
\partial_{t} t -1 = 0.
\end{equation}
The differential field extension $(C (t),\, \partial_{t})/(C ,\,\partial_{t})$ is, however, the Picard-Vessiot extension for the linear differential equation
\begin{equation}\label{173}
\partial^{2}_{t} t = 0.
\end{equation}
To understand the relation between \eqref{172} and \eqref{173}, we introduce the 
$2$-dimensional 
$C $-vector space 
\[
E := C t \oplus C \subset C [t].
\]
The vector space $E$ is closed under the action of the 
 derivation $\partial \sb t$ so that $E$ is a 
 $C [\partial\sb t]$-module.
%that is a $\com[\partial_{t}]$-module such that the dimension 
%of the module $E$ as a $\com$-vector space 
%is $2$. 
Solving the differential equation associated with the $C [\partial_{t}]$-module $E$ is to find a differential algebra $(L,\, \partial_{t})/C $ and 
a $C [\partial_{t}]$-module morphism
\[
\varphi \colon E \rightarrow L. 
\]
Writing $\varphi(t) = f_{1},\, \varphi(1)=f_{2}$ that are elements of $L$, we have 
\[
\begin{bmatrix}
\partial_{t} f_{1}\\
\partial_{t} f_{2}
\end{bmatrix} = \begin{bmatrix}
0&1\\
0&0
\end{bmatrix} \begin{bmatrix}
f_{1}\\
f_{2}
\end{bmatrix}.
\]
Since $\partial_{t} t = 1,\, \partial_{t} 1= 0$, 
in the differential field $(C (t),\, \partial_{t})/C $, we find two 
solutions
$%\begin{bmatrix}
^ t (t , \, 1)
%\end{bmatrix}
$ and $
%\begin{bmatrix}
^t (1, \, 0 )
%\end{bmatrix}
$ 
that are two column vectors in $C (t)^2$ satisfying 
\begin{equation}\label{175}
\partial_{t} \begin{bmatrix}
t & 1\\
1 & 0
\end{bmatrix}=\begin{bmatrix}
0 & 1\\
0 & 0
\end{bmatrix}
\begin{bmatrix}
t & 1\\
1 & 0
\end{bmatrix}
\end{equation} 
and 
$$
\begin{vmatrix}
t & 1\\
1 & 0
\end{vmatrix}
\not= 0.
$$
Namely, $C (t)/C $ is the Picard-Vessiot extension for linear differential equation \eqref{175}. 
\par
We can argue similarly for the \itqsi field extension $(C (t),\, \sigma,\, \sigi,\, \theta^{\ast})/C $. 
You will find a subtle difference between the differential case and the
 \itqsi case. 
Quantization of Galois group arises from here.
\par 
Let us set 
$$
M = C t \oplus C \subset C [t]
$$
 that is a $ C [\sigma,\, \sigi , \, \theta^{\ast}
 ]$-module. 
Maybe to avoid the confusion that you might have in Remark \ref{mas} below, 
writing $m\sb 1 = t$ and $m\sb 2 =1$, 
 we had better define formally 
 $$
 M = C m\sb 1 \oplus C m\sb 2
 $$ 
 as a $C $-vector space on which $\sigma$ and $\theta ^{(1)}$ operate by 
\begin{equation} 
%\begin{align}
\begin{bmatrix}
\sigma (m_{1}) \\ 
\sigma (m_{2})
\end{bmatrix} = \begin{bmatrix}
q & 0\\
0 & 1
\end{bmatrix}\begin{bmatrix}
 m_{1}\\
 m_{2}
\end{bmatrix}, \, 
\begin{bmatrix}
\sigi (m_{1}) \\ 
\sigi (m_{2})
\end{bmatrix} = \begin{bmatrix}
q^{-1} & 0\\
0 & 1
\end{bmatrix}\begin{bmatrix}
 m_{1}\\
 m_{2}
\end{bmatrix}, \,
     \label{176}
\begin{bmatrix}
\theta^{(1)}( m_{1})\\
\theta^{(1)} (m_{2})
\end{bmatrix} = \begin{bmatrix}
0 & 1\\
0 & 0
\end{bmatrix}\begin{bmatrix}
 m_{1}\\
 m_{2}
\end{bmatrix}. %\label{177}
%\end{align}
\end{equation}
Since in \eqref{176} the first equation is equivalent to the second, we consider the first and third equations. 
Solving $ C [\sigma,\, \sigi , \, \theta^{\ast}]$-module $M$ is equivalent to 
find elements $f\sb 1 , \, f\sb 2 $ in a \itqsi algebra 
$(A, \, \sigma ,\, \sigi, \, \theta ^ * )$ satisfying 
the system of linear difference-differential equation
\begin{equation}
\begin{bmatrix}
\sigma (f_{1})\\
\sigma (f_{2})
\end{bmatrix} = \begin{bmatrix}
q & 0\\
0 & 1
\end{bmatrix}\begin{bmatrix}
 f_{1}\\
 f_{2}
\end{bmatrix}, \qquad \label{176b}
\begin{bmatrix}
\theta^{(1)}( f_{1})\\
\theta^{(1)} (f_{2})
\end{bmatrix} = \begin{bmatrix}
0 & 1\\
0 & 0
\end{bmatrix}\begin{bmatrix}
 f_{1}\\
 f_{2}
\end{bmatrix} 
\end{equation}
in the \itqsi algebra $A$.

%%%%%Lemma 
\begin{lemma} \label{lem1}
Let $(L,\, \sigma,\, \sigi,\, \theta^{\ast})/C $ be a \itqsi field extension. If a $2\times 2$ matrix $Y=(y_{ij}) \in M_{2}(L)$ satisfies 
a system of difference-differential equations 
\begin{equation}
\sigma Y = \begin{bmatrix}
q & 0\\
0 & 1
\end{bmatrix}Y \textit{ and } %\qquad
\label{1720} 
\theta^{(1)}Y = \begin{bmatrix}
0 & 1\\
0 & 0
\end{bmatrix}Y, %\label{1721}
\end{equation}
then $\det Y = 0$. 
\end{lemma}
\begin{proof}
It follows from \eqref{1720} 
\begin{equation}\label{1722}
\sigma(y_{11}) = qy_{11},\; \sigma(y_{12})=qy_{12},\; \sigma(y_{21}) = y_{21},\; \sigma(y_{22}) = y_{22}
\end{equation}
and 
\begin{equation}\label{1723}
\theta^{(1)}(y_{11}) = y_{21},\, \theta^{(1)}(y_{12})= y_{22},\, \theta^{(1)}(y_{21}) = 0,\, \theta^{(1)}(y_{22}) = 0. 
\end{equation}
It follows from \eqref{1722} and \eqref{1723}
\begin{equation} \label{1724}
\theta^{(1)}(y_{11}y_{12}) = \theta^{(1)}(y_{11})y_{12} + \sigma(y_{11})\theta^{(1)}(y_{12}) = y_{21}y_{12} + qy_{11}y_{22}
\end{equation}
and similarly 
\begin{equation}\label{1725}
\theta^{(1)}(y_{12}y_{11}) = y_{22}y_{11} + qy_{12}y_{21}. 
\end{equation}
As $y_{11}y_{12}=y_{12}y_{11}$, equating \eqref{1724} and \eqref{1725}, we get
\[
(q-1)(y_{11}y_{22}-y_{12}y_{21}) = 0
\]
so that $\det Y = 0$. 
\end{proof}
\begin{cor}\label{cor2}
Let $(K,\, \sigma,\, \sigi,\, \theta^{\ast})$ be a \itqsi field over $C $. Then the \itqsi linear equation
\begin{equation}\label{1011a}
\sigma Y = \begin{bmatrix}
q & 0\\
0 & 1
\end{bmatrix}Y \, \text{\it \/ and }\, 
\theta^{(1)}Y = \begin{bmatrix}
0 & 1\\
0 & 0
\end{bmatrix}Y
\end{equation}
has no \itqsi Picard-Vessiot extension
$L/K$ in 
 the following sense. There exists a solution 
$Y \in \GL \sb 2 (L)$ to \eqref{1011a} such that the abstract field $L$ is generated by the entries of the matrix $Y$ over $K$. The field of constants of the \itqsi over-field $L$ coincides with 
the field of constants of the base 
field $K$. 
 \end{cor}
\begin{proof}
This is a consequence of Lemma \ref{lem1}.
\end{proof}
\begin{remark}\label{mas} 
We note 
that 
 Corollary \ref{cor2} is compatible with Remark 4.4 and Theorem 4.7 of Hardouin \cite{har10}. See also Masuoka and Yanagawa \cite{masy}. 
They assure 
the existence of Picard-Vessiot extension for a 
$K[ \sigma , \, \sigi,\, \theta ^ *]$-module $N$ if the following two conditions are satisfied; 
\begin{enumerate}
\renewcommand{\labelenumi}{(\arabic{enumi})} 
 \item 
 The \itqsi base field $K$ contains $(C (t), \, \sigma ,\sigi,\, \, \theta^* )$, 
 \item
 The operation of $\sigma$ and $\theta^{(1)} $ on the module $N$ 
 as well as on the base field $K$, 
 satisfy the relation 
$$
\theta ^{(1)} = \frac{1}{(q-1) t}(\sigma - \Id \sb N).
$$ 
\end{enumerate}
In fact, even if the base field $K$ contains 
$(C (t), \, \sigma , \,\sigi,\, \theta ^ * )$, in 
$K\otimes \sb C M$, we have by definition of 
the 
$ C [\sigma , \, \sigi , \, \theta ^*]$-module $M$, 
$$
\theta ^{(1)} (m\sb 1 ) = m\sb 2 \not= 
\frac{1}{t}m\sb 1 = 
\frac{1}{(q-1)t}
\left( \sigma (m\sb 1) - m\sb 1\right) .
$$ 
So 
$K\otimes \sb C M$
does not satisfy the second condition above.

\end{remark}
%%%%%%%%%%%%%%%%%%%%%%%%%%%%%%%%%%%%%%%%%%%%%%%%%%%
\section{Quantum normalization of $( C (t) ,\, \sigma , \,
\sigi,\, \theta ^{*})/ C $}\label{14.6.25b}
We started from the \itqsi field extension $C (t)/C $. The column vector 
%%%%%%176 177
$^t(t,\, 1) \in C (t)^{2}$ is a solution to the system of equations \eqref{176}, i.e. we have
\[
\begin{bmatrix}
\sigma (t)\\
\sigma (1)
\end{bmatrix} = \begin{bmatrix}
q & 0\\
0 & 1
\end{bmatrix}\begin{bmatrix}
t\\
1
\end{bmatrix}, \\
\qquad
\begin{bmatrix}
\theta^{(1)} (t)\\
\theta^{(1)} (1)
\end{bmatrix} = \begin{bmatrix}
0 & 1\\
0 & 0
\end{bmatrix}\begin{bmatrix}
t\\
1
\end{bmatrix} .
\]
By 
applying to the \itqsi field extension $( C (t) , \sigma , \,\sigi ,\, 
 \theta ^{*})/ C $, 
the general procedure of 
\cite{ume96.2}, 
\cite{hei10} that is believed to lead us to the normalization,
we arrived at the Galois hull 
 $\eL = C (t)\langle Q,\, Q^{-1},\, X\rangle \sb{alg}$ modulo localization. 
This suggests an appropriate model of the non-commutative 
\itqsi ring extension $C (t)\langle Q,\, Q^{-1}\rangle\sb{alg}/C $ is a (maybe the), \itqsi Picard-Vessiot ring of the system of equations \eqref{176}. More precisely, $Q$ is a variable over $C (t)$ satisfying the commutation relation
\[
Qt=qtQ. 
\]
We understand $R=C \langle t,\, Q,\,Q^{-1} \rangle \sb{alg}$ as a sub-ring of 
$$
S=C [[t,\, Q]][t^{-1},\,Q^{-1}].
$$ 
We know that, in the previous line, 
the usage of $\langle \, \rangle \sb{alg}$ is more logical than $[ \; ]$, but as it is too heavy, we do not adopt it. 
The ring $S$ is a non-commutative \itqsi algebra by setting 
\[
\sigma(Q) = qQ,\, \theta^{(1)}(Q) = 0\, \textit{ and }\, \sigma(t) = qt,\, \theta^{(1)}(t) = 1
\]
and $R=C \langle t,\, Q,\,Q^{-1}\rangle \sb{alg}$ is a \itqsi sub-algebra. 
Thus we get a \itqsi ring extension 
$$
(R,\, \sigma , \, \sigi ,\, \theta^{\ast})/C = (C \langle t,\, Q,\, Q^{-1}\rangle \sb{alg},\, \sigma , \, \sigi ,\, \theta^{\ast})/C .
$$ 
We examine that $(R,\, \sigma , \, \sigi ,\, \theta^{\ast})/C $ is a non-commutative Picard-Vessiot ring for the systems of equations \eqref{176}. 
\begin{observation}\label{obs1}
The left $R$-module $M$ has two solutions in the 
\itqsi ring $R$ 
linearly independent over $C $. In fact, setting
\begin{equation}
Y:= \begin{bmatrix}
Q & t \\
0 & 1
\end{bmatrix} \in M_{2}(R), 
\end{equation}
we have 
\begin{equation}\label{140929a}
\sigma Y = \begin{bmatrix}
q & 0\\
0 & 1
\end{bmatrix}Y \textit{ and }
\theta^{(1)}Y = \begin{bmatrix}
0 & 1\\
0 & 0
\end{bmatrix}Y. 
\end{equation}
So the column vectors $^{t}(Q,\, 0),\, ^{t}(t,\, 1) \in R^{2}$ are $C $-linearly independent solution of the system of equations \eqref{176}. 
\end{observation}
\begin{observation}\label{obs2}
The ring $R= C \langle t,\, Q,\, Q^{-1}\rangle \sb{alg}$ has no zero-divisors. We can consider the ring $K$ of total fractions of $R=C \langle t,\, Q,\, Q^{-1}\rangle\sb{alg}$. 
\end{observation}
\begin{proof}
In fact, we have $R \subset C [[t,\, Q]][t^{-1},\,Q^{-1}]$. In the latter ring every non-zero element is invertible. 
\end{proof}
\begin{observation}\label{obs3}
Let $K$ be the ring of total fractions of $R$. The ring of \itqsi constants $C_{K}$ coincides with $C $. The ring of $\theta^{\ast}$ constants of $C [[t,\, Q]][t^{-1},\,Q^{-1}]$ is $C (Q)$. Moreover 
as we assume that $q$ is not a root of unity, 
the ring of $\sigma$-constants of $C (Q)$ is equal to $C $. 
\end{observation}
\begin{lemma}\label{1016a}
The non-commutative \itqsi algebra $R$ is simple. There is no \itqsi bilateral ideal of $R$ except for the zero-ideal and $R$.
\end{lemma} 
\begin{proof}
Let $I$ be a non-zero \itqsi bilateral ideal of $R$. 
We take an element 
$$
 0 \not= f:= a\sb 0 + ta\sb 1 + \cdots + t^n\, a\sb n \in I,
$$
where $a\sb i \in C [ Q, \, Q^{-1}]$ for $ 0 \le i \le n$. We may assume $a\sb n \not= 0$. 
Applying $\theta^{(n)}$ to the element $f$, we conclude that 
$ 0 \not= a\sb n \in C [Q, \, Q^{-1} ]$ is in the ideal $I$. Multiplying a monomial $bQ^l$ with $b\in C $, we find 
a polynomial 
$h = 1 + b\sb 1 Q + \cdots + b\sb s Q ^s \in C [Q]$ 
with $b\sb s \not= 0$ 
is in the ideal $I$. 
We show that $1$ is in $I$ by induction on $s$.
If $s\, =\, 0$, then there is nothing to prove. 
Assume that the assertion is proved for $ s \le m$. We have to show the assertion for $s = m+1$. 
Then,
since $Q^i$ is an eigenvector of the operator 
$\sigma$ with 
eigenvalue $q^i$ for $i \in \N$, 
$$
\frac{1}{q^{m+1} -1}(q^{m+1}h \, - \, \sigma (h)) 
= 1 + c\sb 1 Q + \cdots + c\sb {m}Q^m \in C [Q]
$$
is an element of $I$ and by induction hypothesis 
$1$ is in the ideal $I$. 
%applying the operators 
%$\sigma ^ j $ to $h$ for $ 0 \le j \le s$, 
%we can express $1$ as a $\com$-linear 
%combination of the $\sigma ^i (h)$'s 
%$( \, 0\, \le i \, \le s )$. 
%and hence 
%$1 \in I$. 
\end{proof} 
\begin{observation}\label{obs4}
The extension $R/C $ trivializes the $ C [\sigma,\, \sigi , \, \theta^{\ast}]$-module $M$. Namely, there exist \itqsi constants $c_{1},\,c_{2} \in R \otimes_{C }M$ such that there exists a left $R$-\itqsi module isomorphism
\[
 R \otimes_{C }M \simeq Rc_{1}\oplus Rc_{2}. 
\]
\end{observation}
\begin{proof}
In fact, it is sufficient to set
\[
c_{1} := Q^{-1}m_{1}-Q^{-1}tm_{2},\,\qquad c_{2} := m_{2}. 
\]
Then 
\[
\sigma(c_{1}) = c_{1},\, \qquad \sigma(c_{2}) = c_{2},\, \qquad \theta^{(1)}(c_{2}) = 0
\]
and
\[
\theta^{(1)}(c_{1}) = q^{-1}Q^{-1}\theta^{(1)}(m_{1})
-q^{-1}Q^{-1}m_{2}= q^{-1}Q^{-1}m_{2}
-q^{-1}Q^{-1}m_{2} = 0. 
\]
So we have an $(R,\, \sigma , \, \sigi ,\, \theta^{\ast})$-module isomorphism $ R \otimes_{C }M \simeq Rc_{1}\oplus Rc_{2}$. 
\end{proof}
\begin{observation}\label{obs5} %%%%%
The Hopf algebra $\mathfrak{H}_{q}= C \langle u,\, u^{-1},\, v\rangle$ with $uv=q\,vu$ co-acts from right on the non-commutative algebra 
$R$. Namely, we have an algebra morphism 
\begin{equation} \label{14.7.2a}
R \to 
R \otimes \sb C \mathfrak{H}\sb q
\end{equation}
sending 
\[
t \mapsto t\otimes 1 \, +\, Q \otimes v,\qquad Q\mapsto Q \otimes u, \qquad Q^{-1} \mapsto Q^{-1}\otimes u^{-1}.
\] 
Morphism \eqref{14.7.2a} is compatible with $ C [\sigma , \, \sigi , \, \theta^{(1)}]$-module structures, where $\sigma, \, \sigi $ and $\theta^{(1)}$ operate on 
the Hopf algebra 
$\gH \sb q$ trivially. 
\end{observation}
%%%%%%%%%%%%%%%%%%%
We can prove the assertion of Observation \ref{obs5} by a simple direct calculation, which is very much unsatisfactory. 
For, we are eager to know where the Hopf algebra $\gH\sb q $ comes from. We answer this question in two steps:
\begin{enumerate}
\renewcommand{\labelenumi}{(\arabic{enumi})} 
 \item 
 Characterization of the non-commutative algebra $\gH \sb q$.
 \item 
 Origin of the co-multiplication structure on the Hopf algebra $\gH \sb q$. 
 \end{enumerate}
 We answer question (1) in Corollary \ref{140927b}, 
and question (2) in Observation \ref{140927a}. To this end, we admit the algebra structure of $\gH \sb q$ and characterize it. 

Let us first 
fix some notations. 
For a not necessarily commutative $C $-\itqsi 
algebra $T$ and for a morphism 
$\varphi : R \to T$ of \itqsi algebras over $C $, we set 
$$
\varphi (Y) = 
\begin{bmatrix} 
\varphi (Q) & \varphi (t) \\
 0 &   1
 \end{bmatrix}.
$$
So $\varphi (Y) $ is an invertible
element in the matrix ring 
 $M\sb 2(T)$, the inverse being given by 
$$
\varphi (Y) ^{-1} =
\begin{bmatrix}
a^ {-1} & -a^{-1}b \\
0 & 1
\end{bmatrix}, 
$$
where we set $a = \varphi (Q)$ and $ b =\varphi (t)$ 
so that we have 
$$
\varphi (Y) =
\begin{bmatrix}
a & b \\
0 & 1
\end{bmatrix}.
$$
\par 
We have seen above the following Lemma.
\begin{lemma}\label{140913a}
For a not necessarily commutative $C $-qsi algebra $T$, there exists 
a $C $-qsi algebra morphism $\varphi :R \to T$ 
such that 
$$
\varphi (Y) =\begin{bmatrix} 
a & b \\
0 & 1
\end{bmatrix}
$$
%sending $Q$ to $a$ and $t$ to $b$ 
if and only if the following three conditions are satisfied: 
\begin{enumerate}
\renewcommand{\labelenumi}{(\arabic{enumi})} 
 \item 
 We have a commutation relation
 $$
 ab = qba, 
$$
\item 
the elements $a, \, b$ satisfies difference differential equations 
$$
\sigma (a) = qa, \qquad \theta^{(1)} (a) = 0, \qquad \sigma (b) = qb,
\qquad \theta^{(1)} (b) =1, 
$$
\item the element 
$a$ is invertible in the ring $T$ or equivalently the matrix 
$$
\begin{bmatrix} 
a & b \\
0 & 1 
\end{bmatrix}
$$
is invertible in the ring $M\sb 2 (T)$.
\end{enumerate}
\end{lemma}
%%%%%%%%%%%%%%%%%%%
%We are going to see where the Hopf algebra structure arises from.
%For a \itqsi algebra $S$ and a \itqsi algebra 
%morphism $\varphi :R \to T$ over $\com$, 
%we denote 
%$$
%\varphi (Y) = 
%\begin{bmatrix}
%\varphi (Q) & \varphi (t) \\
%0   & 1
%\end{bmatrix}
%$$
%so that $\varphi (Y)$ is an invertible matrix in $M\sb 2 (S)$. 
%%%%%%%%
\begin{cor}\label{140913b}
Let $\varphi : R \to T$ be a \itqsi algebra morphism over $C $. 
Using the notation above, let 
$$
H' = 
\begin{bmatrix}
u^\prime & v^\prime \\
0 & 1
\end{bmatrix}
\in M\sb 2 (C\sb T)
$$ 
be an invertible element in the matrix ring 
$ M\sb 2(C\sb T)$ 
satisfying the following two conditions.
\begin{enumerate}
\renewcommand{\labelenumi}{(\arabic{enumi})} 
 \item 
$u^\prime v^\prime = qv^\prime u^\prime$ and the element $u^\prime $ is invertible in the ring $C\sb T$ of constants of $T$. 
\item 
 The set $\{ u^\prime, \, v^\prime \}$ and 
 the set of entries of the matrix 
 $\varphi (Y)$ are mutually commutative.
\end{enumerate}
 Then, there exists 
a \itqsi algebra morphism $\psi \in \Hom\sb{C\text{-}qsi} (R, \, T)$ 
over $C $ such that 
$$
\psi (Y) = \varphi (Y)H'.
 $$
\end{cor}
\begin{proof}
By Lemma \ref{140913a}, 
the matrix 
$$
\varphi (Y) 
= 
\begin{bmatrix} 
a & b \\
0 & 1 
\end{bmatrix}
$$
satisfies conditions of Lemma \ref{140913a}. This, together with the assumption (1) and (2) in this Corollary, implies 
that the matrix 
$$
\varphi (Y) H'
$$
satisfies the conditions of Lemma \ref{140913a}. Now the assertion follows from Lemma \ref{140913a}.
\end{proof}
%To show that we have the morphism $R \to R \otimes \sb C \gH \sb q$ in the Observation, we had better notice Corollary \ref{140913b} above and apply it to 
%$T = R\otimes \sb C \mathfrak{H}\sb q$ and the inclusion morphism 
%$$
%\varphi \sb 0 : R \to R\otimes \sb C \gH
%\sb q, \qquad a \mapsto a\otimes 1. 
%$$
%\end{proof}
In particular if we take $T = R \otimes \sb C \gH\sb q$ and 
$$
H'= 
\begin{bmatrix}
u & v \\
0 & 1 
\end{bmatrix}
$$
and 
$$
\varphi
\sb 0 (Y) =
Y=
\begin{bmatrix} 
Q & t \\
0 & 1 
\end{bmatrix}
=
\begin{bmatrix} 
Q \otimes 1 & t\otimes 1 \\
0 & 1 
\end{bmatrix},
$$
then the conditions of Corollary are satisfied and we get the morphism 
$R \to R \otimes \sb C \gH\sb q$ in the Observation \ref{obs5}. 
%The following Corollary describes in the quantum case, 
%that with respect to the right co-action of the Hopf algebra $\gH \sb q$, the Picard-Vessiot ring $R$ looks as if it were a $\gH\sb q$-torsor. See \eqref{160502a} bellow. 
\par 
It also characterizes the algebra $\gH\sb q$. We notice an $S \in ob(NCAlg / C)$ has a trivial \itqsi-algebra structure over $C$ that if we set $\sigma = \Id\sb{S},\ \theta^{(0)} = \Id\sb{S}$ and $\theta^{(i)} = 0 $ for $i \geq 1$. 
Namely, if we consider a 
functor 
$$
F: (NCAlg /C ) \to (Set) 
$$ 
on the category of not necessarily commutative $C $-algebras defined 
by 
$$
F(S) = \Hom \sb{C\text{-}qsi} ( R, \, R\otimes \sb C S ) \text{\it \/ for \/}
 S \in ob(NCAlg/C ),
$$
then the functor $F$ is representable by the algebra $\gH\sb q$.
\begin{cor}\label{140927b} 
For an object $S$ of the category $(NCAlg/C )$, we have 
\begin{equation}\label{160502a}
\Hom \sb{C\text{-}qsi} ( R, \, R \otimes \sb{C}S ) \simeq 
\Hom \sb {C\text{-}alg} ( \gH \sb q , S),
\end{equation}
where the left hand side denotes the set of \itqsi algebra morphisms over $C $ and 
the right hand side is the set of $C $-algebra morphisms. 
\end{cor}
\begin{proof}
If we notice $C\sb{R \otimes \sb C S} = S$ and take as $\varphi 
: R \to R \otimes \sb C S$ the canonical inclusion 
$$
\varphi \sb 0 :R \to R \otimes \sb C S, \qquad a \mapsto a \otimes 1, 
$$
 it follows from Corollary \ref{140913b} that 
we have a map 
$$
\Hom \sb {C\text{-}alg} ( \gH \sb q , \, S) \to \Hom \sb {C\text{-}qsi} (R , \, R\otimes \sb C S ) .
$$
that sends $\pi \in \Hom \sb{C\text{-}alg} 
(\gH\sb q, \, S )$
 to $\psi \in \Hom \sb{C\text{-}qsi} ( R , \, R\otimes \sb C S)$ such that 
$$
\psi (Y) = \varphi \sb 0 (Y) 
\begin {bmatrix}
\pi (u) & \pi (v) \\ 
0 & 1 
\end{bmatrix}.
$$
To get the mapping of the other direction, 
let $\psi : R \to R \otimes \sb C \gH\sb q$ be a \itqsi morphism over $C $. 
Then using the morphism $\varphi \sb 0$ above, since both 
$\varphi\sb0 (Y)$ and $\psi(Y)$ are solutions to the linear \itqsi equations \eqref{176}, 
an easy calculation shows that 
the entries of the matrix 
$$
H ' := \varphi (Y) ^{-1}\psi (Y) \in M\sb 2 (R\otimes \sb C \gH \sb q )
$$
are constants so that 
$$
H' \in M\sb 2 (S) \subset M\sb 2 (R\otimes \sb C S ). 
$$ 
We single out a 
Sub-lemma because we later use the same argument. 
\begin{sublemma}\label{140929d}
We have the commutation relation 
$$
 u'v'= qv'u' 
$$
among the entries of the matrix 
$$
H' :=
\begin{bmatrix}
u' & v' \\
0 & 1 
\end{bmatrix}.
$$
\end{sublemma}
\begin{proof}[Proof of Sub-lemma]
Let us set 
$$
\varphi\sb 0 (Y)
=
\begin{bmatrix}
a & b \\
0 & 1 
\end{bmatrix}, \text{ and }
\psi (Y) =
\begin{bmatrix}
a' & b' \\
0 & 1 
\end{bmatrix}.
 $$
 So we have 
 %\begine{quation}
 \begin{align}\label{140929b}
 ab &= qba \qquad &a'b'&=q b'a' \\
 \label{140929c}
 a' &= au' \qquad &b'&= av'+ b .
 %\enad{equation}
 \end{align} 
 Since the set $\{ u', \, v' \} \subset S$ and $\{ a, \, b \} \subset R$ are mutually commutative in $R\otimes \sb C S$, 
 substituting equations \eqref{140929c} into the second equation of \eqref{140929b} and then using the first equation of \eqref{140929b}, 
 Sublemma follows. 
\end{proof}
By Sub-lemma, we get a morphism 
$\pi \sb {\psi}: \gH\sb q \to S$ sending $u$ to $u'$ and $v$ to $v'$.  
So $\psi \mapsto \pi\sb \psi$ gives the mapping of the other direction.
\end{proof}
\begin{remark}
%As we are in the non-commutative case, the converse of the Corollary is false. 
%In fact, 
For two \itqsi morphisms $\varphi , \, \psi : R \to T$ over $C $, 
let us set 
$$
\varphi (Y) = 
\begin{bmatrix} 
a & b \\
0 & 1
\end{bmatrix}, \qquad 
\psi (Y) =
\begin{bmatrix}
c & d \\
0 & 1
\end{bmatrix}. 
$$ 
It follows from difference differential equations 
%%%%%%%%%%%%%%
\begin{align*}
\sigma \left(
\begin{bmatrix}
a & b \\
0 & 1 
\end{bmatrix}\right)
& =
 \begin{bmatrix}
q & 0 \\
0 & 1 
\end{bmatrix} 
\begin{bmatrix}
a & b \\
0 & 1 
\end{bmatrix} , \qquad 
&\theta^{(1)}\left(
\begin{bmatrix}
a & b \\
0 & 1 
\end{bmatrix}\right)
 =
 \begin{bmatrix}
0 & 1 \\
0 & 0 
\end{bmatrix} 
\begin{bmatrix}
a & b \\
0 & 1 
\end{bmatrix} \\
%%%%%%%%%%%%
\sigma \left(
\begin{bmatrix}
c & d \\
0 & 1 
\end{bmatrix}\right)
& =
 \begin{bmatrix}
q & 0 \\
0 & 1 
\end{bmatrix} 
\begin{bmatrix}
c & d \\
0 & 1 
\end{bmatrix} , \qquad 
&\theta^{(1)}\left(
\begin{bmatrix}
c & d \\
0 & 1 
\end{bmatrix}\right)
 =
 \begin{bmatrix}
0 & 1 \\
0 & 0 
\end{bmatrix} 
\begin{bmatrix}
c & d \\
0 & 1 
\end{bmatrix} 
\end{align*}
%%%%%%%%%%%%%%%%%%%%%%%%%%%%%
that the entries of the matrix 
$$
\begin{bmatrix}
u^\prime & v^\prime \\
0 & 1 
\end{bmatrix} 
:= \varphi (Y)^{-1} \psi (Y) 
= 
\begin{bmatrix}
a & b \\
0 & 1
\end{bmatrix}^{-1}
\begin{bmatrix}
c & d \\
0 & 1
\end{bmatrix}
=
\begin{bmatrix}
a^{-1}c & a^{-1}d -a^{-1}b \\
0 & 1
\end{bmatrix}
$$
are constants. Namely, 
$$
\begin{bmatrix}
u^\prime & v^\prime \\
0 & 1 
 \end{bmatrix}
 \in M\sb 2 (C\sb T ).
$$
So equivalently 
$$ 
\psi (Y) = \varphi (Y) 
\begin{bmatrix}
u^\prime & v^\prime \\
0 & 1 
 \end{bmatrix}.
$$
The entries of the matrix do not 
necessarily 
 satisfy the commutation relation $u^\prime v^\prime = q u^\prime v^\prime$.
\end{remark}
\par 
For algebras $A, \, S \in ob(NCAlg / C )$, we set 
$$
A(S) := \Hom \sb {C\text{-}alg}( A , \, S)
$$
that is the set of $C $-algebra morphisms. 
\begin{observation}[Origin of co-multiplication 
of the Hopf algebra $\gH\sb q$]\label{140927a}
The co-multi-plication $\Delta : \gH \sb q \to \gH \sb q \otimes \sb C \gH \sb q$ comes from 
the multiplications of matrices. 
%the composition of morphisms. 
More precisely, to construct an algebra morphism $\Delta : \gH \sb q \to \gH \sb q \otimes \sb C \gH \sb q$, 
it is sufficient to give a functorial morphism 
\begin{equation}\label{140927c}
\gH\sb q \otimes \sb C \gH \sb q (S) \to \gH\sb q (S) \qquad \text{ for } S\in ob(NCAlg/C ). 
\end{equation}
An element of $\gH\sb q\otimes \sb C \gH \sb q (S)$
being given, it determines a pair $(\pi \sb 1, \, \pi \sb 2 )$ of morphisms 
$\pi \sb 1, \, \pi \sb 2 : \gH \sb q \to S$ 
such that the images $\pi \sb 1 (\gH\sb q )$ 
and $\pi \sb 2 (\gH\sb q )$ are 
mutually commutative.  
This condition is equivalent to 
mutually commutativity of 
the set of the entries
$\{ u\sb 1', \, v\sb 1' \}$ and 
$\{ u\sb 2' , \, v\sb 2' \}$ 
of the matrices
$$
H\sb 1 ' :=
\begin{bmatrix}
u\sb 1'
& v\sb 1' \\
0 & 1 
\end{bmatrix}
:=
\begin{bmatrix}
\pi \sb 1 (u) & \pi \sb 1(v) \\
0 & 1 
\end{bmatrix} , 
\qquad 
H\sb 2' :=
\begin{bmatrix}
u\sb 2'
& v\sb 2' \\
0 & 1 
\end{bmatrix}
:=
\begin{bmatrix}
\pi \sb 2 (u) & \pi \sb 2(v) \\
0 & 1 
\end{bmatrix} .
$$
\end{observation}
We show that there exists a morphism $\pi \sb 3 : \gH \sb q \to S$
such that 
$$
H\sb 1'H\sb 2 ' =
\begin{bmatrix}
\pi \sb 3 (u) & \pi \sb 3 (v) \\
0 & 1
\end{bmatrix}.
$$
In fact, by Corollary \ref{140913b}, there exists a morphism $\psi \sb 1: R \to R \otimes \sb C S$ such that $$\psi\sb 1 (Y ) = \varphi \sb 0(Y) H\sb 1'.$$ 
 Then since the entries of 
$H\sb 2'$ and 
the union 
$$
 \text{(the entries of $ \varphi\sb 0 (Y)$)} \, \cup \, \text{(the entries of 
$H^\prime \sb 1$)}
$$
 are mutually commutative and consequently 
 the entries of $H\sb 2'$ and the entries of the product $\varphi \sb 0 (Y)H\sb 1'$ are mutually commutative, 
 by Corollary 
\ref{140913b}, there exists a morphism $\psi\sb 2 : R \to R\otimes \sb C S$ such that 
$$
\psi \sb 2 (Y) = (\varphi\sb0 (Y)H\sb 1 ')H\sb 2 '. = \varphi\sb 0 (Y)(H\sb 1' H\sb 2').
$$ 
So if we note that the entries of $H\sb 1' H\sb 2'$ and the entries of the matrix $\varphi \sb 0 (Y)$ are mutually commutative, writing 
$$
H\sb 1' H\sb 2' =
\begin{bmatrix}
u\sb3' & v\sb3' \\
0 & 1 
\end{bmatrix} , 
$$ 
the argument of the proof of Sub-lemma \ref{140929d} shows us that, 
we have 
$
u\sb 3' v\sb3'=qv\sb3' u\sb 3 '
$.
Hence 
there exists a morphism $\pi \sb 3 : \gH\sb q \to S$ sending $u$ to $u\sb 3'$ and $v$ to $v\sb 3'$. 
Now the mapping $(\pi \sb 1 , \, \pi \sb 2) \mapsto \pi\sb 3$ defines the morphism 
\eqref{140927c}. 
\begin{proposition}\label{160609}
The right co-action
\[
\rho \colon R \to R \otimes_{C} \gH_{q}
\]
is an $\gH_{q}$ torsor in the following sense. The $C$-algebra morphism $\rho$ defines a $C$-linear map
\begin{equation}\label{160608a}
\varphi \colon R \otimes_{C} R \to R\otimes_{C} \gH_{q} 
\end{equation}
such that $\varphi(ab) = (a\otimes 1)\rho (b)$ for $a, b \in R$. The $C$-linear map $\varphi$ is an isomorphism of the $C$-vector spaces. 
\end{proposition}
\begin{proof}
The $C$-linear morphism $\varphi$ is, in fact, a left $R = (R\otimes_{C} 1)$-module morphism. The $C$-linear algebra morphism $\rho \colon R \to R\otimes_{C} \gH_{q}$ sends $C$-linear basis $\{ Q^{m}t^{n}\}_{m \in \Z, n \in \N}$ of $R$ to
\begin{align*}
\rho(Q^{m} t^{n}) &= \varphi(1 \otimes Q^{m}t^{n}) \\
&= (Q\otimes u)^{m}(t\otimes 1 + Q\otimes v)^{n} \\
&= (Q^{m+n}\otimes 1)(1 \otimes u^{m})\times \\
&\qquad (1\otimes v^{n} + \text{\it \/ an $R$-linear combination of $1\otimes v^{i}$ for $0 \leq i \leq n-1$}). 
\end{align*}
Since $Q$ is an invertible element of $R$, 
the latters form an $R = (R\otimes_{C} 1)$-linear basis of $R\otimes_{C} \gH_{q}$. So the $R$-linear map $\varphi \colon R\otimes_{C} R \to R\otimes_{C} \gH_{q}$ sends the $R = (R\otimes_{C} 1)$ linear basis $\{ 1\otimes Q^{m}t^{n} \}_{m \in \Z, n \in \N}$ to the other $R= (R\otimes_{C} 1)$-linear basis of $R\otimes_{C} \gH_{q}$. 

So the $R=(R\otimes 1)$-module morphism $\varphi \colon R\otimes_{C} R \to R\otimes_{C} \gH_{q}$ is an isomorphism. 
\end{proof}
%%%%%%%%%%%%%%%%%%%%%%%%%%%%%%%%%%%%%%%
%Let ma 
%$\varphi , \, \psi \in \gH \sb q (S)$. 
%It follows from Corollary \ref{} that they determine morphisms 
%\varphi ^\prime ,\, \psi ^\prime : R \to R \otimes \sb \com S$. 
%%The composite of \itqsi morphisms 
%$$
%\varphi ^\prime ; R \to R \otimes\sb \com S \text{ and }
% \psi \otimes \Id\sb S : R\otimes \sb \com S \to (R \otimes \com (
%S) \otimes \sb \com S 
%$$
 %gives us a \itqsi morphis
 %$$
 %\eta : R \to R \otimes \sb \com S \otimes \sb \com S.
 %$$
 % Let $\pi^\prime : \gH \to S \otimes \sb \com S$ be the morphism 
 % corresponding to $\eta $ be Corollary \ref{}. 
 % We set $\pi := \pi ^\prime \circ \pi \sb{12}$, where 
%$ \pi \sb {12} S\otimes \sb\com S \to S \otimes \sb \com S $
%is the switching factor automorphism. So 
%$\pi : \gH \sb q \to R \otimes R$. 
%%%%%%%%%%%%%%%%%%%%%%%%%%%%%%%%%%%%%%%%%%%%%%%%%%%%%%
\section{On the uniqueness of the Picard-Vessiot ring}\label{141223b}
We show that our Picard-Vessiot ring $R/C $ is unique. 

Let us start with a Lemma 
 on the $\R$-module $\R ^n$ of column 
vectors for not necessarily commutative $C $-algebra $\R$. The Lemma is trivial if the ring is commutative. We give a proof of the Lemma so that the reader could understand the logical 
structure of the whole argument. 
\begin{lemma}\label{140829b}
 Let $Y= ( \by\sb 1, \, \by \sb2, \, \ldots , \by\sb n ) \in 
 M\sb n (\R )$ be an $n \times n$-matrix with entries in the ring $\R$ so that the $\by\sb i$'s are column vectors 
for $1\le i \le n$. The following conditions (1), (2) and (3) on the matrix $Y$ are equivalent. 
%We assume two conditions. 
\begin{enumerate}
\renewcommand{\labelenumi}{(\arabic{enumi})}
\item (1.1)The column vectors $\by\sb i $'s ($1 \le i \le n $) generate the right $\R$-module $\R^n$. \\ (1.2) The column vectors $\by\sb i $'s ($1\le i \le n $) are 
right $\R$-linearly independent or they are 
linearly independent 
elements in the right $\R$-module $\R^n$.
\item We have the direct sum decomposition of the right $\R$-module 
 $$
\R^n = \bigoplus \sb{i= 1} ^ n \by\sb i \R . 
 $$
\item The matrix $Y$ is invertible in the ring $ M\sb n (\R)$. 
\end{enumerate}
\end{lemma}
\begin{proof}
The equivalence of (1) and (2) is evident. We prove that (1) implies (3). 
In fact, 
let us set 
$$
e\sb 1 := {}^t (1, \, 0, \ldots , \, 0), \, 
e\sb 2 := {}^t (0, \, 1, \, 0, \ldots , \, 0), \ldots , 
e\sb n := {}^t (0, \, 0, \ldots ,0, \, 1) \in \R^n.
$$
If we assume (1),
since the vectors $e\sb i$'s that are elements of $\R^n$ are right $\R$-linear combinations of the 
column vectors $\by\sb j$'s, 
there exists a matrix $Z \in M\sb n (\R)$ such that 
\begin{equation}\label{140826a}
YZ= I \sb n 
\end{equation}
 or the matrix $Y$ has a right inverse in $ M\sb n (\R)$. 
Multiplying $Y$ on \eqref{140826a} from left, we get 
$$
YZY =Y.
$$
So we have 
\begin{equation}\label{140827a}
Y(ZY - I \sb n) =0.
\end{equation} 
We notice here that (1.2) implies that 
if we have 
$$
Y u =0 \ \text{\it for } u ={} ^t(u\sb1, \, u\sb 2 , \ldots , u\sb n ) \in \R^n, 
$$
then $u=0$. Therefore \eqref{140827a} implies $ZY -I\sb n =0$ and consequently 
$ZY = I\sb n$.
 So $Z$ is the inverse of $Y$ and the condition (3) is satisfied. 
 We now assume the condition (3). Then for every 
 element $v\in \R^n$ a 
 liner equation 
 $$
 Yx = u, \text{\it where } x\in \R^n \text{\it is an unknown column vector in $R^n$, } 
 $$
  has the unique solution $x=Y^{-1}v\in \R^n$ so that (2) is satisfied. 
\end{proof}
Dually we can prove the following result for the left $\R$-module ${}^t\R^n$ of 
row vectors. 
\begin{cor}\label{140901a}
Let $Y= {}^t ( \by\sb 1, \, \by \sb2, \, \ldots , \by\sb n ) \in 
 M\sb n (\R )$ be an $n \times n$-matrix with entries in the ring $\R$ so that the $\by\sb i$'s are column vectors. The following conditions (1), (2) and (3) on the matrix $Y$ are equivalent. 
%We assume two conditions. 
\begin{enumerate}
\renewcommand{\labelenumi}{(\arabic{enumi})}
\item (1.1)The row vectors ${}^t\by\sb i $'s ($1 \le i \le n $) generate 
the 
left $\R$-module ${}^t\R^n$. \\ 
(1.2) The row vectors ${}^t\by\sb i $'s ($1\le i \le n $) are 
left $\R$-linearly independent. 
\item We have the direct sum decomposition of the left $\R$-module 
 $$
{}^t\R^n = \bigoplus \sb{i= 1} ^ n \R\,{}^t\by\sb i . 
 $$
\item The matrix $Y$ is invertible in the ring $ M\sb n (\R)$. 
\end{enumerate}
\end{cor}
\par 
Let $M$ be a left $ C [\sigma , \sigi , \, \theta ^* ]$-module that 
is of finite dimension $n$ as a $C $-vector space. Let 
 $\{ m\sb 1, \, m\sb 2, \ldots , m\sb n \} $
be a basis of the $C $-vector space $M$. Setting 
$m = {}^t (m\sb 1, \, m\sb 2, \ldots , m\sb n )$, 
 there exist matrices 
 $A\in \mathrm{GL}\sb n (C ), \, B \in M\sb n (C ) $ satisfying 
\begin{equation}\label{140903a}
\sigma (m) = Am, \text{ and } \theta^{(1)} (m) =Bm. 
\end{equation}
As we have seen in Section \ref{14.6.25b}, the left $C [\sigma , \, \sigi , \, \theta ^*]$-module $M$ defines a system of 
\QSI equation 
\begin{equation}\label{140829a}
\sigma (y) =A y, \qquad \theta ^ {(1)}( y) = By\ \text{\it for an unknown column vector $y$ of length $n$. }
\end{equation}
We are interested in solutions $y\in \R^n$ for a \QSI algebra $\R$ over $C $. 
\begin{definition}\label{140901b}
 For \itqsi algebra $\R$ over $C $, 
we say that a set $\{ \by\sb 1, \, \by\sb 2, \, \ldots , \by\sb n \}$ 
of solutions to \eqref{140829a} so that 
$\by\sb i \in \R^n $ for $1\le i \le n$,  
is a fundamental system of solutions to \eqref{140829a} if the matrix 
$Y=(\by\sb 1 , \, \by \sb 2, \, \ldots , \by \sb n ) \in M\sb n(R) $ satisfies 
the equivalent conditions of Lemma \ref{140829b}.
\end{definition}
The dual to a fundamental system is a trivializing matrix of the \QSI module $M$.
 \begin{definition}\label{140901c}
We assume that the \QSI module $\R \otimes \sb C M$ is trivialized over a 
$C $-\QSI
 ring $\R$. %over $\com$. 
 Namely, there exist elements 
 $$
 c\sb 1, \, c\sb 2, \, \ldots , c\sb n \in \R\otimes \sb C M
 $$
 such that 
 $$
 \sigma (c\sb i) = c\sb i,\qquad \theta^{(1)}(c\sb i ) = 0 \text{ for every } 1\le i \le n 
 $$
and such that we have $\R$-module decomposition
$$
\R \otimes \sb C M = \bigoplus \sb{i=1}^n 
\R c\sb i.
$$
So writing the elements $c\sb i$'s as a left $\R$-linear combination of the basis 
$$
\{m\sb 1, \, m\sb 2, \ldots , m\sb n\}, 
$$ 
we get a matrix 
$
Y\in M\sb n (\R)
%(\by \sb{ij})\sb{1\le i, j \le n} \in \R^n 
$ 
 such that 
 $$
 {}^t (c\sb 1, \, c\sb 2, \ldots , c\sb n ) = Y\,{}^t(m\sb 1, m\sb 2, \ldots , m\sb n).
 $$
 We call the matrix $Y\in M\sb n (\R)$ a trivializing matrix of 
 \QSI module $M$ over $\R$.
 \end{definition}
 \begin{lemma}
A trivializing matrix over $\R$ is 
invertible in the matrix ring $ M\sb n (\R )$. 
\end{lemma}
\begin{proof}
It is sufficient to follow the argument of 
the proof of 
Corollary \ref{140901a}.
\end{proof}
Now we make clear the relation between fundamental system and trivializing matrix.
\begin{proposition}\label{160113a}
The following four conditions on an invertible matrix $Y \in M\sb n (\R)$ are equivalent. 
We denote $Y^{-1}$ by $Z$ or $Y = Z^{-1}$.
\begin{enumerate}
\renewcommand{\labelenumi}{(\arabic{enumi})}
\item The matrix $Y$ satisfies \QSI equations 
\begin{equation}\label{140901d}
\sigma (Y) = AY, \qquad \theta^{(1)}(Y) = BY,
\end{equation}
$A, \, B$ being the matrices in \eqref{140903a}.
\item The matrix $Y$ is a fundamental system of solutions of $M$. 
\item The matrix $Z$ satisfies \QSI equations 
\begin{equation}\label{140901e}
 \sigma (Z) = ZA^{-1}, \qquad \theta ^{(1)} (Z) = -ZA^{-1}B
 \end{equation}
 \item 
 The matrix $Z$ is a trivializing matrix for $M$ over $\R$. 
\end{enumerate}
\end{proposition} 
\begin{proof}
The equivalence of conditions (1) and (2) follows from Lemma \ref{140829b} and Definition
\ref{140901b}. 
To prove the equivalence of (3) and (4), 
 we set 
$$
 {}^t (c\sb 1, \, c\sb 2, \ \ldots , c\sb n ):= Z{}^t(m\sb 1, \, m\sb 2, \ldots , m\sb n) , 
$$
where the $m\sb i$'s are the basis of $M$ chosen above, 
so that 
$$
c\sb i = \sum \sb{l= 1}^n z\sb {il}m\sb l \ \text{\it for every $ 1\le i \le n$.}
$$
It is convenient to introduce 
$$
\mathbf{c} :={}^t (c\sb 1, \, c\sb 2, \ldots , c\sb n ),\ \text{\it and }
\mathbf{m} :={}^t (m\sb 1, \, m\sb 2, \ldots , m\sb n ).
$$
%if we write $ \mathbf{c} = {}^t (c\sb 1, \, \ldots , c\sb n)$,¡¡
 So we have 
 \begin{equation}\label{140901f}
 \mathbf{c}=Z \mathbf{m}.
 \end{equation}
 Now we assume Condition (3) and show Condition (4). 
To this end, we prove that the $c\sb i$'s that are elements of $\R \otimes \sb C M$, are constants. In fact, if we apply $\sigma$ to \eqref{140901f}, 
it follows from the first equation in \eqref{140901d}, 
\begin{align*}
\sigma (\mathbf{c}) &= \sigma (Z) \sigma (\mathbf{m}) \\
     &= (ZA^{-1})(A\mathbf{m}) \\
     &= Z\mathbf{m} \\
     &=\mathbf{c}.
\end{align*} 
Namely $\sigma (\mathbf{c}) = \mathbf{c}$. 
Now we apply $\theta^{(1)}$ to \eqref{140901f} to get 
\begin{align*}
\theta ^{(1)}(\mathbf{c}) &= \theta^{(1)}(Z)\mathbf{m} + \sigma (Z) \theta^{(1)}(\mathbf{m}) \\
   &= (-ZA^{-1}B)\mathbf{m} + (ZA^{-1})B\mathbf{m} \\    
   &=0. 
\end{align*}
So $\theta ^{(1)}(\mathbf{c}) =0$ and $\mathbf{c}$ is a constant. Hence $Z$ is a trivializing matrix by Definition \ref{140901c} and the argument in the proof of Corollary \ref{140901a}. 
Conversely, we start from Condition (4). 
If we recall 
\begin{equation}\label{140902a}
\mathbf{c} := Z\mathbf{m}, 
\end{equation}
then, as we assume Condition (4), $\mathbf{c}$ is a constant. Applying $\sigma$ 
and $\theta ^{(1)}$ to \eqref{140902a}, we get Condition (3).
%Then, 
%\eqref{}.
% Now the equivalence of (3) and (4) is a result of the argument of the proof of Corollary \ref{140901a} and 
%Definition \ref{140901c}.
\par 
It remains to show the equivalence of (1) and (3).
Let us assume (1) and show (3). 
If we apply the automorphism $\sigma$ 
to the equality $ZY=\mathrm{I} \sb n$, 
the first equality in \eqref{140901d} implies the first equality of \eqref{140901e}. 
On the otherhand, 
applying $\theta ^{(1)}$ to the equality 
$ZY= I\sb n$, 
we get
$$
\theta ^{(1)}(Z) Y+ \sigma (Z)\theta ^{(1)}(Y) = 0.
$$ 
It follows from equation \eqref{140901d} 
\begin{equation}
\theta^{(1)}(Z)Y + ZA^{-1}BY=0.
\end{equation}
Since the matrix $Y$ is invertible, we conclude 
$$
\theta ^{(1)}Z= -ZA^{-1}B.
$$ 
So the matrix $Z$ satisfies Condition (3). 
The proof of the converse that Condition 
(3) implies (1) is similar. Applying first $\sigma$ and then $\theta^{(1)}$ to $ZY =I \sb n$, we immediately get Condition (1). 
\end{proof}
We are ready to characterize the Picard-Vessiot ring $R/C $. 
Besides the properties we mentioned above, 
we have a $C $-morphism or a $C $-valued point of the abstract ring $R\n$ 
\begin{equation}\label{140821a}
R\n \to C \text{ sending } Q^{\pm 1} \mapsto 1, \quad X \mapsto 0.
\end{equation} 
We sometimes call it a $C $-rational point. 
\begin{lemma}\label{140905c}
Let $\mathcal{R}$ be a simple \QSI algebra over $C $. If the abstract algebra
$\mathcal{R} \n$ has a $C $-valued point, then the ring of constants of 
$\mathcal{R}$ coincides with $C $. 
\end{lemma}
\begin{proof}
Assume to the contrary. Then there would be a constant $f \in \mathcal {R}$
that is not an element of $C $. Let 
$$
\varphi: \mathcal{R} \to C 
$$
be the $C $-valued point. We set $c:= \varphi (f)$ that is an element 
of $C $. 
So the element $f-c \not=0$ is a constant of \QSI algebra $\mathcal{R}$. 
Therefore the bilateral ideal $I$ generated by $f-c$ is a
\QSI bilateral ideal of $\mathcal{R}$ because the ideal $I$ is generated by the constant $f-c$. As the ideal $I$ contains $f-c\not= 0$, the simplicity of $\mathcal{R}$ implies $I = \mathcal{R}$. So there would be a positive integer $n$ and elements $a\sb i , \, 
b\sb i \in \mathcal{R}$ for $0 \le i \le n$ such that 
\begin{equation} \label{140821b}
\sum \sb {i= 1}^n a\sb i (f-c)b\sb i = 1.
\end{equation}
Applying the morphism $\varphi$, we would have $0=1$ in $C $ 
by \eqref{140821b}
that is a contradiction. 
\end{proof}
%The next Definition of a certain minimality is indispensable for the characterization of the extension $R/\com $. 
%%%%%%%%%%%%%%%%%%%%%%%%%%%%%%%%%%%%%%%
%\begin{definition} We keep the notation above so that 
%$M$ is a $\com [\sigma , \, \theta^*]$-module that is finite dimensional as a 
%$\com$-vector space. 
%We assume $M$ has a fundamental system solution over 
%$\R$.
%is trivialized over a \QSI algebra $\R$. 
%If for every 
%fundamental
%system of solutions 
%$Y\in M\sb n (\R )$
%for $M$ so that the matrix $Y$ is invertible in $M\sb n %(\R)$, the abstract ring $\R\n$ is generated by the %entries of $Y$ and $Y^{-1}$ 
%over $\com$, then we say that $\R$ is minimally %generated 
%by fundamental systems of solutions of $M$. 
 %\end{definition}
 So far in this Section, we studied general 
 $ C [\sigma, \, \sigi ,\, \theta ^{*}]$-module $M$. From now on, we come back to the $ C [\sigma, \, \sigi , \, \theta ^{*}]$-module $M$ in Section \ref{14.6.25b} so that 
 writing $\mathbf{m} ={}^t (m\sb 1, \, m\sb2 )$, 
 \begin{align}
 M &=C m\sb 1 \oplus C m\sb 2, \\ \label{140905a}
 \sigma (\bfm) &= A\bfm , \qquad \theta ^{(1)}(\bfm)
 =B\bfm , 
 \end{align}
 where 
 $$
 A= \begin{bmatrix}
q & 0\\
0 & 1
\end{bmatrix} \text{ and } 
 B = \begin{bmatrix}
0 & 1\\
0 & 0
\end{bmatrix}.
$$
\par 
\begin{theorem}\label{th1}
Using the notation above, we can characterize the Picard-Vessiot ring $R/ C $ for $M$ constructed in Section \ref{14.6.25b}, in the following way. \par
Let $\R/C $ be a \QSI extension satisfying the following conditions. 
\begin {enumerate} 
\renewcommand{\labelenumi}{(\arabic{enumi})}
\item There exists a fundamental system of solutions 
$\mathbf{Y} \in 
 M\sb 2 (\R)$ for $M$ such that 
$$
\R= C \langle \mathbf{Y}, \, \mathbf{Y}^{-1} \rangle\sb{alg}.
$$
\item The \QSI algebra $\R$ is simple. 
\item 
There exists a $C $-rational point of the abstract $C $-algebra $\R\n$.
 \end{enumerate}
 Then the \QSI algebra $\R$ is $C $-isomorphic to the Picard-Vessiot ring $R$. 
\end{theorem}
\begin{proof}
Let us express $\mathbf{Y}$ in the matrix form: 
$$
\mathbf{Y}= \begin{bmatrix}
a & b \\
c& d
\end{bmatrix} \in M\sb 2 (\R).
$$
Hence by \eqref{140905a}, the matrix $\mathbf{Y}$ satisfies 
\begin{equation*}
\begin{bmatrix}
\sigma (a) & \sigma (b) \\ 
\sigma (c) & \sigma (d)
 \end{bmatrix}=
\begin{bmatrix} 
q & 0 \\ 
0 & 1
 \end{bmatrix}
\begin{bmatrix}
a & b \\ 
c & d
 \end{bmatrix},
 \qquad 
 \begin{bmatrix}
\theta^{(1)} (a) & \theta^{(1)} (b) \\ 
\theta^{(1)} (c) & \theta^{(1)} (d)
 \end{bmatrix}=
 \begin{bmatrix} 
0 & 1 \\ 
0 & 0
 \end{bmatrix}
\begin{bmatrix}
a & b \\ 
c & d
 \end{bmatrix}.
\end{equation*} 
or to be more concrete
\begin{alignat}{4}\label{140905b}
&\sigma (a)= q a, \qquad & \theta^{(1)}(a) =c, \qquad
& \sigma (c) =c, \qquad &\theta ^{(1)}(c) = 0, \\ 
\label{140905bb}
&\sigma (b)= q b, \qquad &\theta^{(1)}(b) =d, \qquad 
&\sigma (d) =d , \qquad &\theta ^{(1)}(d) = 0. 
\end{alignat}
It follows from \eqref{140905b} and \eqref{140905bb}
that 
$c, \, d$ are constants of $\R$.
 By Lemma \ref{140905c} and assumption (2) on $\R$, the ring $C\sb{\R}$ of constants of $\R$ coincides with $C $. So $c, \, d$ are complex numbers and hence
 by replacing the column vectors of the matrix $\mathbf{Y}$ by their appropriate $C $-linear combinations if necessary, 
 we may assume that $c= 0$ and $d=1$ so that 
$$
\mathbf{Y} = \begin{bmatrix}
a & b \\
0 & 1
\end{bmatrix}.
$$
Consequently the set of equations \eqref{140905b} and \eqref{140905bb} reduces to 
\begin{equation}\label{140905d}
\sigma (a) =qa, \qquad \theta^{(1)}(a) =0,\qquad \sigma (b) = qb, \qquad
\theta^{(1)}(b) =1. 
\end{equation}
Since the matrix $\mathbf{Y}$ is invertible, $a$ is an invertible element of the ring $\R$. We show that 
$ f:=qa^{-1}b -ba^{-1} \in \R $ is a constant. 
In fact, since the complex number $q$ is in the center 
of $\R$, it follows from \eqref{140905d} that 
\begin{align*}
\sigma (f) &= q\sigma (a^{-1})\sigma (b) - \sigma (b)\sigma 
(a^{-1}) \\
 & = qa^{-1}q^{-1}qb - qb a^{-1}q^{-1} \\
 & =qa^{-1}b - ba^{-1} \\
 &= f 
\end{align*}
and 
\begin{align*}
\theta ^{(1)} (f) & = q\sigma (a^{-1})\theta^{(1)}(b) - 
\theta ^{(1)} (b)a^{-1} \\
 & = qa^{-1}q^{-1}1 - 1a^{-1} \\
& = 0.
\end{align*}
Therefore $f$ is a complex number. Now we denote by $g$ 
the complex number 
$$
\frac{f}{1-q}.
$$
and set 
$$
b^\prime : = b + ga. 
$$
Then 
$$
\mathbf{Y}^\prime := 
\begin{bmatrix}
a & b^\prime \\
0 & 1
\end{bmatrix} 
$$
is a fundamental system of solutions
so that 
\begin{equation}\label{140906a}
\sigma (a) = qa, \qquad \sigma (a^{-1}) =q^{-1}a^{-1}, \qquad 
 \sigma (b^\prime )= q b^\prime , \qquad \theta^{(1)} (b^\prime ) = 1 
\end{equation}
 and we have 
\begin{equation}\label{140905f}
%\mathbf{Y}\
\R = C \langle \mathbf{Y} , \, \mathbf{Y}^{-1}
\rangle \sb{alg} = C \langle \mathbf{Y}^\prime , \, \mathbf{Y}^{\prime -1} \rangle \sb{alg} = C \langle a, \, b^\prime , \,
a^{-1} \rangle \sb{alg}
\end{equation}
and moreover
 we have 
\begin{equation}\label{140905h}
a b^\prime = qb^\prime a .
\end{equation}
We have seen in Section \ref{14.6.25b}
that $R=\langle Q, Q^{-1}, \, t\rangle \sb{alg}$ and the relations among the generators $Q, \, Q^{-1}, \, t$ are reduced to 
$$
Q Q^{-1}= Q^{-1}Q=1, \qquad qtQ -Qt=0, \qquad \text{ $C $ commutes with $Q, \, Q^{-1}$ and $t$.}
$$
Thus, there exists a $C $-morphism $\varphi : R \to \R$ 
of abstract $C $-algebras by \eqref{140905h}. It follows from \eqref{140906a} and difference differential equations for $Q, \, t$, the morphism $\varphi$ is \QSI morphism. 
By \eqref{140905f}, the morphism $\varphi$ is surjective.
Since $R$ is simple \itqsi algebra, the kernel of the \itqsi morphism $\varphi$ is $0$ and the morphism is injective. 
Therefore the \itqsi morphism $\varphi$ is an isomorphism. 
\end{proof}
%%%%%%%%%%%%%%%%%%
%%%%%%%%%%%%%%%%%%%%%%%%%%%%%%%%%%%%%%%%%
%%%%%%%%%%%%%%
\section{Generalized Tannaka equivalence of categories. }
Let us review classical Picard-Vessiot theory formulated by Tannaka equivalence of two rigid tensor categories. 
Let $k$ be a differential field of characteristic $0$ and we assume the field $C = C_{k}$ of constants of $k$ is algebraically closed. 
We denote by $\mathcal{D} = k[\partial]$ the ring of linear differential operators. 
We denote by $(\mathcal{D}\text{-}mod)$ the category of left $\mathcal{D}$-modules that are finite-dimensional $k$-vector spaces. 
We know $(\mathcal{D}\text{-}mod)$ is a rigid tensor category. 
Namely, for two objects $M_{1}, M_{2} \in ob((\mathcal{D}\text{-}mod))$, the tensor product $M_{1}\otimes_{k}M_{2}$ and the internal homomorphism $\Hom_{k}(M_{1}, M_{2})$ exist in $ob((\mathcal{D}\text{-}mod))$. 
Let $\mathcal{G}$ be a commutative Hopf algebra. Then similarly the category of right $\mathcal{G}$-co-modules that are 
finite-dimensional as $C$-vector spaces form a rigid tensor category. 
Let $H$ be an object of the category $(\mathcal{D}\text{-}mod)$. Let $G$ be the Galois group of Picard-Vessiot ring of the system of linear differential equations determined by $H$. So $G$ is an affine group scheme over $C$. Hence the coordinate ring $C[G]$ is a Hopf algebra. 
\begin{theorem}[van der Put and Singer \cite{put-singer}, Theorem 2.33 in 2.4]
The rigid tensor sub-category $\{\{ H \}\}$ of the rigid tensor category $(\mathcal{D}\text{-}mod)$ generated by $H$ is equivalent to the rigid tensor category $(comod\text{-}C[G])$ of right $C[G]$-co-modules that are finite-dimensional $C$-vector spaces. 
\end{theorem}
The result of preceding Sections suggest that we might as well expect a similar result for the Hopf algebra $C[\sigma,\ \sigi,\ \theta^{\ast}]$. 
To be more precise in the classical differential case, we take the Hopf algebra $C[\mathbb{G}_{a}]$ that is commutative and co-commutative, and the base field $k$. 
In the qsi case, the Hopf algebra $C[\sigma,\ \sigi,\ \theta^{\ast}]$ is neither commutative nor co-commutative, and the base field $k$ coincides with $C$. 
We studied a concrete example of qsi module $M$ in Section \ref{14.6.25b}. 
We denote by $(C[\sigma,\ \sigi,\ \theta^{\ast}]\text{-}mod)$ the rigid tensor category of left $C[\sigma,\ \sigi,\ \theta^{\ast}]$-modules that are finite-dimensional $C$-vector spaces. 
For two objects $(M_{1},\ \sigma,\ \sigi,\ \theta^{\ast})$ and $(M_{2},\ \sigma,\ \sigi,\ \theta^{\ast})$, $M_{1}\otimes M_{2}$ and the internal homomorphisms $\mathrm{Hom}(M_{1}, M_{2})$ exist. 
In fact, let $N :=\mathrm{Hom}(M_{1}, M_{2})$ be the set of $C $-linear maps from $M_{1}$ to $M_{2}$. 
It sufficient to consider two $C $-linear maps
\[
\sigma_{h} \colon N \rightarrow N \text{\it \/ and \/ } \theta^{(1)}_{h} \colon N \rightarrow N 
\]
given by
\[
\sigma_{h}(f) := \sigma_{2} \circ f \circ \sigma_{1}^{-1} \text{\it \/ and \/ } \theta^{(1)}_{h}(f):= -(\sigma_{2}f)\circ\theta^{(1)}_{1} + \theta^{(1)}_{2}\circ f.
\]
So we have $q \sigma_{h}\circ \theta^{(1)}_{h} = \theta^{(1)}_{h} \circ \sigma$. 
Since $q$ is not a root of unity, we define $\theta^{(m)}_{h}$ in an evident manner
\[
\theta^{(m)}_{h} = \begin{cases}
\Id_{N}, \quad & \text{ for } m=0, \\
\frac{1}{[m]_{q}!}\left(\theta^{(1)}_{h}\right)^{m}, \quad 
&\text{ for } m \geq 1.
\end{cases}
\]
Since $ C [\sigma,\, \sigi , \,\theta^{\ast}]$ is a Hopf algebra, for two objects $M_{1},\, M_{2} \in ob((C[\sigma,\ \sigi,\ \theta^{\ast}]\text{-}mod))$ the tensor product $M_{1}\otimes_{C }M_{2}$ is defined as an object of $(C[\sigma,\ \sigi,\ \theta^{\ast}]\text{-}mod)$. 
However, as $C[\sigma, \, \sigi,\, \theta^{\ast}]$ is not co-commutative, we do not have, in general, $M_{1}\otimes_{C }M_{2} \simeq M_{2}\otimes_{C }M_{1}$. 
%%%

We consider the rigid tensor category $(comod\text{-}\mathfrak{H}_{q})$ of right $\mathfrak{H}_{q}$-co-modules that are 
finite-dimensional $C$-vector spaces. 
The results of Sections \ref{14.6.25b} and \ref{141223b} would imply the following 
\begin{expectation}\label{160808a}
We denote by $\{\{ M \}\}$ the rigid tensor sub-category of $(C[\sigma, \sigi, \theta^{\ast}]\text{-}mod)$ generated by the left qsi-module $M$. 
Then the rigid tensor category $\{\{ M \}\}$ is equivalent to the rigid tensor category $(comod\text{-}\mathfrak{H}_{q})$ of right $\mathfrak{H}_{q}$-co-modules that are finite-dimensional $C$-vector spaces. 
\end{expectation}
The Expectation is too naive. It is false but it is not so absurd. 
Since the Hopf algebra $C[\sigma,\ \sigi,\ \theta^{\ast}]$ is neither commutative nor co-commutative, the arguments for the commutative and co-commutative Hopf algebra $C[\mathbb{G}_{a}]$ require subtle modifications. 
We prove a corrected version of Expectation \ref{160808a} in Sections \ref{sec:tensor-equivalence} and \ref{sec:review_examples} of Part \ref{part_III}. %%%%%%%%%%%%%%%%
\begin{observation}\label{1018c}
We have an imperfect Galois correspondence between the elements of the two sets. 
\begin{enumerate}
\renewcommand{\labelenumi}{(\arabic{enumi})}
\item The set of 
quotient 
$C $-Hopf algebras of $\mathfrak{H}_{q}$:
\[ \mathfrak{H}_{q},\,
\mathfrak{H}_{q}/I, \; 
 C \]
 with the sequences of the quotient 
 morphisms %corresponding to 
inclusions 
%\eqref{1113a}
%\begin{equation}
%\mathfrak{H}_{q} \to \com 
%\textit{ and } 
\begin{equation}
 \mathfrak{H}_{q} \to \mathfrak{H}_{q}/I \to C ,
\end{equation}
where 
 $I$ is the bilateral ideal of the Hopf algebra $\gH \sb q$ generated by $v$. 
\end{enumerate}
\item 
The sub-set of intermediate \itqsi division rings of $K/C $: 
$$
C , \, C (t), \, K 
$$
with inclusions 
\begin{equation}\label{1113a}
C \subset C (t) \subset K. 
%\textit{ and }
%\com \subset \com (Q) \subset K.
\end{equation} 
%\end{enumerate}
%\[\com \subset \com(t) \subset K\]
The intermediate \itqsi division rings $C (Q)$ 
is not written as the ring of invariants of a Hopf ideal so that our Galois correspondence is imperfect. 
\end{observation}
The extensions 
$$
K/C \, \text{ and }K/ C (t) 
%\textit{ and } \com (Q) / \com 
$$
are \itqsi Picard-Vessiot extensions with Galois groups
\begin{equation*}
\mathrm{Gal}\, (K/C ) \simeq \mathfrak{H}_{q}, \quad 
\mathrm{Gal}\, (K/C (t)) \simeq 
C [\G\sb{a \,C }] 
\end{equation*}
Here we denote by $C [G]$ the Hopf algebra of the 
coordinate ring of an affine group scheme $G$ over $C $. 
%%%%%%%%%%%%%%%%%%%%%%%%%%%%%%%%%%%%%%%%%%%%%%%%%%%%%%
%%%%%%%%%%%%%%%%%%%%%%%%%%%%%%%%

\section{Further examples and generalizations}\label{141223a}
Looking at the Example above, 
analogous to theory of linear differential equations with constant coefficients, 
Pierre Cartier \cite{car13} discovered that 
one can generalize the results to every 
qsi linear equations over $C $. Let us see other Examples to understand what happens better. 
% \cite{mas14}. 
\begin{example}
Let us consider two $3 \times 3$ matrices 
$$ 
A = 
\begin{bmatrix}
q & 1 & 0 \\
0 & q & 0 \\
0 & 0 & 1
\end{bmatrix}
 , \qquad 
 B =
 \begin{bmatrix}
0 & 0 & 1 \\
0 & 0 & 0 \\
0 & 0 & 0
\end{bmatrix}.
$$
so that $AB =qBA$. 
%%%%%%%%%%%%%
As in the previous Section, we consider 
%%%%%%%%%%%%%%

% and we get 
%$\com [\sigma , \sigma ^{-1}, \, \theta ^{(1)}]$-module of rank $3$ over $\com$,
%%%%%%%%%%%%%%%%%%%%%
%a linear \itqsi equation
\begin{equation}\label{141210a}
\sigma Y = A Y \text{ and } \theta ^{(1)}Y= BY
\end{equation}
over $C $, where $Y$ is a $3\times 3$ unknown matrix. 
\end{example} 
The linear \itqsi equation is equivalent to considering a 
3-dimensional 
vector space $V$ equipped with q \itqsi-module structure defined by the $C $-algebra morphism 
$$
 C [\sigma , \sigma ^{-1}, \theta ^* ] \to M\sb 3 ( C ) = {\rm End}(V), \qquad \sigma ^{\pm 1} \mapsto {}^tA^{\pm 1}, \, \theta^{(1)} \mapsto {}^tB.
 $$ 
\par
The first task is to solve linear \itqsi equation 
\eqref{141210a}
 in \itqsi algebra 
$F(\Z , \, C )[[t]]$.
 To this end, we set 
 \begin{equation}\label{141211a}
 Y:= \sum \sb{ i = 0} ^\infty t^ i A\sb i \in 
 M\sb 3 ( F(\Z , \, C )[[t]]) 
 = M\sb 3 (F(\Z , C ))[[ t ]]
 \end{equation} 
so that $A\sb i \in M\sb 3 (F(\Z , \, C ))$
for every $i \in \N$.
We may also identify 
$$
 M\sb 3 (F(\Z , \, C )) = F( \Z , \, M\sb 3 (C )).
$$
Therefore $A\sb i $ is a function on the set $\Z$ taking values in the set $M\sb 3 (C )$ of matrices. So 
$$
A\sb i = 
\begin{bmatrix} 
\cdots & -2 & -1 & 0 & 1 & 2 & \cdots \\
\cdots & a\sb{-2}^{(i)} & a\sb{-1}^{(i)}& a\sb 0 ^{(i)} & a\sb 1 ^{(i)} & a\sb 2 ^{(i)} & \cdots 
\end{bmatrix}
$$
with 
$a\sb j ^{(i)} \in M\sb 3 (C )$ for every $i\in \N, \, j \in \Z$. 
Substituting \eqref{141211a} into \eqref{141210a} and 
comparing coefficients of $t^i$, we get 
recurrence relations among the $A\sb i$'s 
\begin{equation}\label{141211b}
\sigma (A\sb i) = q^{-i}A A\sb i \qquad \theta^{(1)}( A\sb {i + 1}) =\frac{1}{[i+1]\sb q}BA\sb i
\end{equation}
If we solve recurrence relations \eqref{141211b} with the initial condition $a\sb 0 ^{(0)} = I\sb 3$, since $B^2= 0$, 
$A\sb i =0$ for $i \ge 2$ and 
$$
A\sb 0 = 
\begin{bmatrix}
\cdots & -2 & -1 &0 & 1 & 2 & \cdots \\ 
\cdots & A^{-2} & A^{-1} & I\sb 3 & A & A^2 & \cdots 
\end{bmatrix} 
=\begin{bmatrix}
Q & q^{-1}Z Q & 0 \\
0 & Q & 0 \\
0 & 0 & 1 
\end{bmatrix}, 
\quad 
A\sb 1 = BA\sb 0 = 
\begin{bmatrix}
0 & 0 & 1 \\
0 & 0 & 0 \\
0 & 0 & 0
%0 & 1 & 2 & \cdots \\ 
%B & BA & BA^2 & \cdots 
\end{bmatrix} .
$$
So 
$$
Y = A\sb 0 + t BA\sb 0 =
\begin{bmatrix} 
Q & q^{-1}ZQ & t \\
0 & Q & 0   \\ 
0 & 0 & 1
\end{bmatrix},
$$
where $Z$ is an element of the ring of functions $F(\Z , \, C )$ taking the value $n$ at $n \in \N$ so that
$$
Z = 
\begin{bmatrix}
\cdots & -2 & -1 & 0 & 1 & 2 & \cdots \\
\cdots & -2 & -1 & 0 & 1 & 2 & \cdots 
 \end{bmatrix}
$$
The solution $Y$ is an invertible element in the matrix ring $M\sb 3 (F(\Z , \, C )[[t]])$. 
We introduce a \itqsi $C $-algebra 
$R$ generated by the entries of the matrices $Y$ and $Y^{-1}$ in the \itqsi $C $-algebra 
$F(\Z , C )[[t]]$. To be more concrete 
$$
R := C \langle Q , \, Q^{-1}, Z, \, t \rangle \sb {alg}. 
$$ 
The commutation relations among the generators are 
\begin{equation}\label{141212a}
QQ^{-1}= Q^{-1}Q = 1, \qquad qtQ= Qt, \qquad t(Z+ 1)=Zt \qquad ZQ- QZ.
\end{equation}
and the operators act
as
\begin{align} \label{141212b}
\sigma (t)&= qt, & \qquad \theta ^{(1)}(t)&= 1, \\ 
\label{141212c}
\sigma Q &= qQ, & \qquad \theta ^{(1)} (Q) &= 0, \\ 
\label{141212d} 
\sigma (Q^{-1}) &= q^{-1}Q ^{-1}, &\qquad \qquad \theta ^{(1)} (Q^{-1})
&= 0, \\
\label{141212e}
 \sigma Z &= Z + 1, & \qquad 
\theta ^{(1)}(Z) &= 0.
\end{align}
Then the arguments in the previous Example show that 
the ring $R$ trivializes the \itqsi module defined by the matrices $A$ and $B$, the \itqsi ring $R$ is simple and that 
the ring of constants $C \sb R = C $. The abstract $C $-algebra $R^{\n}$ has a $C $-algebra morphism $R\n \to C $. 
So we may call it the Picard-Vessiot ring of the \itqsi module. 
Of course we can prove % \marginpar{we can prove(?)}
 the uniqueness.

Now we are going to speak of the Galois group of \itqsi equation \eqref{141210a}.
The argument of the previous Section, 
\eqref{141212a} and the actions of the operators \eqref{141212b}, 
\eqref{141212c}, \eqref{141212d} and \eqref{141212e}
allow us to prove the following result. 
\begin{lemma} \label{141212g}
The following conditions for
a $C $-algebra $T$ and four elements $e, \, e', \, f, \, g
\in T$ are equivalent. 
\begin{enumerate}
\renewcommand{\labelenumi}{(\arabic{enumi})}
\item 
There exists a 
a $C $-qsi morphism 
$$
\varphi :R \to R \otimes \sb C T 
$$
such that 
$$
\varphi(Q ) = eQ, \qquad \varphi(Q^{-1}) 
= e' Q^{-1}, \qquad 
\varphi (Z) = Z + f,\qquad
\varphi (t) = t + gQ .
$$
\item The four elements satisfy the following relations. 
\begin{equation}\label{141212f}
ee'= e' e =1, \qquad eg = qfg, \qquad ef =fe, \qquad fg- gf = g. 
\end{equation}
\end{enumerate}
\end{lemma}
Lemma \ref{141212g} tells us the universal co-action. To see the co-algebra structure, 
let $\varphi \sb 1 : R \to R\otimes \sb C T$ be the $C $-qsi morphism determined by four elements $e\sb 1, \, e\sb 1',\, f\sb 1,\, g\sb 1 \in T$ satisfying relations \eqref{141212f}. We take another $C $-qsi algebra morphism 
$\varphi \sb 2 : R \to R\otimes \sb C T$ defied by four elements 
$e\sb 2 ,  e\sb 2', f\sb2 , g \sb2\in T$ satisfying relations \eqref{141212f}. 
We assume that the subsets $\{\, e\sb 1, \, e\sb 1',\, f\sb 1, \, g\sb 1 \}$ and $\{\, e\sb 2, \, e\sb 2',\, f\sb 2, \, g\sb 2 \}$ of $T$ 
are 
mutually commutative. Let us compose $\varphi \sb 1$ and $\varphi\sb 2$. 
\begin{align*}\label{141213a}
Q &\mapsto e\sb 1 Q \qquad &\mapsto \qquad e\sb 2 (e\sb 1 Q) &= (e\sb 1 e\sb 2)Q ,
\\
Q^{-1} &\mapsto e\sb 1 ^{-1} Q^{-1} \qquad &\mapsto \qquad e\sb 2 ^{-1}(e\sb 1^{-1} Q^{-1}) &= (e\sb 1^{-1} e\sb 2^{-1})Q, \\
Z &\mapsto Z + f\sb 1\qquad &\mapsto \qquad (Z + f\sb 1) + f\sb 2 & = Z + (f\sb 1 + f\sb 2 ) ,
\\ 
t &\mapsto t + g\sb 1 Q \qquad &\mapsto \qquad (t + g\sb 1 Q) + g\sb 2 e\sb 1 Q &= 
t + (e\sb 1 g\sb 2 + g\sb 1)Q .
\end{align*}
Let us now set 
$$
A:= C \langle e, \, e' , \, f, \, g \rangle \sb {alg}, 
$$ 
where we {\it \/ assume that the elements} $e, \, e' , \, f, \, g $ {\it \/ satisfy only relations \eqref{141212f}\/} so that we¡¡have an isomorphism 
$$
R\n \simeq A, \qquad Q \mapsto e, 
Q^{-1} \mapsto e' , \, 
\, 
Z \mapsto f, 
\, t \mapsto g
$$ 
as abstract $C $-algebras. 
It follows from the result above of the composition of $\varphi\sb 1 $ and $\varphi \sb 2$ that 
$$
\Delta : A \to A \otimes \sb C A 
$$
with 
$$
\Delta (e) = e\otimes e, \quad \Delta (e') = e' \otimes e', \quad \Delta (f) = f\otimes 1 + 1\otimes f , \quad \Delta (g) = g\otimes 1 + e\otimes g
$$ 
defines a $C $-algebra morphism and together with 
a $C $-algebra morphism 
$$
\epsilon : A \to C , \text{ with } \epsilon (e) =\epsilon (e')=1, \, \epsilon (f)=\epsilon (g) = 0
$$
makes $A$ a Hopf algebra over $C $.
\par
{ \it The Galois group of the rank} $3$ {\/ \it qsi module is the Hopf algebra\/} $A$.
\par 
We add another example.
\begin{example}
We consider matrices 
$$
A=
\begin{bmatrix}
lq & 0 \\
0 & l 
\end{bmatrix}, \qquad 
B= 
\begin{bmatrix}
0 & 1 \\
0 & 0
\end{bmatrix}
\in M\sb 2 (C ) , 
$$
where $l$ is an element of the field $C$. Since $AB =q BA$, 
the $C $-algebra morphism 
$$
 C [\sigma , \sigma ^{-1}, \theta ^* ] \to M\sb 2 (C ) = {\rm End}(V), \qquad \sigma ^{\pm 1} \mapsto {}^tA^{\pm 1}, \, \theta^{(1)} \mapsto {}^tB
$$
defines on a 2-dimensional $C $-vector space $V$ a 
$2$-dimensional \itqsi module structure. We assume that $q$ and $l$ are linearly independent over $\Q$.
\end{example}
 We do not give details here as it is useless to repeat the arguments. 
 \begin{enumerate}
\renewcommand{\labelenumi}{(\arabic{enumi})}
 \item
 The solution matrix in $ M\sb 2 (F(\Z , \, C )[[t]])$ is 
 $$
 \begin{bmatrix}
L Q & t \\
0 & L 
\end{bmatrix}, 
 $$
 where 
 $$
 L = 
 \begin{bmatrix}
\cdots & -1 & 0 & 1 & 2 & \cdots \\
\cdots & l^ {-1} & 1 & l & l^2 & \cdots 
 \end{bmatrix} 
 \in F(\Z , \, C ).
 $$
 \item 
 The Picard-Vessiot ring is 
 $$
 C \langle Q, \, Q^{-1}, \, L, \, L^{-1}, \, t \rangle \sb{alg} 
 $$
 with commutation relations
 $$
 QQ^{-1}= Q^{-1}Q = 1, \quad 
 L L^{-1}= L^{-1}L = 1, \quad 
 QL=LQ, \quad 
 Qt=qtQ, \quad Lt= ltL.
 $$
 Actions of operators: 
 \begin{align*}
 \sigma Q &=& qQ, 
 \quad 
 \sigma (Q^{-1}) &=& q^{-1}Q^{-1}, 
 \quad 
 \sigma L &=& lL , 
 \quad 
 \sigma (L^{-1}) &=& l^{-1}L^{-1}, 
 \quad 
 \sigma (t) &=& qt \\
%%%%%%%%%
 \theta^{(1)} Q &=& 0, 
 \quad 
 \theta^{(1)} (Q^{-1}) &=& 0, 
 \quad 
 \theta^{(1)}(L)&=& 0 , 
 \quad 
 \theta^{(1)}(L^{-1})&=& 0, 
 \quad 
 \theta^{(1)} (t) &= & 1 .
%%%%%%%%%
 \end{align*}
\item 
The Galois group is the Hopf algebra 
$$
\gH\sb q := C \langle e, \, e ^{-1}, \, g, \, h, \, 
h^{-1}
\rangle \sb{alg},
$$
satisfying commutation relations
$$
ee^{-1}= e^{-1}e =1, \quad 
hh^{-1}= h^{-1}h= 1, 
\quad eg =q ge, \quad hg =lgh. 
$$ 
 Co-algebra structure $\Delta :\gH \sb q \to \gH\sb q \otimes \sb C \gH\sb q$:
 $$
 \Delta (e^{\pm 1}) =e^{\pm 1} \otimes e^{\pm 1}, \,
 %\Delta (e^{-1}) =e^{-1} \otimes e^{-1}, \,
 \Delta (h^{\pm 1}) =h^{\pm 1}\otimes h^{\pm 1}, \,
 %\Delta (h^{-1}) =h^{-1}\otimes h^{-1}, \, 
 \Delta (g)=g\otimes 1 + e \otimes g.
 $$
 The co-unit $\epsilon : \gH\sb q \to C $ is given by 
 $$
 \epsilon (e) = \epsilon ( e^{-1}) =\epsilon (h) = \epsilon (h^{-1}) = 1, \qquad \epsilon (g) = 0.
 $$ 
 \end{enumerate} 
The last example is inspired of work of Masatoshi Noumi \cite{nou92} on the quantization of hypergeometric functions. His idea is that $q$-hypergeometric functions should live on the quantized Grassmannians. Namely, he quantizes the framework of Gelfand of defining general hypergeometric functions. 
\begin{example} Let $V$ be the natural $2$-dimensional representation of $U\sb{q}(sl\sb 2)$ over $C $. Hence $V$ is a left
 $U\sb{q}(sl\sb 2)$-module. So we can speak of the Picard-Vessiot extension $R/C $ attached to the left $U\sb q (sl\sb 2)$-module $V$. 
 The argument in the Examples so far studied allows us to 
 guess that $R$ is given by 
$$
R:= C \langle a, b, \, c, \, d \rangle \sb {alg}.
$$  
 with relations
 $$
 ab = qba, \ bd =qdb, \ 
 ac =qca, \ cd = qdc, \ 
 bc = cb, \ 
 ad - da = (q + q^{-1})bc, \ 
 ad -qbc = 1.
 $$
 \par
 Imagine a matrix 
 $$
 \begin{bmatrix}
 a & b \\
 c & d
 \end{bmatrix}
 $$
 and on the space of matrices, the quantum group or Hopf algebra 
 $U\sb q (sl\sb 2)$ operates from right. 
 \end{example}
 Let us recall the definitions. 
 The Hopf algebra $U\sb q (sl\sb 2) = C \langle q^{\pm \frac{H}{2}}, \, X, \, Y\rangle$ is generated by four elements 
 $$
 1,\quad q^{\frac{H}{2}}, \quad q^{-\frac{H}{2}}, \quad X , \quad Y
 $$
 over $C $ satisfying the 
commutation relations 
 $$
 q^{\frac{H}{2}} q^{-\frac{H}{2}}= q^{-\frac{H}{2}}q^{\frac{H}{2}} = 1, \, %[q^H, \, X]=H, \, 
 q^{\frac{H}{2}} X q^{-\frac{H}{2}} = q^2 X ,
 q^{\frac{H}{2}} Y q^{-\frac{H}{2}} = q^{-2} Y , 
 \, [X, \, Y ]= \frac{q^{H} - q^{-H}}{q-q^{-1}}.
 $$
 The co-algebra structure $\Delta : U\sb q (sl\sb 2) \to 
 U\sb q (sl\sb 2) \otimes \sb C U\sb q (sl\sb 2)$
 is given by 
 $$
 \Delta ( q^{\pm \frac{H}{2}} ) = q^{\pm \frac{H}{2}} \otimes q^{\pm \frac{H}{2}} , \quad 
\Delta (X) = X \otimes 1 + q^{\frac{H}{2}} \otimes X, \quad 
 \Delta (Y) = Y\otimes q^{-\frac{H}{2}} + 1\otimes Y.
 $$
 We define the co-unit $\epsilon : U\sb q (sl\sb 2 ) \to C $ by 
 $$
 \epsilon (q^{\pm \frac{H}{2}})= 1 , \quad \epsilon (X) =\epsilon (Y) = 0.
 $$
 See S. Majid \cite{maj95}, 3.2, for example. 
 \par 
 The $C $-algebra $R$ is a $U\sb q (sl\sb 2)$-module algebra by the action 
 of $U\sb q(sl\sb 2)$ on $R$ defined by
 $$
 q^{\pm \frac{H}{2}}.
 \begin{bmatrix}
 a & b \\
 c & d 
 \end{bmatrix} 
 =
 \begin{bmatrix}
 q^{\pm \frac{1}{2}}a & q^{\mp \frac{1}{2}}b \\
 q^{\pm\frac{1}{2}} & q^{\mp \frac{1}{2} }d
 \end{bmatrix}, \quad 
 %%%%%%%%%%%%%%%%%%%%%%%%%
 X.
 \begin{bmatrix}
 a & b \\
 c & d 
 \end{bmatrix} 
 =
 \begin{bmatrix}
 0 & a \\
 0 & c
 \end{bmatrix}, 
 \quad 
 Y.
 \begin{bmatrix}
 a & b \\
 c & d 
 \end{bmatrix} 
 =
 \begin{bmatrix}
 b & 0 \\
 d & 0
 \end{bmatrix}.
 $$
 We had not exactly examined but we believed
 \begin{enumerate}
 \renewcommand{\labelenumi}{(\arabic{enumi})}

 \item 
 The algebra extension $R/C $ is the Picard-Vessiot extension for the $U\sb q(sl\sb 2)$-module $V$. 
 \item 
 The Galois group is 
 the Hopf algebra on the abstract $C $-algebra $R$ 
 with adjunction of 
 the co-algebra structure 
 defined by 
 $$
 \Delta (a)= a \otimes a + b \otimes c , \quad 
 \Delta (b)= a \otimes b + b \otimes d , \quad , 
 \Delta (c)= c \otimes a + d \otimes b , \quad 
 \Delta (d)= c \otimes b + d \otimes d 
 $$ 
 and the co-unit 
 $\epsilon : R \to C $ with
 $$
 \epsilon (a) = \epsilon (d) = 1, \, 
 \epsilon (b) = \epsilon (c) = 0.
 $$
 \end{enumerate}
Indeed, the exact consequence that follows from a general theory. See Remark \ref{rem:enveloping}. 
%%%%%%%%%%%%%%%%%%%%%%%%%%%%%%%%%%%%%%%%%
\part{Hopf-algebraic interpretations}\label{part_III}

This part is devoted to giving Hopf-algebraic interpretations to some of those results 
on linear equations with constant coefficients which have been obtained so far. 
The restriction ``constant coefficients" makes the situation quite simple, since differential 
modules are then quantized simply to modules over a Hopf algebra, say $\mathcal{H}$. 
It is shown in Section \ref{sec:co-rep_Hopf} that given a finite-dimensional
$\mathcal{H}$-module $M$, the left rigid, abelian tensor category $\{ \{ M\} \}$ generated by 
$M$ is isomorphic, by a standard duality, to the category $(comod\text{-}\mathcal{H}^{\circ}_{\pi})$ of 
finite-dimensional co-modules over what we call
the co-representation Hopf algebra $\mathcal{H}^{\circ}_{\pi}$, where $\pi$ indicates
the matrix representation of $\mathcal{H}$ associated with $M$. Therefore, this Hopf algebra 
$\mathcal{H}^{\circ}_{\pi}$ plays the role of the Picard-Vessiot quantum group of $M$. 
In Section \ref{sec:PV_ring_properties}
we define the notion of Picard-Vessiot rings for finite-dimensional $\mathcal{H}$-modules, and
prove that the Hopf algebra $\mathcal{H}^{\circ}_{\pi}$ is a Picard-Vessiot ring for $M$. 
We also show that (part of) the properties of Picard-Vessiot rings are enjoyed by some objects
that generalize $\mathcal{H}^{\circ}_{\pi}$; this suggests that Picard-Vessiot rings for a given $M$
may not be unique, in contrast to classical Picard-Vessiot Theory. 
The non-uniques is confirmed by explicit examples in the last Section 
\ref{sec:non-uniqueness}. Another contrast to the classical theory arises for the
tensor-equivalence $\{ \{ M\} \}\approx (comod\text{-}\mathcal{H}^{\circ}_{\pi})$ mentioned above. 
Since $\mathcal{H}^{\circ}_{\pi}$ is the trivial 
$\mathcal{H}^{\circ}_{\pi}$-torsor one might expect 
that the tensor-equivalence can be realized by the torsor as in the classical theory.
But this is not the case in general, as is explained in Section \ref{sec:tensor-equivalence};
the result modifies correctly Expectation \ref{160808a}.
The remaining Section \ref{sec:review_examples} reviews some examples from Part II, 
refining the results using the arguments in the present part. 

Throughout Part III we work over a fixed, arbitrary field $C$.
Vector spaces, \mbox{(co-)}algebras and Hopf algebras are supposed to be those over $C$. Given a
vector space $V$, we let $V^*$ denote the dual vector space. 
The unadorned $\otimes$ denotes the tensor product over $C$. 

\section{The co-representation Hopf algebra $\mathcal{H}^{\circ}_{\pi}$}\label{sec:co-rep_Hopf}
Let $\mathcal{H}$ be a Hopf algebra. Let $\mathcal{H}^{\circ}$ denote the dual Hopf algebra 
\cite[Section 6.2]{swe69}.
Thus, $\mathcal{H}^{\circ}$ is the filtered union of the finite-dimensional co-algebras $(\mathcal{H}/I)^*$
in $\mathcal{H}^*$, 
\[ \mathcal{H}^{\circ}=\bigcup_{I} \, (\mathcal{H}/I)^*, \]
where $I$ runs over the set of all co-finite ideals of $\mathcal{H}$; note that $(\mathcal{H}/I)^*$ is the
co-algebra dual to the finite-dimensional quotient algebra $\mathcal{H}/I$ of $\mathcal{H}$. 

For a representation-theoretic interpretation of $\mathcal{H}^{\circ}$, note that each $I$ coincides
with the kernel of some matrix representation, $\pi : \mathcal{H}\to M_n(C)$, of $\mathcal{H}$.
Then one recalls that $(\mathcal{H}/I)^*$ coincides with the image
\[ \mathrm{cf}(\pi)=\mathrm{Im}(\pi^*) \]
of the dual $\pi^*:M_n(C)^* \to \mathcal{H}^*$ of $\pi$. 
Let $\{ e_{ij}^*\}$ denote the dual basis of the basis $\{ e_{ij} \}$ of $M_n(C)$ consisting of the matrix
units $e_{ij}$, and set $c_{ij} =\pi^*(e_{ij}^*)$. Then $\mathrm{cf}(\pi)$ is a co-algebra
spanned by $c_{ij}$; it may be called
the \emph{coefficient co-algebra} of $\pi$, whose structure is determined by
\[ \Delta(c_{ij}) =\sum_{k=1}^n c_{ik}\otimes c_{kj},\quad \epsilon(c_{ij})=\delta_{ij}. \]
We present as
\begin{equation}\label{eq:co-rep_matrix}
\mathbf{Y}_{\pi} = \big[ c_{ij} \big]_{1\le i,j\le n}= 
\begin{bmatrix} 
c_{11}& c_{12} & \dots \\ 
c_{21} & c_{22} & \dots \\
\vdots & \vdots & \ddots 
\end{bmatrix}, 
\end{equation}
and call this the \emph{co-representation matrix} of $\pi$. 
One sees that $\mathrm{cf}(\pi)=\mathrm{cf}(\pi')$, if matrix representations $\pi$ and $\pi'$ are
equivalent. We have
\[ \mathcal{H}^{\circ}= \bigcup_{\pi}\, \mathrm{cf}(\pi), \]
where $\pi$ runs over the set of all equivalence classes of matrix representations of $\mathcal{H}$. 
Given a matrix representation $\pi$, 
the image $S(\mathrm{cf}(\pi))$ of $\mathrm{cf}(\pi)$ by the antipode $S$ equals the coefficient co-algebra 
$\mathrm{cf}(\pi^t)$ of the transpose $\pi^t$ of $\pi$. Given another matrix representation $\pi'$,
we have $\mathrm{cf}(\pi) \subset \mathrm{cf}(\pi')$ if and only if $\pi$ is a sub-quotient of
the direct sum $\pi'\oplus \dots \oplus \pi'$ of some copies of $\pi'$. 

Let 
$(\mathcal{H}\text{-}mod)$\ (resp., $(comod\text{-}\mathcal{H}^{\circ})$) 
denote the category of left $\mathcal{H}$-modules (resp., right $\mathcal{H}^{\circ}$-co-modules)
of finite dimension. These two categories are $C$-linear abelian tensor categories. They are both 
left rigid. Indeed, given an object $M$, the dual $M^*$ is naturally a \emph{right} $\mathcal{H}$-module
(resp., a \emph{left} $\mathcal{H}^{\circ}$-co-module). This, with the side switched through the antipode,
gives the left dual of $M$. We have a $C$-linear tensor-isomorphism 
\begin{equation}\label{eq:tensor-isom}
(\mathcal{H}\text{-}mod)\simeq (comod\text{-}\mathcal{H}^{\circ}),
\end{equation}
which is given by the following one-to-one correspondence between the structures defined on a fixed 
vector space $M$ of finite dimension. 
Given an $\mathcal{H}^{\circ}$-co-module structure 
$\rho : M \to M \otimes \mathcal{H}^{\circ}$, we present it as
\begin{equation}\label{eq:sigma_notation}
\rho(m)=\sum_{(m)} m_{(0)} \otimes m_{(1)},
\end{equation}
following \cite[Section 2.0]{swe69}. 
Then the corresponding $\mathcal{H}$-module structure is defined by
\[ hm = \sum_{(m)} m_{(0)} \, \langle m_{(1)}, h\rangle,\quad m \in M,\ h\in \mathcal{H}. \]
Conversely, given an $\mathcal{H}$-module structure on $M$, or a matrix representation 
$\pi : \mathcal{H} \to M_n(C)$ with respect to some basis $\{ m_i\}_{1\le i\le n}$ of $M$, the corresponding 
$\mathcal{H}^{\circ}$-co-module structure is defined by
\[ \rho(m_j) = \sum_{i=1}^n m_i \otimes c_{ij},\quad 1 \le j \le n, \]
where $\big[ c_{ij}\big]_{1\le i,j\le n}$ is the co-representation matrix of $\pi$. 
See \cite[Section 2.1]{swe69}. 

Let $M \in (\mathcal{H}\text{-}mod)$. 
Let $\{ \{ M \} \}$ denote the left rigid, abelian 
tensor sub-category of $(\mathcal{H}\text{-}mod)$ generated by $M$. It is a full sub-category
consisting of those objects which are sub-quotients of 
some direct sum $L_1 \oplus \dots \oplus L_s$, $s \ge 0$, of tensor products 
$L_i=X_{i1}\otimes \dots \otimes X_{i,t_i}$,
$t_i \ge 0$, where $X_{ij}$ is either 
$M$, $M^*$, $M^{**}$ or some further iterated left-dual $M^{**\dots *}$ of $M$. 

Choose a matrix representation $\pi$ associated with $M$, and define $\mathcal{H}^{\circ}_{\pi}$
to be the smallest Hopf sub-algebra of $\mathcal{H}^{\circ}$ that includes $\mathrm{cf}(\pi)$. If
$\big[ c_{ij} \big]_{1\le i,j\le n}$ is the co-representation matrix of $\pi$, then $\mathcal{H}^{\circ}_{\pi}$
is the sub-algebra of $\mathcal{H}^{\circ}$ which is generated by the images
\[ S^k(c_{ij}),\quad k\ge 0,\ 1\le i,j \le n \]
of all $c_{ij}$ by the iterated antipodes. We call $\mathcal{H}_{\pi}^{\circ}$ the \emph{co-representation
Hopf algebra} of $\pi$, or of $M$. 
The category $(comod\text{-}\mathcal{H}_{\pi}^{\circ})$ of finite-dimensional right 
$\mathcal{H}^{\circ}_{\pi}$-co-modules is regarded as a full sub-category of 
$(comod\text{-}\mathcal{H}^{\circ})$. 

We now reach the following standard result, which is easy to see. 

\begin{proposition}\label{prop:restricted_tensor-isom}
The isomorphism \eqref{eq:tensor-isom} restricts to the $C$-linear tensor-isomorphism
\begin{equation}\label{eq:restricted_tensor-isom}
\{ \{ M \} \} \simeq (comod\text{-}\mathcal{H}^{\circ}_{\pi}).
\end{equation}
\end{proposition}

\section{Picard-Vessiot rings in the non-commutative situation}\label{sec:PV_ring_properties}

Let $M \in (\mathcal{H}\text{-}mod)$, and keep the notation as above. 
In view of Proposition \ref{prop:restricted_tensor-isom} we see 
\begin{itemize}
\item[(I)] $\mathcal{H}^{\circ}_{\pi}$ plays the role of the Picard-Vessiot quantum group of $M$.
\end{itemize}
Note that the isomorphism $\eqref{eq:tensor-isom}$ extends to a $C$-linear tensor-isomorphism between
the category of locally finite left $\mathcal{H}$-modules and the category $(Comod\text{-}\mathcal{H}^{\circ})$
of right $\mathcal{H}^{\circ}$-co-modules of possibly infinite dimension.
Since $\mathcal{H}^{\circ}_{\pi}$ is naturally in $(Comod\text{-}\mathcal{H}^{\circ})$, it is a left 
$\mathcal{H}$-module, and is indeed a left $\mathcal{H}$-module algebra.
The $\mathcal{H}$-module structure is explicitly given by
\begin{equation}\label{eq:explicit_module_struc}
h x = \sum_{(x)} x_{(1)}\, \langle x_{(2)}, h\rangle,\quad h\in \mathcal{H},\ x \in \mathcal{H}^{\circ}_{\pi}. 
\end{equation}

We say that $M$ \emph{is trivialized} by a left $\mathcal{H}$-module algebra $\mathcal{R}$ 
(or $\mathcal{R}$ \emph{trivializes} $M$),
if the $\mathcal{H}$-module 
$M \otimes \mathcal{R}$ of tensor product is isomorphic to the direct sum of some copies of $\mathcal{R}$. 
The same term was used in Observation \ref{obs4}, replacing the $M \otimes \mathcal{R}$ here 
with $\mathcal{R}\otimes M$ with conversely ordered tensor factors; it will be shown 
in Remark \ref{rem:bijective_antipode} (3) that the order does not matter under some assumption
that is satisfied by our examples.

\begin{proposition}\label{prop:PV_properties}
We have the following.
\begin{itemize}
\item[$\mathrm{(II)}$] The $\mathcal{H}$-module $M$ is trivialized by $\mathcal{H}^{\circ}_{\pi}$. 
\item[$\mathrm{(III)}$] The $\mathcal{H}$-module algebra $\mathcal{H}^{\circ}_{\pi}$ is simple in the sense
that it includes no non-trivial $\mathcal{H}$-stable ideals. Moreover, it 
includes no non-trivial $\mathcal{H}$-stable right ideals.
\end{itemize}
\end{proposition}
\begin{proof}
(II)\ Let $(M)\otimes \mathcal{H}^{\circ}_{\pi}\, (=(\mathcal{H}^{\circ}_{\pi})^{\dim M})$ denote the 
left $\mathcal{H}$-module for which $\mathcal{H}$ acts on the single factor $\mathcal{H}^{\circ}_{\pi}$. 
Then a desired isomorphism 
\begin{equation}\label{eq:desired_isom}
M \otimes \mathcal{H}^{\circ}_{\pi} 
\overset{\simeq}{\longrightarrow} (M) \otimes \mathcal{H}^{\circ}_{\pi}
\end{equation}
is given, with the notation \eqref{eq:sigma_notation}, by 
$m \otimes x \mapsto \sum_{(m)} m_{(0)}\otimes m_{(1)}x$; the inverse is given by
$m \otimes x \mapsto \sum_{(m)} m_{(0)}\otimes S(m_{(1)})x$.

(III)\ In general, every Hopf sub-algebra $\mathfrak{H}$ of $\mathcal{H}^{\circ}$ is regarded naturally as
a left $\mathcal{H}$-module algebra, and is simple. Indeed, $\mathfrak{H}$ is a $\mathfrak{H}$-Hopf module
\cite[Section 4.1]{swe69}, whence by \cite[Theorem 4.0.5]{swe69}, 
it includes no non-trivial $\mathcal{H}$-stable right ideals. 
\end{proof}

\begin{remark}\label{rem:PV_properties}
Properties (II) and (III) above suggest us to call $\mathcal{H}^\circ_{\pi}$ 
a Picard-Vessiot ring for $M$, as we will actually do; see Definition \ref{def:PV_ring}.
However, the algebra is generated by the entries of all
\[ 
\mathbf{Y}_0=\mathbf{Y}_{\pi}=\big[ c_{ij} \big] , \ 
\mathbf{Y}_1=\big[ S(c_{ij}) \big] , \
\mathbf{Y}_2=\big[ S^2(c_{ij}) \big] , \dots .
 \]
Note that if $k$ is even, $\mathbf{Y}_k$ and $\mathbf{Y}_{k+1}$ are inverse to each other,
while if $k$ is odd, the transposes $\mathbf{Y}_k^t$ and $\mathbf{Y}_{k+1}^t$ are inverse to each other.
Thus the circumstance is naturally different from the commutative situation in which the Picard-Vessiot 
ring is generated by the matrix entries of some $\mathbf{Y}$ and its inverse $\mathbf{Y}^{-1}$. 
As a more crucial difference we will show the non-uniqueness, or namely, Properties (II) and (III) 
can be shared with other left $\mathcal{H}$-module algebras; see Section \ref{sec:non-uniqueness}. 
\end{remark}

Proposition \ref{prop:PV_properties} is generalized by the following.

\begin{proposition}\label{prop:split_simple}
Let
$\mathfrak{H}$ be an arbitrary Hopf sub-algebra of $\mathcal{H}^{\circ}$ including
$\mathcal{H}^{\circ}_{\pi}$. Let $\mathcal{R}$ be a non-zero 
right $\mathfrak{H}$-co-module algebra, which is naturally
regarded as a left $\mathcal{H}$-module algebra. 
\begin{itemize}
\item[$\mathrm{(II)}'$] If $\mathcal{R}$ is cleft, then the $\mathcal{H}$-module $M$ 
is trivialized by $\mathcal{R}$. 
\item[$\mathrm{(III)}'$] If $\mathcal{R}$ is an $\mathfrak{H}$-torsor, then 
it includes no non-trivial $\mathcal{H}$-stable right ideals. 
\end{itemize}
\end{proposition}

To prove the Proposition in a generalized situation, suppose that $\mathfrak{H}$ is an arbitrary Hopf algebra,
and let $\mathcal{R}$ be a non--zero right $\mathfrak{H}$-co-module algebra. 

First, as for (III)$'$, we say that $\mathcal{R}$ is an $\mathfrak{H}$-\emph{torsor} 
(or an $\mathcal{H}$-\emph{Galois extension} \cite[Definition 8.1.1]{mon93} over $C$), 
if the map $\mathcal{R} \otimes \mathcal{R} \to \mathcal{R} \otimes \mathfrak{H}$ given by 
$x\otimes y\mapsto \sum_{(y)}xy_{(0)}\otimes y_{(1)}$, with the same notation as in \eqref{eq:sigma_notation},
is bijective. Note that $\mathfrak{H}$ itself is naturally an $\mathfrak{H}$-torsor, 
whence (III)$'$ generalizes (III) of Proposition \ref{prop:PV_properties}. 
This (III)$'$ follows since we see from \cite[Theorem 8.5.6]{mon93} that 
every $\mathfrak{H}$-torsor 
includes no non-trivial $\mathcal{H}$-stable right ideals.

Next, as for (II)$'$, recall from \cite[Definition 7.2.1]{mon93} that
$\mathcal{R}$ is said to be \emph{cleft}, if there exists an $\mathfrak{H}$-co-module
map $\phi : \mathfrak{H} \to \mathcal{R}$ that is invertible 
with respect to the convolution-product \cite[Definition 1.2.1]{mon93}; the inverse $\phi^{-1}$ 
is then characterized by
\[ 
\sum_{(x)} \phi(x_{(1)})\, \phi^{-1}(x_{(2)})=\epsilon(x)1= 
\sum_{(x)} \phi^{-1}(x_{(1)})\, \phi(x_{(2)}),\quad x \in \mathfrak{H}.
\]
Note that $\mathfrak{H}$ itself is cleft with respect to $\phi=\mathrm{Id}$, the identity, and
$\phi^{-1}=S$, the antipode.
Hence (II)$'$ generalizes (II) of Proposition \ref{prop:PV_properties}. 
We see that (II)$'$ follows from the next Lemma.

\begin{lemma}\label{lemma:split}
Suppose that $\mathcal{R}$ is cleft.
Given a right $\mathfrak{H}$-co-module $N$, an $\mathfrak{H}$-co-module isomorphism
\begin{equation}\label{eq:isom_cleft}
N \otimes \mathcal{R} 
\overset{\simeq}{\longrightarrow} 
(N) \otimes \mathcal{R}
\end{equation}
(analogous to \eqref{eq:desired_isom}) is given by
$n \otimes x \mapsto \sum_{(n)} n_{(0)}\otimes \phi(n_{(1)})x$, where 
$N \otimes \mathcal{R}$ denotes the $\mathfrak{H}$-co-module of tensor product, and
$(N) \otimes \mathcal{R}$ denotes the $\mathfrak{H}$-co-module for which $\mathfrak{H}$ co-acts
on the single factor $\mathcal{R}$. 
\end{lemma}
\begin{proof}
The inverse is given by
$n \otimes x \mapsto \sum_{(n)} n_{(0)}\otimes \phi^{-1}(n_{(1)})x$.
\end{proof}

It is known (see \cite[Theorem 8.2.4]{mon93})
that the cleft $\mathfrak{H}$-torsors coincide with each of the following two classes of
non-zero right $\mathfrak{H}$-co-module algebras: 
\begin{itemize}
\item
Those $\mathfrak{H}$-torsors which are isomorphic 
to $\mathfrak{H}$ as $\mathfrak{H}$-co-module;
\item
The cleft right $\mathfrak{H}$-co-module algebras $\mathcal{R}=(\mathcal{R}, \rho)$
such that
$ \mathcal{R}^{co\, \mathfrak{H}}=C$, where 
$ \mathcal{R}^{co\, \mathfrak{H}} = \{ x \in \mathcal{R} \mid \rho(x)=x \otimes 1 \} $
denotes the sub-algebra of $\mathcal{R}$ consisting of all $\mathfrak{H}$-co-invariants. 
\end{itemize}

We remark that given an $\mathfrak{H}$-torsor $\mathcal{R}$, an $\mathfrak{H}$-co-module
map $\mathfrak{H} \to \mathcal{R}$ is an isomorphism if and only if it is invertible
with respect to the convolution-product; see the proof of \cite[Theorem 8.2.4]{mon93}. 

One sees easily that 
if $\mathfrak{H}$ is finite-dimensional, then
every $\mathfrak{H}$-torsor is necessarily cleft. 

\begin{definition}\label{def:PV_ring}
Let $M \in (\mathcal{H}\text{-}mod)$.
We say that a non-zero left $\mathcal{H}$-module algebra $\mathcal{R}$ is a \emph{Picard-Vessiot ring}
for $M$, if it trivializes $M$, and is simple. 
\end{definition}

By Proposition \ref{prop:PV_properties}, the co-representation Hopf algebra $\mathcal{H}^{\circ}_{\pi}$ is
a Picard-Vessiot ring for $M$. 
Proposition \ref{prop:split_simple} tells us that if $\mathfrak{H}$ is a Hopf sub-algebra
of $\mathcal{H}^{\circ}$ which includes $\mathcal{H}^{\circ}_{\pi}$, every cleft $\mathfrak{H}$-torsor
is a Picard-Vessiot ring for $M$. 

\begin{remark}\label{rem:bijective_antipode}
(1)\
Let $\mathfrak{H}$ be a Hopf algebra in general. Assume that the antipode $S$ of $\mathfrak{H}$ is bijective.
Then the inverse $\overline{S}$ of $S$ is
the unique linear endomorphism of $\mathfrak{H}$ such that
\[ 
\sum_{(x)} x_{(2)}\, \overline{S}(x_{(1)})=\epsilon(x)1= 
\sum_{(x)} \overline{S}(x_{(2)})\, x_{(1)},\quad x \in \mathfrak{H}.
\]
The assumption is satisfied if $\mathfrak{H}$ is finite-dimensional, commutative, co-commutative or
pointed \cite[Section 8.0]{swe69}. 

(2)\
Given an algebra $\mathcal{R}$, let $\mathcal{R}^{op}$ denote the opposite algebra. 
Then $\mathfrak{H}^{op}$, given the original co-algebra structure on $\mathfrak{H}$, is a bi-algebra.
This $\mathfrak{H}^{op}$ is a Hopf algebra if and only if $S$ is bijective. Suppose that this is the case.
Then $\overline{S}$ is the antipode of $\mathfrak{H}^{op}$. 
Through $\mathcal{R}\leftrightarrow \mathcal{R}^{op}$,
the (cleft) $\mathfrak{H}$-torsors and the (cleft) $\mathfrak{H}^{op}$-torsors are in one-to-one
correspondence. Let $\mathcal{R}$ be a cleft $\mathfrak{H}$-torsor, and suppose that 
$\phi : \mathfrak{H}\to \mathcal{R}$ is an $\mathfrak{H}$-co-module map which is invertible
with respect to the convolution-product. 
Regarded as an $\mathfrak{H}^{op}$-co-module map $\mathfrak{H}^{op} \to \mathcal{R}^{op}$,
$\phi$ remains invertible (see the remark before Definition \ref{def:PV_ring}), 
whence there uniquely exists a linear map $\overline{\phi} : \mathfrak{H}\to \mathcal{R}$ such that
\[ 
\sum_{(x)} \phi(x_{(2)})\, \overline{\phi}(x_{(1)})=\epsilon(x)1= 
\sum_{(x)} \overline{\phi}(x_{(2)})\, \phi(x_{(1)}),\quad x \in \mathfrak{H}.
\]
Note that if $\mathcal{R}=\mathfrak{H}$ and $\phi=\mathrm{Id}$, then $\overline{\phi}$ coincides with
the $\overline{S}$ above.

(3)\ 
Suppose that we are in the situation of Proposition \ref{prop:split_simple} and Lemma \ref{lemma:split}. 
Assume that the antipode of $\mathfrak{H}$ is bijective, and let $\mathcal{R}$ be a cleft $\mathfrak{H}$-torsor
with $\phi$, $\overline{\phi}$ as above. 
Given a right $\mathfrak{H}$-co-module $N$, we have an $\mathfrak{H}$-co-module isomorphism
analogous to \eqref{eq:isom_cleft},
\begin{equation*}
\mathcal{R} \otimes N
\overset{\simeq}{\longrightarrow} 
(N) \otimes \mathcal{R},\quad x \otimes n \mapsto \sum_{(n)} n_{(0)}\otimes x\, \phi(n_{(1)}), 
\end{equation*}
whose inverse is given by $n \otimes x \mapsto \sum_{(n)} x\, \overline{\phi}(n_{(1)})\otimes n_{(0)}$. 
In particular, in the situation of Proposition \ref{prop:PV_properties}, 
if the antipode of $\mathcal{H}^{\circ}_{\pi}$ is bijective, then we have an $\mathcal{H}$-module
isomorphism analogous to \eqref{eq:desired_isom}, 
\[
\mathcal{H}^{\circ}_{\pi} \otimes M
\overset{\simeq}{\longrightarrow} (M) \otimes \mathcal{H}^{\circ}_{\pi},\quad
x \otimes m \mapsto \sum_{(m)} \mapsto m_{(0)} \otimes x m_{(1)}. 
\]
In the examples which we will review in Section \ref{sec:review_examples}, the
Hopf algebras $\mathcal{H}^{\circ}_{\pi}$ are pointed, so that we have the last 
isomorphism for those. 
\end{remark}

\section{Analogous tensor-equivalences}\label{sec:tensor-equivalence}

Let $M \in (\mathcal{H}\text{-}mod)$, and keep the notation as above.
Let us return to Property (I). Recall that in the classical theory, 
Picard-Vessiot rings give the tensor-equivalences as found in 
\cite[Theorem 8.11]{amaetal09}, for example; for this, essential is that those rings are torsors. 
From (I) one may expect an analogous tensor-equivalence 
$\{ \{ M \} \} \approx (comod\text{-}\mathcal{H}^{\circ}_{\pi})$, since 
$\mathcal{H}^{\circ}_{\pi}$ is indeed a (trivial) $\mathcal{H}^{\circ}_{\pi}$-torsor. 
We will see that under some mild assumption, there exists such a tensor-equivalence, for which, however, 
$\mathcal{H}^{\circ}_{\pi}$ does not act as a torsor any more.
The result gives the ``corrected version of Expectation \ref{160808a}" as referred to 
just after the Expectation.

Let us write $\mathcal{R}$ for the trivial $\mathcal{H}^{\circ}_{\pi}$-torsor $\mathcal{H}^{\circ}_{\pi}$,
to make its role clearer. Recall that $\mathcal{R}$ has the natural left $\mathcal{H}$-module structure
as given by \eqref{eq:explicit_module_struc}. 
Define a right $\mathcal{H}^{\circ}_{\pi}$-co-module structure 
$\varrho : \mathcal{R} \to \mathcal{R}\otimes \mathcal{H}^{\circ}_{\pi}$ on $\mathcal{R}$ by
\[
\varrho(x) = \sum_{(x)} x_{(2)}\otimes S(x_{(1)}),\quad x \in \mathcal{R}, 
\]
where $S$ denotes the antipode of $\mathcal{H}^{\circ}_{\pi}$. Note that 
$(\mathcal{R}, \varrho)$ is not an $\mathcal{H}^{\circ}_{\pi}$-torsor nor even a
right $\mathcal{H}^{\circ}_{\pi}$-co-module algebra in general. 
Given $N \in (\mathcal{H}\text{-}mod)$, endow the left $\mathcal{H}$-module $\mathcal{R}\otimes N$
with the right $\mathcal{H}^{\circ}_{\pi}$-co-module structure induced by $\varrho$. Since the
$\mathcal{H}$-module structure on $\mathcal{R}$ commutes with $\varrho$, the $\mathcal{H}$-invariants
$(\mathcal{R} \otimes N)^{\mathcal{H}}$ in $\mathcal{R} \otimes N$, consisting of those
elements on which $\mathcal{H}$-acts trivially through the co-unit, form 
an $\mathcal{H}^{\circ}_{\pi}$-sub-co-module.

\begin{lemma}\label{lem:analogous_equivalence}
Assume that the antipode of $\mathcal{H}^{\circ}_{\pi}$ is bijective. Then
$N \mapsto (\mathcal{R} \otimes N)^{\mathcal{H}}$ gives a tensor-equivalence
$\{ \{ M \} \} \approx (comod\text{-}\mathcal{H}^{\circ}_{\pi})$. 
\end{lemma}
\begin{proof}
Let $\overline{S}$ denote the inverse of the antipode $S$ of $\mathcal{H}^{\circ}_{\pi}$. 
Let $N \in \{\{ M\} \}$, and regard it as an object 
also in $(comod\text{-}\mathcal{H}^{\circ}_{\pi})$ through \eqref{eq:tensor-isom}.
Using the notation \eqref{eq:sigma_notation}, we see that
\[ 
\alpha_N : N \to (\mathcal{R} \otimes N)^{\mathcal{H}},\quad
\alpha_N(n) = \sum_{(n)}\overline{S}(n_{(1)})\otimes n_{(0)}.
\]
is an isomorphism in $(comod\text{-}\mathcal{H}^{\circ}_{\pi})$.
Indeed, the last sum is $\mathcal{H}$-invariant since it is seen to be 
$\mathcal{H}^{\circ}_{\pi}$-co-invariant. The inverse $\alpha_N^{-1}$ 
associates $\sum_i \epsilon(x_i)\, n_i$ to an element 
$\sum_i x_i \otimes n_i \in (\mathcal{R} \otimes N)^{\mathcal{H}}$.
The isomorphism $\alpha_N$ is natural in $N$. It translates the tensor-isomorphism 
\eqref{eq:tensor-isom} into the tensor-equivalence 
$F: N \mapsto (\mathcal{R} \otimes N)^{\mathcal{H}}$, whose tensor structure
is given by $\alpha_C$ and
\[ 
F(N) \otimes F(N')
\to F(N\otimes N'),\quad 
\big(\sum_i x_i \otimes n_i\big)\otimes \big(\sum_j y_j \otimes n'_j\big)\mapsto
\sum_{i,j}y_jx_i \otimes (n_i\otimes n'_j)
\]
for $N, N'\in \{ \{ M \} \}$. Note that the last map, composed with $\alpha_N \otimes \alpha_N'$,
coincides with $\alpha_{N\otimes N'}$, and is, therefore, an isomorphism. 
\end{proof}

Assume that $\mathcal{H}$ is co-commutative, and so that
$\mathcal{H}^{\circ}_{\pi}=\mathcal{R}$ is commutative. Their antipodes are then involutions. 
Through the antipode of $\mathcal{H}^{\circ}_{\pi}$ transfer the structures of $\mathcal{R}$ 
onto its copy, which we denote by $\mathcal{R}'$.
Then $\mathcal{R}'$ has the original, 
natural right $\mathcal{H}^{\circ}_{\pi}$-co-module structure, under which it is the trivial 
$\mathcal{H}^{\circ}_{\pi}$-torsor. The left $\mathcal{H}$-module structure on $\mathcal{R}'$ is given by 
\[ 
h x = \sum_{(x)} \langle x_{(1)}, S(h) \rangle \, x_{(2)},\quad h \in \mathcal{H},\ x \in \mathcal{R}',
\]
where $S$ denotes the antipode of $\mathcal{H}$. 
The co-commutativity assumption ensures that $\mathcal{R}'$ is a left $\mathcal{H}$-module algebra, and 
it remains 
a Picard-Vessiot ring for $M$. The tensor-equivalence above now reads
$N \mapsto (\mathcal{R}'\otimes N)^{\mathcal{H}}$, which is, given by the torsor $\mathcal{R}'$, 
of the same form as the classical ones as found in 
\cite[Theorem 8.11]{amaetal09}.

\section{Reviewing some examples}\label{sec:review_examples}

To review two examples from Part II, let $\mathcal{H}$ be the Hopf algebra $\mathcal{H}_q$ defined by Definition 3.9.
Assume that the element $0 \ne q \in C$ is not a root of $1$, unless otherwise stated. 
We write $t$ for the element $t_1$ given in the Definition. Then
\[ \mathcal{H} = C[s,s^{-1} ]\otimes C[t] . \]
This means that $\mathcal{H}$ includes the Laurent polynomial algebra $C[s,s^{-1}]$ and
the polynomial algebra $C[t]$ so that the product map 
$C[s,s^{-1}]\otimes C[t] \to \mathcal{H}$ is bijective. 
The algebra structure on $\mathcal{H}$ is determined by the relation $ts=qst$. 
Obviously, $\mathcal{H}$ has
\[ v_{m,n} = s^m \frac{t^n}{[n]_q!},\quad m \in \mathbb{Z},\ n\in \mathbb{N} \]
as a basis. The co-algebra structure is given by
\begin{equation}\label{eq:co-alg_struc1}
\Delta(v_{m,n}) = \sum_{i+j=n} v_{m+j, i}\otimes v_{m,j},\quad \epsilon(v_{m,n}) = \delta_{n,0}.
\end{equation}
In particular, $s$ and $s^{-1}$ are grouplikes, and $t$ is $(s,1)$-primitive, or 
$\Delta(t)=s \otimes t+ t\otimes1$. Generated by those elements, 
$\mathcal{H}$ is a pointed Hopf algebra, whose antipode is necessarily bijective.
We remark that $s$ and $t$ are denoted in Section \ref{14.6.25b}
by $\sigma$ and $\theta^{(1)}$, respectively. 

Define elements $e,f,g \in \mathcal{H}^*$ by
\[ 
\langle e, v_{m,n} \rangle = q^m\, \delta_{n,0},\quad 
\langle f, v_{m,n} \rangle =m\, \delta_{n,0},\quad 
\langle g, v_{m,n} \rangle =\delta_{n,1}. 
\]
Then we see that
\begin{equation}\label{eq:Hopf_frakH} 
\mathfrak{H} = C[e,e^{-1},f]\otimes C[g] 
\end{equation}
in included in $\mathcal{H}^{\circ}$; this means, just as before, that the two commutative algebras 
are included in $\mathcal{H}^{\circ}$ so that the product map 
$C[e,e^{-1},f]\otimes C[g]\to \mathcal{H}^{\circ}$ is injective.
By using \eqref{eq:co-alg_struc1} we obtain
\[ eg = q\, ge, \quad [f,g] = g. \]
These relations show that $\mathfrak{H}$ is a sub-algebra of $\mathcal{H}^{\circ}$, 
determining its algebra structure. Since $e$\ (resp., $f$) is in fact a $C$-valued algebra map 
(resp., derivation) defined on the quotient Hopf algebra $\mathcal{H}/(t) = C[s,s^{-1}]$ of $\mathcal{H}$, 
it follows that 
\begin{equation}\label{eq:co-alg_struc2}
\Delta(e) = e \otimes e,\quad \epsilon(e) = 1,\qquad
\Delta(f) = 1 \otimes f + f \otimes 1,\quad \epsilon(f) = 0.
\end{equation}
(As an additional remark, $g^n$, $n\ge 0$, can be grasped as follows: 
the quotient co-algebra $\mathcal{Q}$ of $\mathcal{H}$ divided by the right ideal generated by $s-1$
is spanned by the $\infty$-divided power sequence $1, t, t^2/[2]_q!, \dots$. The $g^n$ are 
the elements of $\mathcal{Q}^*$ given by
$\langle g^n, t^{\ell}/[\ell]_q!\rangle =\delta_{n,\ell}$.) 
Since one computes
\[ \langle g, v_{m,n}v_{k,\ell} \rangle = \begin{cases} 1 &\text{if}~~n=0,\ \ell=1, \\
q^k &\text{if}~~n=1,\ \ell =0, \\
0 & \text{otherwise},
\end{cases}\]
it follows that 
\begin{equation}\label{eq:co-alg_struc3}
\Delta(g) = 1 \otimes g + g \otimes e,\quad \epsilon(g) = 0.
\end{equation}
By \eqref{eq:co-alg_struc2} and \eqref{eq:co-alg_struc3}, $\mathfrak{H}$ is a sub-bi-algebra of 
$\mathcal{H}^{\circ}$. Indeed, it is a Hopf sub-algebra, having the antipode determined by
\[ S(e) = e^{-1},\quad S(f)=-f, \quad S(g) = -ge^{-1}. \]
From this explicit form of the antipode we will see that for the following two examples of 
matrix representations $\pi$ of $\mathcal{H}$, the co-representation Hopf algebra $\mathcal{H}^{\circ}_{\pi}$
is generated by the entries of the co-representation matrix $\mathbf{Y}_{\pi}$ and $e^{-1}=S(e)$. 
Note that the Hopf algebra $\mathfrak{H}$ and those $\mathcal{H}^{\circ}_{\pi}$ 
which will be obtained as its Hopf sub-algebras are all pointed.

\begin{example}\label{ex:deg_2_matrix_rep}
Recall from Section \ref{14.6.25b}
the matrix representation
$\pi$ of $\mathcal{H}$ determined by
\begin{equation}\label{eq:matrix_rep}
\pi(s) = \begin{bmatrix} q & 0\\ 0 & 1\end{bmatrix},\quad 
\pi(t) = \begin{bmatrix} 0 & 0\\ 1 & 0\end{bmatrix}. 
\end{equation}
These matrices are the transposes of those in Section \ref{14.6.25b}, since the basis of the corresponding
$\mathcal{H}$-module is supposed here to be in a column vector whereas it is supposed in Section \ref{14.6.25b}
to be in a row vector.
Since one computes 
\[ \pi(v_{m,n}) = \begin{cases} \begin{bmatrix} q^m & 0\\ 0 & 1\end{bmatrix} & n=0, \\
\ \begin{bmatrix} 0 & 0 \\ 1 & 0 \end{bmatrix} & n=1, \\ 
\quad O & n \ge 2, \end{cases}
\]
it follows that
\[ \mathbf{Y}_{\pi} = \begin{bmatrix} e & 0 \\ g & 1\end{bmatrix}. \]
We conclude that $\mathcal{H}^{\circ}_{\pi}=C\langle e^{\pm 1}, g \rangle$, the Hopf sub-algebra
of $\mathfrak{H}$ generated by $e^{\pm 1}$ and $g$.

Direct computations show that the natural left $\mathcal{H}$-module structure on this 
$C\langle e^{\pm 1}, g \rangle$ is given by
\[
s(e^{\pm1},\ g)=(q^{\pm1}e^{\pm1},\ q\, g), \quad t(e^{\pm1},\ g)=(0,\ 1).
\]
One then sees that $e^{\pm1}\mapsto Q^{\pm1}$, $g \mapsto t$ give an isomorphism from 
$C\langle e^{\pm 1}, g \rangle$ to the $\mathcal{H}$-module algebra 
$R=C\langle t, Q, Q^{-1}\rangle_{alg}$ given in Observation \ref{obs2}.
As was shown by Proposition \ref{160609},
the Hopf algebra $\mathfrak{H}_q$ constructed in Lemma 4.12
co-acts on this last $R$ from the right, so that $R$ is an $\mathfrak{H}_q$-torsor, 
and the $\mathfrak{H}_q$-co-action commutes with the $\mathcal{H}$-action.
Note that the Hopf algebra $\mathcal{H}^{\circ}_{\pi}=C\langle e^{\pm 1}, g \rangle$ here obtained differs from $\mathfrak{H}_q$
only in that their co-multiplications are opposite to each other. The difference arises, as was seen 
in the preceding section, since in the quantized situation, 
the role of $\mathcal{H}^{\circ}_{\pi}$ as a Picard-Vessiot ring is not necessarily compatible with 
its role as a torsor which was expected to give the tensor equivalence 
\eqref{eq:restricted_tensor-isom}.
\end{example}

\begin{example}\label{ex:deg_3_matrix_rep}
Recall from Example 12.1
the matrix representation
$\pi$ of $\mathcal{H}$ determined by
\[ 
\pi(s) = \begin{bmatrix} q & 0 & 0\\ 1 & q & 0\\ 0 &0 &1\end{bmatrix},\quad 
\pi(t) = \begin{bmatrix} 0 & 0 &0 \\ 0 & 0 & 0\\ 1 & 0 & 0\end{bmatrix}. 
\]
Since one computes 
\[ 
\pi(v_{m,n}) = \begin{cases} \begin{bmatrix} q^m & 0 & 0\\ mq^{m-1}& q^m & 0\\ 0& 0 & 1\end{bmatrix} & n=0, \\
\ \begin{bmatrix} 0 & 0& 0 \\ 0 & 0 & 0\\ 1 & 0 & 0\end{bmatrix} & n=1, \\ 
\quad O & n \ge 2, \end{cases}
\]
it follows that
\[ 
\mathbf{Y}_{\pi} = \begin{bmatrix} e & 0 & 0 \\ \frac{1}{q}ef & e & 0 \\ g & 0 & 1\end{bmatrix}.
\]
We conclude that $\mathcal{H}^{\circ}_{\pi}=\mathfrak{H}$, the Hopf algebra 
given above by \eqref{eq:Hopf_frakH}.

One sees that this $\mathfrak{H}$ is isomorphic, as an $\mathcal{H}$-module algebra, to the 
$R = C\langle Q,Q^{-1},Z,t\rangle_{alg}$ obtained in Example 12.1
via $e^{\pm1} \mapsto Q^{\pm1}$, $f \mapsto Z$, $g \mapsto t$. 
As Hopf algebras, $\mathfrak{H}$ and the $A = C\langle e, e',f,g\rangle_{alg}$ resulting from 
Lemma \ref{141212g} 
coincide except in that their co-multiplications are opposite to
each other, again; see the last paragraph of the preceding Example. 
\end{example}

\begin{remark}\label{rem:enveloping}
Suppose that characteristic $\mathrm{char}\, C$ of $C$ is not $2$, and that $0 \ne q \in C$ 
is not a root of $1$. Let $\mathcal{H}=U_q(sl_2)$ be the quantized enveloping algebra of $sl_2$ as 
given in Example 12.4. This Example is 
essentially the attempt to determine the Hopf algebra $\mathcal{H}^{\circ}_{\pi}$ for
the fundamental representation $\pi$ of $\mathcal{H}$. It is known that $\mathcal{H}^{\circ}_{\pi}$
is the coordinate Hopf algebra $O_q(SL_2)$ of the quantized $SL_2$, in which the elements
$a, b, c, d$ given in the Example are the standard generators; the antipode of this $O_q(SL_2)$
is bijective. See \cite{tak92}, for example. 
\end{remark}

\section{Non-uniqueness of Picard-Vessiot rings}\label{sec:non-uniqueness}

We will see the non-uniqueness mentioned in Remark \ref{rem:PV_properties}.
The result does not contradict the Characterization Theorem, Theorem \ref{th1}; 
the difference
arises mainly from the additional assumption of the Theorem that
the Picard-Vessiot ring has $C$-rational points.

Let $\mathcal{H}=C[s,s^{-1} ]\otimes C[t]$ be the Hopf algebra given at the beginning of the preceding section.
Note that the Hopf algebra $\mathcal{H}^{\circ}_{\pi}= C\langle e^{\pm 1}, g \rangle$
obtained in Example \ref{ex:deg_2_matrix_rep} is isomorphic to $\mathcal{H}$;
in fact, $s \mapsto e^{-1}$ and $t\mapsto e^{-1}g$ define an isomorphism 
$\mathcal{H}\overset{\simeq}{\longrightarrow} \mathcal{H}^{\circ}_{\pi}$. 
Since this $\mathcal{H}$ is 
pointed, every $\mathcal{H}$-torsor is necessarily cleft. It is known
and easy to see that every $\mathcal{H}$-torsor is isomorphic to the trivial $\mathcal{H}$. 

Assume that $q$ is a primitive $N$-th root of $1$, where $N > 1$. Then $s^N-1$ and $t^N$ generate
a Hopf ideal of $\mathcal{H}$. The resulting quotient Hopf algebra $\mathcal{T}$, 
which is finite-dimensional and pointed, is called \emph{Taft's Hopf algebra}. 
Denote the natural images of $s$, $t$ in $\mathcal{T}$ by the same symbols. 

\begin{example}\label{ex:cleft_over_Taft}
Given an element $\lambda \in C$, 
let $\mathcal{R}_{\lambda}$ denote the algebra generated by two elements $s'$, $t'$, and defined by the
relations
\[
t's'=q\, s't',\quad s'^N=1,\quad t'^N=\lambda. 
\] 
One sees that 
\[
\rho(s')=s' \otimes s, \quad \rho(t')= s'\otimes t+t' \otimes 1
\]
defines an algebra map $\rho : \mathcal{R}_{\lambda} \to \mathcal{R}_{\lambda}\otimes \mathcal{T}$, by
which $\mathcal{R}_{\lambda}$ is a right $\mathcal{T}$-co-module algebra; it coincides with 
$\mathcal{T}$ if $\lambda=0$. It is essentially proved in \cite{mas94} that 
$\mathcal{R}_{\lambda}$ is a (necessarily, cleft) $\mathcal{T}$-torsor, and
$\mathcal{R}_{\lambda}\simeq \mathcal{R}_{\lambda'}$ if and only if $\lambda=\lambda'$. Moreover,
if $C^{\times}= (C^{\times})^N$, every $\mathcal{T}$-torsor is isomorphic to 
$\mathcal{R}_{\lambda}$ for some $\lambda$. 
\end{example}

By the same equations as in \eqref{eq:matrix_rep}
one can define a matrix representation $\pi$ of $\mathcal{T}$, for which we see that
$\mathcal{T}^{\circ}_{\pi}\simeq \mathcal{T}$. 
By Proposition \ref{prop:split_simple}, the mutually non-isomorphic $\mathcal{T}$-torsors
$\mathcal{R}_{\lambda}$, $\lambda \in C$, show that
Picard-Vessiot rings for $\pi$ are not unique. 
%%%%%%%%%%%%%%%%%%%%%%%%%%%%%%%
%%%%%%%%%%%%%%%%%%
 \bibliographystyle{plain.bst}
\bibliography{umemura2b}
\end{document}